%% file: matrix3.tex
\documentclass[11pt,nocut]{report}
\usepackage{xcolor}
\usepackage[pdfpagemode=UseNone,bookmarksopen=false,colorlinks=true,urlcolor=blue,citecolor=magenta,citebordercolor=blue,linkcolor=blue]{hyperref}

\usepackage{./packages}
\usepackage{./notations}
\usepackage[bottom]{footmisc}

\title{Phase transitions in spiked matrix estimation: information-theoretic analysis}
\author{Léo Miolane}
\date{\today}

\begin{document}
\maketitle

\section*{Introduction}
\addcontentsline{toc}{chapter}{Introduction}

Estimating a low-rank object (matrix or tensor) from a noisy observation is a fundamental problem in statistical inference with applications in machine learning, signal processing or information theory. 
These notes mainly focus on so-called ``spiked'' models where we observe a signal spike perturbed with some additive noise.
We should consider here two popular models.

The first one is often denoted as the \textit{spiked Wigner} model. One observes
\begin{equation} \label{eq:spiked_wigner}
	\bbf{Y} = \sqrt{\frac{\lambda}{n}} \, \bbf{X} \bbf{X}^{\sT} +\bbf{Z}
\end{equation}
where $\bbf{X} = (X_1, \dots, X_n) \iid P_0$ is the signal vector and  $\bbf{Z}$ is symmetric matrix that account for noise with standard Gaussian entries: $ (Z_{i,j})_{i \leq j} \iid \mathcal{N}(0,1)$.  $\lambda \geq 0$ is a signal-to-noise ratio.

The second model we consider is the non-symmetric version of~\eqref{eq:spiked_wigner}, sometimes called \textit{spiked Wishart}\protect\footnotemark or spiked covariance model:
\begin{equation} \label{eq:spiked_wishart}
	\bbf{Y} = \sqrt{\frac{\lambda}{n}} \, \bbf{U} \bbf{V}^{\sT} +\bbf{Z}
\end{equation}
where $\bbf{U} = (U_1, \dots, U_n) \iid P_U$, $\bbf{V}=(V_1, \dots, V_m) \iid P_V$ are independent. $\bbf{Z}$ is a noise matrix with standard normal entries: $ Z_{i,j} \iid \mathcal{N}(0,1)$.  $\lambda > 0$ captures again the strength of the signal. We are here interested in the regime where $n,m \to +\infty$, while $m/n \to \alpha>0$.
In both models (\ref{eq:spiked_wigner}-\ref{eq:spiked_wishart}) the goal of the statistician is to estimate the low-rank signals ($\bbf{X}\bbf{X}^{\sT}$ or $\bbf{U}\bbf{V}^{\sT}$) from the observation of $\bbf{Y}$. This task is often called Principal Component Analysis (PCA) in the literature.
\footnotetext{This terminology usually refers to the case where $\bbf{V}$ is a standard Gaussian vector. We consider here a slightly more general case by allowing the entries of $\bbf{V}$ to be taken i.i.d.\ from any probability distribution.}

These spiked models have received a lot of attention since their introduction by~\cite{johnstone2001distribution}. 
From a statistical point of view, there are two main problems linked to the spiked models (\ref{eq:spiked_wigner}-\ref{eq:spiked_wishart}).
\begin{itemize}
	\item The recovery problem: how can we recover the planted signal $\bbf{X}$ / $\bbf{U},\bbf{V}$? Is it possible? Can we do it efficiently?
	\item The detection problem: is it possible to distinguish between the pure noise case ($\lambda=0$) and the case where a spike is present ($\lambda > 0$)? Is there any efficient test to do this?
\end{itemize}

We will focus here on the recovery problem. We let the reader refer to
\cite{berthet2013optimal,onatski2013asymptotic,dobriban2017sharp,banks2016information,perry2016optimality,alaoui2017finite,alaoui2018detection} and the references therein for a detailed analysis of the detection problem.
\\

The spiked models (\ref{eq:spiked_wigner}-\ref{eq:spiked_wishart}) has been extensively studied in random matrix theory. The seminal work of~\cite{baik2005phase} (for the complex spiked Wishart model, and~\cite{baik2006eigenvalues} for the real spiked Wishart)
established the existence of a phase transition:
there exists a critical value of the signal-to-noise ratio $\lambda$ above which the largest singular value of $\bbf{Y}/\sqrt{n}$ escapes from the Marchenko-Pastur bulk. The same phenomenon holds for the spiked Wigner model, see~\cite{peche2006largest,feral2007largest,capitaine2009largest,benaych2011eigenvalues}.
It turns out that for both models the eigenvector (respectively singular vector) corresponding to the largest eigenvalue (respectively singular value) also undergo a phase transition at the same threshold, see~\cite{hoyle2007statistical,paul2007asymptotics,nadler2008finite,benaych2011eigenvalues,benaych2012singular}.

For the spiked Wigner model~\eqref{eq:spiked_wigner}, the main result of interest to us is the following. For any probability distribution $P_0$ such that $\E_{P_0}[X^2]=1$, we have
\begin{itemize}
	\item if $\lambda\leq 1$, the top eigenvalue of $\bbf{Y}/\sqrt{n}$ converges a.s.\ to $2$ as $n\to \infty$, and the top eigenvector $\bbf{\widehat{x}}$ (with norm $\|\bbf{\widehat{x}}\|^2=n$) has asymptotically trivial correlation with $\bbf{X}$: $\frac{1}{n}\bbf{\widehat{x}}^{\sT}\bbf{X}\to 0$ a.s.
	\item if $\lambda>1$, the top eigenvalue of $\bbf{Y}/\sqrt{n}$ converges a.s.\ to $\sqrt{\lambda}+1/\sqrt{\lambda}>2$ and the top eigenvector $\bbf{\widehat{x}}$ (with norm $\|\bbf{\widehat{x}}\|^2=n$) has asymptotically nontrivial correlation with $\bbf{X}$: $\left(\frac{1}{n}\bbf{\widehat{x}}^{\sT}\bbf{X}\right)^2\to 1-1/\lambda$ a.s.
\end{itemize}
An analog statement for the spiked Wishart model is proved in~\cite{benaych2012singular}.
These results give us a precise understanding of the performance of the top eigenvectors (or top singular vectors) for recovering the low-rank signals.

However these naive spectral estimators do not take into account any prior information on the signal.
Thus many algorithms have been proposed to exploit additional properties of the signal, such as sparsity~\cite{johnstone2004sparse,aspremont2005direct,zou2006sparse,amini2008high,deshpande2014sparse} or positivity~\cite{montanari2016non}.
\\

Another line of works study Approximate Message Passing (AMP) algorithms for the spiked models above, see~\cite{rangan2012iterative,deshpande2014information,DBLP:conf/isit/LesieurKZ15,montanari2017estimation}.
Motivated by deep insights from statistical physics, these algorithms are believed (for the models (\ref{eq:spiked_wigner}-\ref{eq:spiked_wishart}), when $\lambda$ and the priors $P_0,P_U,P_V$ are known by the statistician) to be optimal among all polynomial-time algorithms.
A great property of these algorithms is that their performance can be precisely tracked in the high-dimensional limit by a simple recursion called ``state evolution'', see~\cite{bayati2011amp,javanmard2013state}. For a detailed analysis of message-passing algorithms for the models (\ref{eq:spiked_wigner}-\ref{eq:spiked_wishart}), see~\cite{lesieur2017constrained}.
\\

In the following we will not consider any particular estimator but rather try to compute the best performance achievable by any estimator. We will suppose to be in the so-called ``Bayes-optimal'' setting, where the statistician knows the prior $P_0$ (or $P_U$, $P_V$) and the signal-to-noise ratio $\lambda$. In that situation, we will study the posterior distribution of the signal given the observations. 
As we should see in the sequel, both estimation problems (\ref{eq:spiked_wigner}-\ref{eq:spiked_wishart}) can be seen as mean-field spin glass models similar to the Sherrington-Kirkpatrick model, studied in the ground-breaking book of Mézard, Parisi and Virasoro~\cite{mezard1987spin}.
Therefore, the methods that we will use here come from the mathematical study of spin glasses, namely from the works of Talagrand~\cite{talagrand2010meanfield1,talagrand2010meanfield2}, Guerra~\cite{guerra2003broken} and Panchenko~\cite{panchenko2013SK}. 
\\

In order to further motivate the study of the models (\ref{eq:spiked_wigner}-\ref{eq:spiked_wishart}) let us mention some interesting special cases, depending on the choice of the priors $P_0$ / $P_U, P_V$.
\begin{itemize}
	\item \textbf{Sparse PCA}. 
		Consider the spiked Wishart model with $P_U = \Ber(\epsilon)$ and $P_V = \cN(0,1)$. In that case, one see that conditionally on $\bbf{U}$ the columns of $\bbf{Y}$ are i.i.d.\ sampled from $\cN\big(\bbf{0}, \bbf{I}_n + \lambda / n \bbf{U}\bbf{U}^{\sT} \big)$, which is a sparse spiked covariance model. The spiked Wigner model with $P_0 = \Ber(\epsilon)$ has also been used to study sparse PCA.
	\item \textbf{Submatrix localization}. Take $P_0 = \Ber(p)$ in the spiked Wigner model. The goal of submatrix localization is now to extract a submatrix of $\bbf{Y}$ of size $pn \times pn$ with larger mean.
	\item \textbf{Community Detection} in the Stochastic Block Model (SBM). As shown in~\cite{deshpande2016asymptotic,lelarge2016fundamental} recovering two communities of size $pn$ and $(1-p)n$ in a dense SBM of $n$ vertices is (in some sense) ``equivalent'' to the spiked Wigner model with prior
		$$
		P_0 = 
		p \, \delta_{\!\sqrt{\frac{1-p}{p}}}
		+(1-p) \, \delta_{\!-\sqrt{\frac{p}{1-p}}} \,.
		$$
	\item \textbf{$\Z / 2$ synchronization}. This corresponds to the spiked Wigner model with Rademacher prior $P_0 = \frac{1}{2} \delta_{-1} + \frac{1}{2} \delta_{+1}$.
	\item \textbf{High-dimensional Gaussian mixture clustering}. Consider the multidimensional version of the spiked Wishart model where $\bbf{U} \in \R^{n \times k}$ and $\bbf{V} \in \R^{m \times k}$. If one take $P_V$ (the distribution of the rows of $\bbf{V}$) to be supported by the canonical basis of $\R^k$, the model is equivalent to the clustering of $m$ points (the columns of $\bbf{Y}$) in $n$ dimensions from a Gaussian mixture model. The centers of the clusters are here the columns of $\bbf{U}$.
\end{itemize}

{%
	\setcounter{tocdepth}{1}
	\tableofcontents
}

\setcounter{tocdepth}{3}

\input{./gaussian_noise.tex}
\input{./decoupling3.tex}
\input{./xx3.tex}
\input{./uv3.tex}

\appendix
\chapter*{Appendix}
\addcontentsline{toc}{chapter}{Appendix}
\renewcommand{\thesection}{\Alph{section}}

\numberwithin{proposition}{section}
\numberwithin{lemma}{section}
\numberwithin{theorem}{section}
\numberwithin{corollary}{section}
\input{./appendix.tex}

{%
	\bibliographystyle{plain}
	\bibliography{../../formule_RS/references}
}
\end{document}

%% file: gaussian_noise.tex
\chapter{Bayesian inference in Gaussian noise}\label{chap:gaussian}

We introduce in this section some general properties of Bayesian inference in presence of additive Gaussian noise, that will be used repeatedly in the sequel.

\section{Definitions and problem setting}

As explained in the introduction, we will be interested in inference problems of the form:
\begin{equation}\label{eq:inference}
	\bbf{Y} = \sqrt{\lambda} \, \bbf{X} + \bbf{Z} \,,
\end{equation}
where the signal $\bbf{X}$ is sampled according some probability distribution $P_X$ over $\R^n$, and where the noise $\bbf{Z} =(Z_1, \dots, Z_n) \iid \cN(0,1)$ is independent from $\bbf{X}$. In Sections~\ref{sec:xx} and~\ref{sec:uv}, $\bbf{X}$ will typically be a low-rank matrix. The parameter $\lambda \geq 0$ plays the role of a signal-to-noise ratio.
We assume that $P_X$ admits a finite second moment: $\E \|\bbf{X}\|^2 < \infty$.

Given the observation channel~\eqref{eq:inference}, the goal of the statistician is to estimate $\bbf{X}$ given the observations $\bbf{Y}$. We will assume to be in the ``Bayes-optimal'' setting, where the statistician knows all the parameters of the inference model, that is the prior distribution $P_X$ and the signal-to-noise ratio $\lambda$. We measure the performance of an estimator $\widehat{\theta}$ (i.e.\ a measurable function of the observations $\bbf{Y}$) by its Mean Squared Error:
$$
\MSE(\widehat{\theta}) = \E \left[ \| \bbf{X} - \widehat{\theta}(\bbf{Y}) \|^2 \right].
$$
One of our main quantity of interest will be the Minimum Mean Squared Error
$$
\MMSE(\lambda) = \min_{\widehat{\theta}} \MSE(\widehat{\theta}) = \E \left[ \big\| \bbf{X} - \E[\bbf{X}|\bbf{Y}] \big\|^2 \right],
$$
where the minimum is taken over all measurable function $\widehat{\theta}$ of the observations $\bbf{Y}$.
Since the optimal estimator (in term of Mean Squared Error) is the posterior mean of $\bbf{X}$ given $\bbf{Y}$, a natural object to study is the posterior distribution of $\bbf{X}$.

By Bayes rule, the posterior distribution of $\bbf{X}$ given $\bbf{Y}$ is
\begin{equation}\label{eq:posterior_0}
	dP( \bbf{x} \, | \, \bbf{Y}  ) 
	= \frac{1}{\cZ(\lambda,\bbf{Y})} e^{H_{\lambda,\bbf{Y}}(\bbf{x})} dP_X(\bbf{x}) \,,
\end{equation}
where 
$$
H_{\lambda,\bbf{Y}}(\bbf{x})=
\sqrt{\lambda}\, \bbf{x}^{\sT} \bbf{Y} - \frac{\lambda}{2} \| \bbf{x} \|^2
=
\sqrt{\lambda} \, \bbf{x}^{\sT} \bbf{Z} + \lambda \, \bbf{x}^{\sT} \bbf{X} - \frac{\lambda}{2} \| \bbf{x} \|^2 \,.
$$
\begin{definition}
	$H_{\lambda,\bbf{Y}}$ is called the Hamiltonian\protect\footnotemark and the normalizing constant
	$$
	\cZ(\lambda,\bbf{Y}) = \int dP_X(\bbf{x}) e^{H_{\lambda,\bbf{Y}}(\bbf{x})}
	$$
	is called the partition function. 
\end{definition}
\footnotetext{According to the physics convention, this should be minus the Hamiltonian, since a physical system tries to minimize its energy. However, we chose here to remove it for simplicity.}
Expectations with respect the posterior distribution~\eqref{eq:posterior_0} will be denoted by the Gibbs brackets $\langle \cdot \rangle_{\lambda}$:
$$
\big\langle f(\bbf{x}) \big\rangle_{\lambda}
= \E \big[f(\bbf{X})|\bbf{Y} \big]
= \frac{1}{\cZ(\lambda,\bbf{Y})} \int dP_X(\bbf{x}) f(\bbf{x}) e^{H_{\lambda,\bbf{Y}}(\bbf{x})} \,,
$$
for any measurable function $f$ such that $f(\bbf{X})$ is integrable.

\begin{definition}
	$F(\lambda) = \E \log \cZ(\lambda,\bbf{Y})$ is called the free energy\footnotemark. It is related to the mutual information between $\bbf{X}$ and $\bbf{Y}$ by
	\begin{equation}\label{eq:f_i}
		F(\lambda) = \frac{\lambda}{2} \E \| \bbf{X} \|^2 - I(\bbf{X};\bbf{Y})\,.
	\end{equation}
\end{definition}
\footnotetext{This is in fact minus the free energy, but we chose to remove the minus sign for simplicity.}
\begin{proof}
	The mutual information $I(\bbf{X};\bbf{Y})$ is defined as the Kullback-Leibler divergence between $P_{(X,Y)}$, the joint distribution of $(\bbf{X},\bbf{Y})$ and $P_X \! \otimes \! P_Y$ the product of the marginal distributions of $\bbf{X}$ and $\bbf{Y}$. $P_{(X,Y)}$ is absolutely continuous with respect to $P_X \otimes P_Y$ with Radon-Nikodym derivative:
	$$
	\frac{d P_{(X,Y)}}{d P_X \! \otimes \! P_Y}(\bbf{X},\bbf{Y})
	=
	\frac{\exp\left(-\frac{1}{2}\| \bbf{Y} - \sqrt{\lambda} \bbf{X} \|^2\right)}{\int \exp\left(-\frac{1}{2}\| \bbf{Y} - \sqrt{\lambda} \bbf{x} \|^2\right) dP_X(\bbf{x})} \,.
	$$
	Therefore
	\begin{align*}
		I(\bbf{X};\bbf{Y}) 
		&= \E \log\left( \frac{d P_{(X,Y)}}{d P_X \!\otimes \!P_Y}(\bbf{X},\bbf{Y}) \right)
		= - \E \log \int dP_X(\bbf{x}) \exp \left( \sqrt{\lambda} \bbf{x}^{\sT} \bbf{Y} - \sqrt{\lambda} \bbf{X}^{\sT} \bbf{Y} - \frac{\lambda}{2} \|\bbf{x}\|^2 + \frac{\lambda}{2}\|\bbf{X}\|^2 \right)
		\\
		&= - F(\lambda) + \frac{\lambda}{2} \E \| \bbf{X} \|^2 \,.
	\end{align*}
\end{proof}

We state now two basic properties of the $\MMSE$. A more detailed analysis can be found in~\cite{guo2011estimation,wu2012functional}.
\begin{proposition}\label{prop:mmse_dec}
	$\lambda \mapsto \MMSE(\lambda)$ is non-increasing over $\R_+$. Moreover
	\begin{itemize}
		\item $\MMSE(0) = \E\|\bbf{X} - \E[\bbf{X}]\|^2$,
		\item $\MMSE(\lambda) \xrightarrow[\lambda \to +\infty]{} 0$.
	\end{itemize}
\end{proposition}
\begin{proposition}\label{prop:mmse_cont}
	$\lambda \mapsto \MMSE(\lambda)$ is continuous over $\R_+$.
\end{proposition}
The proofs of Proposition~\ref{prop:mmse_dec} and~\ref{prop:mmse_cont} can respectively be found in Appendix~\ref{sec:proof_mmse_dec} and~\ref{sec:proof_mmse_cont}.

\section{The Nishimori identity}

We will often consider i.i.d.\ samples $\bbf{x}^{(1)}, \dots, \bbf{x}^{(k)}$ from the posterior distribution $P( \cdot \, | \, \bbf{Y} )$, independently of everything else. Such samples are called replicas. The (obvious) identity below	(which is simply Bayes rule) will be used repeatedly. It states that the planted solution $\bbf{X}$ behaves like a replica.

\begin{proposition}[Nishimori identity] \label{prop:nishimori}
	Let $(\bbf{X},\bbf{Y})$ be a couple of random variables on a polish space. Let $k \geq 1$ and let $\bbf{x}^{(1)}, \dots, \bbf{x}^{(k)}$ be $k$ i.i.d.\ samples (given $\bbf{Y}$) from the distribution $\P(\bbf{X}=\cdot \, | \, \bbf{Y})$, independently of every other random variables. Let us denote $\langle \cdot \rangle$ the expectation with respect to $\P(\bbf{X}=\cdot \, | \, \bbf{Y})$ and $\E$ the expectation with respect to $(\bbf{X},\bbf{Y})$. Then, for all continuous bounded function $f$
	$$
	\E \big\langle f(\bbf{Y},\bbf{x}^{(1)}, \dots, \bbf{x}^{(k)}) \big\rangle
	=
	\E \big\langle f(\bbf{Y},\bbf{x}^{(1)}, \dots, \bbf{x}^{(k-1)}, \bbf{X}) \big\rangle\,.
	$$
\end{proposition}
\begin{proof}
	It is equivalent to sample the couple $(\bbf{X},\bbf{Y})$ according to its joint distribution or to sample first $\bbf{Y}$ according to its marginal distribution and then to sample $\bbf{X}$ conditionally to $\bbf{Y}$ from its conditional distribution $\P(\bbf{X}=\cdot \,|\,\bbf{Y})$. Thus the $(k+1)$-tuple $(\bbf{Y},\bbf{x}^{(1)}, \dots,\bbf{x}^{(k)})$ is equal in law to $(\bbf{Y},\bbf{x}^{(1)},\dots,\bbf{x}^{(k-1)},\bbf{X})$.
\end{proof}

\section{The I-MMSE relation}\label{sec:i_mmse}
We present now the very useful ``I-MMSE'' relation from~\cite{guo2005mutual}. This relation was previously known (under a different formulation) as ``de Brujin identity'' see \cite[Equation 2.12]{stam1959some}.
\begin{proposition}\label{prop:i_mmse}
	For all $\lambda \geq 0$,
	\begin{equation}\label{eq:i_mmse}
		\frac{\partial}{\partial \lambda} I(\bbf{X};\bbf{Y}) = \frac{1}{2} \MMSE(\lambda)
		\qquad \text{and} \qquad
		F'(\lambda) = \frac{1}{2} \E \langle \bbf{x}^{\sT} \bbf{X} \rangle_{\lambda} 
		= \frac{1}{2} \big(\E \| \bbf{X}\|^2 - \MMSE(\lambda) \big)
		\,.
	\end{equation}
	$F$ thus is a convex, differentiable, non-decreasing, and $\frac{1}{2}\E\|\bbf{X}\|^2$-Lipschitz function over $\R_{\geq 0}$. If $P_X$ is not a Dirac mass, then $F$ is strictly convex.
\end{proposition}
Proposition~\ref{prop:i_mmse} is proved in Appendix~\ref{sec:proof_i_mmse}.
Proposition~\ref{prop:i_mmse} reduces the computation of the MMSE to the computation of the free energy. This will be particularly useful because the free energy $F$ is much easier to handle than the MMSE.

We end this section with the simplest model of the form~\eqref{eq:inference}, namely the additive Gaussian scalar channel:
\begin{equation}\label{eq:additive_scalar_channel}
	Y = \sqrt{\lambda} X + Z \,,
\end{equation}
where $Z \sim \cN(0,1)$ and $X$ is sampled from a distribution $P_0$ over $\R$, independently of $Z$. The corresponding free energy and the MMSE are respectively
	\begin{equation}\label{eq:def_psi_p0}
		\psi_{P_0}(\lambda)= \E \log \int dP_0(x)e^{\sqrt{\lambda} \,Y x - \lambda x^2/2}
		\quad \text{and} \quad \MMSE_{P_0}(\lambda) = \E \big[\big(X - \E[X|Y]\big)^2\big] \,.
	\end{equation}
The study of this simple inference channel will be very useful in the following, because we will see that the inference problems that we are going to study enjoy asymptotically a ``decoupling principle'' that reduces them to scalar channels like~\eqref{eq:additive_scalar_channel}.


Let us compute the mutual information and the MMSE for particular choices of prior distributions:
\begin{example}[Gaussian prior: $P_0 = \cN(0,1)$]\label{ex:gaussian_psi}
			In that case $\E[X|Y]$ is simply the orthogonal projection of $X$ on $Y$: 
			$$
			\E[X|Y] = \frac{\E[XY]}{\E[Y^2]} Y = \frac{\sqrt{\lambda}}{1+\lambda} Y.
			$$
			One deduces $\MMSE_{P_0}(\lambda) = \frac{1}{1+\lambda}$.
			Using~\eqref{eq:i_mmse}, we get $I(X;Y) = \frac{1}{2} \log(1+\lambda)$ and $\psi_{P_0}(\lambda) = \frac{1}{2}\big(\lambda - \log(1+\lambda)\big)$.
\end{example}

\begin{remark}[Worst-case prior]\label{rem:worst_case}
	Let $P_0$ be a probability distribution on $\R$ with unit second moment $\E_{P_0}[X^2] = 1$. By considering the estimator $\widehat{x} = \frac{\sqrt{\lambda}}{1+\lambda} Y$, one obtain $\MMSE_{P_0}(\lambda) \leq \frac{1}{1+ \lambda}$. We conclude:
	$$
	\sup_{P_0} \MMSE_{P_0}(\lambda) = \frac{1}{1+\lambda} \qquad \text{and} \qquad
	\inf_{P_0} \psi_{P_0}(\lambda) = \frac{1}{2}(\lambda - \log(1+\lambda)),
	$$
	where the supremum and infimum are both over the probability distributions that have unit second moment. The standard normal distribution $P_0 = \cN(0,1)$ achieves both extrema.
\end{remark}

\begin{example}[Rademacher prior: $P_0 = \frac{1}{2} \delta_{+1} + \frac{1}{2} \delta_{-1}$] We compute $\psi_{P_0}(\lambda) = \E \log \cosh(\sqrt{\lambda} Z + \lambda)- \frac{\lambda}{2}$ and $I(X;Y) = \lambda - \E \log \cosh(\sqrt{\lambda} Z + \lambda)$. The I-MMSE relation gives
			\begin{align*}
				\frac{1}{2} \MMSE(\lambda) 
				&= \frac{\partial}{\partial \lambda} I(X;Y)
				= 1 - \E \Big[\big(\frac{1}{2\sqrt{\lambda}}Z+1\big) \tanh\big(\sqrt{\lambda}Z + \lambda\big)\Big]
				\\
				&= 1 - \E \tanh(\sqrt{\lambda}Z + \lambda) - \frac{1}{2} \E \tanh'(\sqrt{\lambda}Z + \lambda)
				\\
				&= \frac{1}{2} - \E \tanh(\sqrt{\lambda}Z + \lambda) + \frac{1}{2} \E \tanh^2(\sqrt{\lambda}Z + \lambda)
			\end{align*}
			where we used Gaussian integration by parts. Since by the Nishimori property $\E \langle x X \rangle_{\lambda} = \E \langle x \rangle_{\lambda}^2$, one has $\E \tanh(\sqrt{\lambda}Z + \lambda) = \E \tanh^2(\sqrt{\lambda}Z + \lambda)$ and therefore $\MMSE(\lambda) = 1 - \E \tanh(\sqrt{\lambda}Z + \lambda)$.
\end{example}

\section{A warm-up: ``needle in a haystack'' problem}\label{sec:rem}

In order to illustrate the results seen in the previous sections, we study now a very simple inference model.
Let $(e_1,\dots,e_{2^n})$ be the canonical basis of $\R^{2^n}$. Let $\sigma_0 \sim \Unif(\{1,\dots,2^n\})$ and define $\bbf{X} = e_{\sigma_0}$ (i.e.\ $\bbf{X}$ is chosen uniformly over the canonical basis of $\R^{2^n}$).
Suppose here that we observe:
$$
\bbf{Y} = \sqrt{\lambda n} \bbf{X} + \bbf{Z} \,,
$$
where $\bbf{Z} = (Z_1, \dots, Z_{2^n}) \iid \cN(0,1)$, independently from $\sigma_0$. The goal here is to estimate $\bbf{X}$ or equivalently to find $\sigma_0$.
The posterior distribution reads:
\begin{align*}
\P(\sigma_0 = \sigma | \bbf{Y}) = 
\P(\bbf{X} = e_{\sigma} | \bbf{Y}) &= 
\frac{1}{\cZ_{n}(\lambda)} 2^{-n} \exp \Big( \sqrt{\lambda n} e_{\sigma}^{\sT} \bbf{Y} - \frac{\lambda n}{2} \|e_{\sigma}\|^2 \Big)
\\
&=
\frac{1}{\cZ_{n}(\lambda)} 2^{-n} \exp \Big( \sqrt{\lambda n}Z_{\sigma} + \lambda n \bbf{1}(\sigma = \sigma_0) - \frac{\lambda n}{2} \Big) ,
\end{align*}
where $\cZ_n(\lambda)$ is the partition function
$$
\cZ_{n}(\lambda) = \frac{1}{2^n} \sum_{\sigma=1}^{2^n} \exp \Big( \sqrt{\lambda n}Z_{\sigma} + \lambda n \bbf{1}(\sigma = \sigma_0) - \frac{\lambda n}{2} \Big) \,.
$$
We will be interested in computing the free energy $F_n(\lambda) = \frac{1}{n} \E \log \cZ_n(\lambda)$ in order to deduce then the minimal mean squared error using the I-MMSE relation~\eqref{eq:i_mmse} presented in the previous section. 

Although its simplicity, this model is interesting for many reasons. 
First, it is one of the simplest statistical model for which one observes a phase transition.
Second it is the ``planted'' analog of the random energy model (REM) introduced in statistical physics by Derrida~\cite{derrida1980random,derrida1981random}, for which the free energy reads $\frac{1}{n} \E \log \sum_{\sigma}\frac{1}{2^n} \exp\big(\sqrt{\lambda n} Z_{\sigma}\big)$. 

We start by computing the limiting free energy:

\begin{theorem}
$$
\lim_{n\to \infty} F_n(\lambda) = 
\begin{cases}
	0 & \text{if} \quad \lambda \leq 2 \log 2 \,, \\
	\frac{\lambda}{2} - \log(2)& \text{if} \quad \lambda \geq 2 \log 2 \,.
\end{cases}
$$
\end{theorem}
\begin{proof}
	Using Jensen's inequality
	\begin{align*}
	F_n(\lambda) 
	&\leq \frac{1}{n} \E \log \E \left[ \cZ_n(\lambda) \middle| \sigma_0, Z_{\sigma_0} \right]
	= \frac{1}{n} \E \log \left(1 - \frac{1}{2^n} + e^{\sqrt{\lambda n} Z_{\sigma_0} + \frac{\lambda n}{2} - \log(2)n}\right)
	\\
	&\leq \frac{1}{n} \E \log \left(1 + e^{\frac{\lambda n}{2} - \log(2)n}\right) + \sqrt{\frac{\lambda}{n}}
	\ \xrightarrow[n \to \infty]{}  \
	\begin{cases}
		0 & \text{if} \quad \lambda \leq 2 \log(2)\,, \\
		\frac{\lambda}{2} - \log(2) & \text{if} \quad \lambda \geq 2 \log(2) \,.
	\end{cases}
	\end{align*}
	$F_n$ is non-negative since $F_n(0) = 0$ and $F_n$ is non-decreasing. We have therefore $F_n(\lambda) \xrightarrow[n \to \infty]{} 0$ for all $\lambda \in [0,2\log(2)]$.
	We have also, by only considering the term $\sigma = \sigma_0$:
	$$
	F_n(\lambda) \geq \frac{1}{n} \E \log \left(\frac{e^{\sqrt{\lambda n}Z_{\sigma_0} + \frac{\lambda n}{2}}}{2^n}\right) = \frac{\lambda}{2} - \log(2) \,.
	$$
	We obtain therefore that $F_n(\lambda ) \xrightarrow[n \to \infty]{} \frac{\lambda}{2} - \log(2)$ for $\lambda \geq 2 \log(2)$.
\end{proof}
\\

Using the I-MMSE relation~\eqref{eq:i_mmse}, we deduce the limit of the minimum mean Squared Error $\MMSE_n(\lambda) =  \min_{\widehat{\theta}} \E \| \bbf{X} - \widehat{\theta}(\bbf{Y}) \|^2$:
$$
\frac{1}{2} \MMSE_n(\lambda) = \E \|\bbf{X}\|^2 - F_n'(\lambda) = 1 - F_n'(\lambda) \,.
$$
$F_n$ is a convex function of $\lambda$, so its derivative converges to the derivative of its limit at each $\lambda$ at which the limit is differentiable, i.e.\ for all $\lambda \in (0,+\infty) \setminus \{2 \log(2) \}$. We obtain therefore that for all $\lambda >0$,
\begin{itemize}
	\item if $\lambda < 2 \log(2)$, then $\MMSE_n(\lambda) \xrightarrow[n \to \infty]{} 1$: one can not recover $\bbf{X}$ better than a random guess.
	\item if $\lambda > 2 \log(2)$, then $\MMSE_n(\lambda) \xrightarrow[n \to \infty]{} 0$: one can recover $\bbf{X}$ perfectly.
\end{itemize}
Of course, the result we obtain here is (almost) trivial since the maximum likelihood estimator 
$$
\widehat{\sigma}(\bbf{Y}) = \underset{1 \leq \sigma \leq 2^n}{\text{arg\, max}} \, Y_{\sigma}
$$
of $\sigma_0$ is easy to analyze. Indeed, $\max_{\sigma} Z_{\sigma} \simeq \sqrt{2 \log(2)n}$ with high probability so that the maximum likelihood estimator recovers perfectly the signal for $\lambda > 2 \log(2)$ with high probability.

%% file: decoupling3.tex
\chapter{A decoupling principle}\label{chap:overlap}

We present in this section a general ``decoupling principle'' that will be particularly useful in the study of planted models.
We consider here the setting where $\bbf{X} = (X_1, \dots, X_n) \iid P_0$ for some probability distribution $P_0$ over $\R$ with support $S$.
Let $\bbf{Y} \in \R^m$ be another random variable that accounts for noisy observation of $\bbf{X}$.
The goal is again to recover the planted vector $\bbf{X}$ from the observations $\bbf{Y}$.
We suppose that the distribution of $\bbf{X}$ given $\bbf{Y}$ takes the following form
\begin{equation}\label{eq:posterior_general}
	\P(\bbf{X} \in A \ | \ \bbf{Y}) = \frac{1}{\cZ_n(\bbf{Y})} \int_{\bbf{x} \in A} dP_0^{\otimes n}(\bbf{x}) e^{H_n(\bbf{x},\bbf{Y})}, \quad \text{for all Borel set } A \subset \R^n,
\end{equation}
where $H_n$ is a measurable function on $\R^{n} \times \R^m$ that can be equal to $-\infty$ (in which case, we use the convention $\exp(- \infty) = 0$) and $\cZ_n(\bbf{Y}) = \int dP_0^{\otimes n}(\bbf{x}) e^{H_n(\bbf{x},\bbf{Y})}$ is the appropriate normalization.
We assume that $\E |\log \cZ_n(\bbf{Y})| < \infty$ in order to define the free energy
$$
F_n = \frac{1}{n} \E \log \cZ_n(\bbf{Y}) = \frac{1}{n} \E \log \left( \int dP_0^{\otimes n}(\bbf{x}) e^{H_n(\bbf{x},\bbf{Y})} \right).
$$
In the following, we are going to drop the dependency in $\bbf{Y}$ of $H_n(\bbf{x},\bbf{Y})$ and simply write $H_n(\bbf{x})$.
\\

We introduce now an important notation: the overlap between to vectors $\bbf{u},\bbf{v} \in \R^n$. This is simply the normalized scalar product:
$$
\bbf{u} \cdot \bbf{v} = \frac{1}{n} \sum_{i=1}^n u_i v_i \,.
$$

One should see $\bbf{x}$ as a system of $n$ spins $(x_1, \dots, x_n)$ interacting through the (random) Hamiltonian $H_n$. Our inference problem should be understood as the study of this spin glass model. A central quantity of interest in spin glass theory is the overlaps $\bbf{x}^{(1)} \!\cdot \bbf{x}^{(2)}$ between two replicas, i.e.\ the normalized scalar product between two independent samples $\bbf{x}^{(1)}$ and $\bbf{x}^{(2)}$ from~\eqref{eq:posterior_general}.
Understanding this quantity is fundamental because it allows to deduce the distance between two typical configurations of the system and thus encodes the ``geometry'' of the ``Gibbs measure''~\eqref{eq:posterior_general}. 

In our statistical inference setting we have
$\bbf{x}^{(1)}\! \cdot \bbf{x}^{(2)} = \bbf{x}^{(1)}\! \cdot \bbf{X}$ in law, 
by the Nishimori identity (Proposition~\ref{prop:nishimori}).
Thus the overlap $\bbf{x}^{(1)} \!\cdot \bbf{x}^{(2)}$ corresponds to the correlation between a typical configuration and the planted configuration. Moreover it is linked to the Minimum Mean Squared Error by
$$
\MMSE = \frac{1}{n} \E \left[ \| \bbf{X} - \langle \bbf{x} \rangle \|^2\right]
= \E_{P_0}[X^2] - \E \left\langle \bbf{x} \cdot \bbf{X} \right\rangle \,,
$$
where $\langle \cdot \rangle$ denotes the expectation with respect to $\bbf{x}$ which is sampled from the posterior $\P(\bbf{X}= \cdot \, | \, \bbf{Y})$ (defined by Equation~\ref{eq:posterior_general}), independently of everything else.
\\

In this section we will see a general principle that states that under a small perturbation of the Gibbs distribution~\eqref{eq:posterior_general}, the overlap $\bbf{x}^{(1)} \cdot \bbf{x}^{(2)}$ between two replicas concentrates around its mean. Such behavior is called ``Replica-Symmetric'' in statistical physics.
It remains to define what ``a small perturbation of the Gibbs distribution'' is.
In spin glass theory, such perturbations are usually obtained by adding small extra terms to the Hamiltonian. 
In our context of Bayesian inference a small perturbation will correspond to a small amount of side-information given to the statistician. This extra information will lead to a new posterior distribution. 
In the following, we will consider two different kind of side-information and we show that the overlaps under the induced posterior concentrate around their mean.

\section{The pinning Lemma}\label{sec:overlap_concentration_montanari}

We suppose here that the support $S$ of $P_0$ is finite. We make this assumption in order to be able to work with the discrete entropy.

In this section, we give extra information to the statistician by revealing a (small) fraction of the coordinates of $\bbf{X}$.
Let us fix $\epsilon\in [0,1]$, and suppose that we have access to the additional observations
$$
Y'_i =
\begin{cases}
	X_i &\text{if } L_i = 1 \,, \\
	* &\text{if } L_i = 0 \,,
\end{cases}
\qquad \text{for} \quad 1 \leq i \leq n ,
$$
where $L_i \iid  \Ber(\epsilon)$ and $*$ is a value that does not belong to $S$. 
The posterior distribution of $\bbf{X}$ is now
\begin{equation}\label{eq:posterior_montanari}
	\P(\bbf{X}=\bbf{x} \, | \, \bbf{Y},\bbf{Y}') = \frac{1}{\cZ_{n,\epsilon}} \left(\prod_{i | L_i=1}\1(x_i=Y'_i) \right) \left( \prod_{i | L_i=0} P_0(x_i) \right) e^{H_n(\bbf{x})} \,,
\end{equation}
where $\cZ_{n,\epsilon}$ is the appropriate normalization constant.
For $\bbf{x} \in S^n$ we will write
\begin{equation} \label{eq:bar}
	\bar{\bbf{x}} = (\bar{x}_1, \dots, \bar{x}_n) = (L_1 X_1 + (1-L_1) x_1, \dots, L_n X_n + (1-L_n)x_n) \,.
\end{equation}
$\bar{\bbf{x}}$ is thus obtained by replacing the coordinates of $\bbf{x}$ that are revealed by $\bbf{Y}'$ by their revealed values. The notation $\bar{\bbf{x}}$ allows us to obtain a convenient expression for the free energy of the perturbed model:
$$
F_{n,\epsilon} = \frac{1}{n} \E \log \cZ_{n,\epsilon} = \frac{1}{n} \E \Big[ \log \sum_{\bbf{x} \in S^n} P_0(\bbf{x}) \exp(H_{n}(\bar{\bbf{x}}))\Big] \,.
$$

\begin{proposition} \label{prop:approximation_f_n_epsilon}
	For all $n \geq 1$ and all $\epsilon\in [0,1]$, we have
	$$
	|F_{n,\epsilon} - F_{n} | \leq H(P_0) \epsilon \,.
	$$
\end{proposition}
\begin{proof} Let us compute
	\begin{align*}
		P\big(\bbf{Y}' \, | \, \bbf{Y},\bbf{L} \big)
		&=
		\int \1(x_i = Y_i' \ \text{for all} \ i \ \text{such that} \ L_i=1) dP(\bbf{x} \, | \, \bbf{Y})
		\\
		&=
		\frac{1}{\cZ_n}
		\sum_{\bbf{x} \in S^n} \1(x_i = Y_i' \ \text{for all} \ i \ \text{such that} \ L_i=1) e^{H_n(\bbf{x})} \prod_{i=1}^n P_0(x_i)
		\\
		&=
		\frac{\cZ_{n,\epsilon}}{\cZ_n}
		\prod\limits_{i | L_i = 1} \!\!\! P_0(Y'_i) 
		=
		\frac{\cZ_{n,\epsilon}}{\cZ_n}
		\, P\big(\bbf{Y}' \, | \, \bbf{L}\big) \,.
	\end{align*}
	Therefore, $nF_{n,\epsilon} - nF_{n} = H(\bbf{Y}' \, | \, \bbf{L}) - H(\bbf{Y}' \, | \, \bbf{Y},\bbf{L})$ and the proposition follows from the fact that $0 \leq H(\bbf{Y}'| \bbf{Y},\bbf{L}) \leq H(\bbf{Y}'|\bbf{L}) = n \epsilon H(P_0)$.
\end{proof}
\\

From now we suppose $\epsilon_0 \in (0,1]$ to be fixed and consider $\epsilon \in [0,\epsilon_0]$.
The following lemma comes from~\cite{andrea2008estimating} and is sometimes known as the ``pinning lemma''. It shows that the extra information $\bbf{Y}'$ forces the correlations between the spins under the posterior~\eqref{eq:posterior_montanari} to vanish.

\begin{lemma}[Lemma 3.1 from~\cite{andrea2008estimating}\,]\label{lem:montanari}
	For all $\epsilon_0 \in [0,1]$, we have
	$$
	\int_0^{\epsilon_0}\!d\epsilon  \left( \frac{1}{n^2} \sum_{1 \leq i,j \leq n} I(X_i;X_j | \bbf{Y},\bbf{Y'}) \right) \leq \frac{2}{n} H(P_0) \,.
	$$
\end{lemma}

Let $\langle \cdot \rangle_{n,\epsilon}$ denote the expectation with respect to two independent samples $\bbf{x}^{(1)},\bbf{x}^{(2)}$ from the posterior~\eqref{eq:posterior_montanari}.
Lemma~\ref{lem:montanari} implies that the overlap between these two replicas concentrates:

\begin{proposition} \label{prop:overlap_concentration_montanari}
	There exists a constant $C>0$ that only depends on $P_0$ such that for all $\epsilon_0 \in [0,1]$, 
	\begin{align*}
		\int_0^{\epsilon_0}\!\!d\epsilon \,\E \left\langle \left( \frac{1}{n} \sum_{i=1}^n x^{(1)}_i x^{(2)}_i - \Big\langle \frac{1}{n} \sum_{i=1}^n x_i^{(1)} x_i^{(2)} \Big\rangle_{\!\! n,\epsilon} \right)^{\!\! 2} \right\rangle_{\!\! n,\epsilon} 
		\leq C \sqrt{\frac{\epsilon_0}{n}}.
	\end{align*}
\end{proposition}
\begin{proof}
	\begin{align*}
		\big\langle (\bbf{x}^{(1)} \cdot \bbf{x}^{(2)} - \langle \bbf{x}^{(1)} \cdot \bbf{x}^{(2)} \rangle_{n,\epsilon} )^2 \big\rangle_{n,\epsilon} &=
		\big\langle (\bbf{x}^{(1)} \! \cdot \bbf{x}^{(2)})^2 \big\rangle_{n,\epsilon} - \big\langle \bbf{x}^{(1)}\! \cdot \bbf{x}^{(2)} \big\rangle_{n,\epsilon}^2 
		\\
		&= \frac{1}{n^2} \sum_{1 \leq i,j \leq n} \big\langle x_i^{(1)} x_i^{(2)} x_j^{(1)} x_j^{(2)} \big\rangle_{n,\epsilon} - \big\langle x_i^{(1)} x_i^{(2)} \big\rangle_{n,\epsilon} \big\langle x_j^{(1)} x_j^{(2)} \big\rangle_{n,\epsilon} \\
		&= \frac{1}{n^2} \sum_{1 \leq i,j \leq n} \langle x_i x_j\rangle_{n,\epsilon}^2 - \langle x_i \rangle_{n,\epsilon}^2 \langle x_j\rangle_{n,\epsilon}^2.
	\end{align*}
	Let now $i,j \in \{1, \dots, n\}$. The support of $P_0$ is finite and thus included in $[-K,K]$ for some  $K>0$. This gives:
	\begin{align*}
		\langle x_i x_j\rangle_{n,\epsilon}^2 &- \langle x_i \rangle_{n,\epsilon}^2 \langle x_j\rangle_{n,\epsilon}^2
		\leq  2K^2 | \langle x_i x_j\rangle_{n,\epsilon} - \langle x_i \rangle_{n,\epsilon} \langle x_j\rangle_{n,\epsilon} | \\
		&= 2K^2\Big| \sum_{x_i,x_j}\! x_i x_j \P(X_i=x_i, X_j=x_j | \bbf{Y},\bbf{Y'})- x_i x_j \P(X_i=x_i | \bbf{Y},\bbf{Y'}) \P(X_j = x_j |\bbf{Y},\bbf{Y'}) \Big| \\
		&\leq 4 K^2\Dtv \big(\P(X_i=\cdot, X_j=\cdot |\bbf{Y},\bbf{Y'}); \P(X_i=\cdot | \bbf{Y},\bbf{Y'})\otimes \P(X_j =\cdot |\bbf{Y},\bbf{Y'}) \big)\\
		&\leq 4 K^2\sqrt{\Dkl \big(\P(X_i=\cdot, X_j=\cdot |\bbf{Y},\bbf{Y'}); \P(X_i=\cdot | \bbf{Y},\bbf{Y'})\otimes \P(X_j =\cdot |\bbf{Y},\bbf{Y'}) \big)}
	\end{align*}
	by Pinsker's inequality.
	Since 
	$$I(X_i;X_j | \bbf{Y},\bbf{Y'}) = \E\big[ \Dkl \big(\P(X_i=\cdot, X_j=\cdot |\bbf{Y},\bbf{Y'}); \P(X_i=\cdot | \bbf{Y},\bbf{Y'})\otimes \P(X_j =\cdot |\bbf{Y},\bbf{Y'}) \big)\big],$$ we get using Lemma~\ref{lem:montanari}:
	\begin{align*}
		\int_0^{\epsilon_0}\! d\epsilon \, 
		\E \left\langle \left( \bbf{x}^{(1)} \!\cdot \bbf{x}^{(2)} - \langle \bbf{x}^{(1)} \!\cdot \bbf{x}^{(2)} \rangle_{n,\epsilon} \right)^{\!2} \right\rangle_{\!\!n,\epsilon} 
		&\leq
		4K^2 \sqrt{\epsilon_0 \int_0^{\epsilon_0} \! d\epsilon \Big( \frac{1}{n^2} \sum_{1 \leq i,j \leq n} I(X_i;X_j | \bbf{Y},\bbf{Y'}) \Big)}
		\\
		&\leq 4K^2 \sqrt{\frac{2\epsilon_0 H(P_0)}{n}}.
	\end{align*}
\end{proof}

\section{Noisy side Gaussian channel}\label{sec:overlap_concentration_gg}

We consider in this section of a different kind of side-information: an observation of the signal $\bbf{X}$ perturbed by some Gaussian noise. It was proved in~\cite{korada2010tight} for CDMA systems that such perturbations forces the overlaps to concentrate around their means. The principle here is in fact more general and holds for any observation system, provided some concentration property of the free energy.

We suppose here that the prior $P_0$ has a bounded support $S \subset [-K,K]$, for some $K > 0$.
Let $a > 0$ and $(s_n)_n \in (0,1]^{\N}$. Let $(Z_i)_{1 \leq i \leq n} \iid \mathcal{N}(0,1)$ independently of everything else. 
The extra side-information takes now the form
\begin{equation} \label{eq:pert_scalar}
	Y'_i = a \sqrt{s_n} X_i + Z_i, \quad \text{for } 1 \leq i \leq n.
\end{equation}
The posterior distribution of $\bbf{X}$ given $\bbf{Y},\bbf{Y}'$ is now $P(\bbf{x} \, | \, \bbf{Y},\bbf{Y}') = \frac{1}{\cZ_{n,a}^{\text{(pert)}}} P_0^{\otimes n}(\bbf{x}) \exp\big(H_{n,a}^{\text{(pert)}}(\bbf{x})\big)$, where
$
H_{n,a}^{\text{(pert)}}(\bbf{x}) = H_{n}(\bbf{x}) + h_{n,a}(\bbf{x})
$
and
$$
h_{n,a}(\bbf{x}) = \sum_{i=1}^n a \sqrt{s_n} Z_i x_i + a^2 s_n x_i X_i -\frac{1}{2} a^2 s_n x_i^2 \,.
$$
$\cZ_{n,a}^{\text{(pert)}}$ is the appropriate normalization.
Let us define
$$
\phi: a \mapsto \frac{1}{n s_n} \log \left( \int dP_0^{\otimes n}(\bbf{x}) e^{H^{\text{(pert)}}_{n,a}(\bbf{x})} \right).
$$
We fix now $A \geq 2$.
Define also $v_n(s_n) = \sup_{1/2 \leq a \leq A+1} \E | \phi(a) - \E \phi(a) |$. The following result shows that, in the perturbed system (under some conditions on $v_n$ and $s_n$) the overlap between two replicas concentrates asymptotically around its expected value.

\begin{proposition}[Overlap concentration] \label{prop:overlap_concentration_gg}
	Assume that $v_n(s_n) \xrightarrow[n \to \infty]{} 0$.
	Then there exists a constant $C>0$ that only depends on $K$ such that for all $A \geq 2$,
	$$
	\frac{1}{A-1}
	\int_1^A \E \Big\langle \big(\bbf{x}^{(1)}\cdot\bbf{x}^{(2)} - \E\langle \bbf{x}^{(1)}\cdot\bbf{x}^{(2)} \rangle_{n,a}\big)^2 \Big\rangle_{\!n,a} da 
	\ \leq \ C \Big(\frac{1}{\sqrt{n s_n}} + \sqrt{v_n(s_n)}\Big),
	$$
where $\langle \cdot \rangle_{n,a}$ denotes the distribution of $\bbf{X}$ given $(\bbf{Y},\bbf{Y'})$. $\bbf{x}^{(1)}$ and $\bbf{x}^{(2)}$ are two independent samples from $\langle \cdot \rangle_{n,a}$, independently of everything else.
\end{proposition}

Proposition~\ref{prop:overlap_concentration_gg} is the analog of \cite[Theorem~3.2]{panchenko2013SK} (the Ghirlanda-Guerra identities, see~\cite{ghirlanda1998general})
and is proved analogously is the remaining of the section.
Denote for $\bbf{x} \in S^n$
$$
U(\bbf{x}) 
= \frac{1}{n s_n} \frac{\partial}{\partial a} h_{n,a}(\bbf{x})
= \frac{1}{n} \sum_{i=1}^n \frac{1}{\sqrt{s_n}} Z_i x_i + 2 a x_i X_i - a x_i^2 \,.
$$

\begin{lemma}\label{lem:concentration_energy}
	Let $\bbf{x}$ be a sample from $\langle \cdot \rangle_{n,a}$, independently of everything else.
	Under the conditions of Proposition~\ref{prop:overlap_concentration_gg}, we have for all $A \geq 2$
	$$
	\frac{1}{A-1}
	\int_{1}^A \E \Big\langle \big| U(\bbf{x}) - \E \langle U(\bbf{x}) \rangle_{n,a} \big| \Big\rangle_{\!n,a} da 
	\leq C \Big(\frac{1}{\sqrt{n s_n}} + \sqrt{v_n(s_n)}\Big)
	\,,
	$$
	for some constant $C>0$ that only depends on $K$.
\end{lemma}
Before proving Lemma~\ref{lem:concentration_energy}, let us show how it implies Proposition~\ref{prop:overlap_concentration_gg}.
\\

\begin{proof}[Proof of Proposition~\ref{prop:overlap_concentration_gg}]
	By the bounded support assumption on $P_0$, the overlap between two replicas is bounded by $K^2$, thus
	\begin{equation}\label{eq:gg_trick}
		\left| \E \!\left\langle U(\bbf{x}^{(1)}) \, \bbf{x}^{(1)}\! \cdot\bbf{x}^{(2)} \right\rangle_{\!n,a} 
		\!\!- \E\!\left\langle \bbf{x}^{(1)}\! \cdot\bbf{x}^{(2)} \right\rangle_{\!n,a} \!\! \E\! \left\langle U(\bbf{x}^{(1)}) \right\rangle_{\!n,a} \right| 
		\leq K^2  \E\! \left\langle \big| U(\bbf{x}) - \E\! \left\langle U(\bbf{x}) \right\rangle_{\!n,a} \big| \right\rangle_{\!n,a}.
	\end{equation}
	Let us compute the left-hand side of~\eqref{eq:gg_trick}.
	By Gaussian integration by parts and using the Nishimori identity (Proposition~\ref{prop:nishimori}) we get $\E \big\langle U(\bbf{x}^{(1)}) \big\rangle_{\!n,a} = 2 a \, \E \big\langle \bbf{x}^{(1)} \! \cdot \bbf{x}^{(2)} \big\rangle_{\!n,a}$. Therefore
	$$
	\E\left\langle \bbf{x}^{(1)}\! \cdot \bbf{x}^{(2)} \right\rangle_{\!n,a} \E \left\langle U(\bbf{x}^{(1)}) \right\rangle_{\!n,a} 
	= 2 a \left(\E \left\langle \bbf{x}^{(1)} \! \cdot \bbf{x}^{(2)} \right\rangle_{\!n,a}\right)^{\!2}.
	$$
	Using the same tools, we compute for $\bbf{x}^{(1)},\bbf{x}^{(2)},\bbf{x}^{(3)},\bbf{x}^{(4)} \iid \langle \cdot \rangle_{n,a}$, independently of everything else:
	\begin{align*}
		\E&  \big\langle U(\bbf{x}^{(1)})(\bbf{x}^{(1)}\!\cdot\bbf{x}^{(2)}) \big\rangle_{\!n,a} 
		\\
		&= 2a \E \big\langle (\bbf{x}^{(1)}\!\cdot\bbf{X})(\bbf{x}^{(1)}\!\cdot\bbf{x}^{(2)}) \big\rangle_{\!n,a} + \frac{1}{n \sqrt{s_n}} \sum_{i=1}^n \E Z_i \big\langle x^{(1)}_i (\bbf{x}^{(1)}\!\cdot\bbf{x}^{(2)}) \big\rangle_{\!n,a} - \frac{a}{n} \sum_{i=1}^n \E \big\langle (x^{(1)}_i)^2 (\bbf{x}^{(1)}\!\cdot\bbf{x}^{(2)}) \big\rangle_{\!n,a} \\
		&= 2a \E \big\langle (\bbf{x}^{(1)}\!\cdot\bbf{X})(\bbf{x}^{(1)}\!\cdot\bbf{x}^{(2)}) \big\rangle_{\!n,a} + a \E \big\langle (\bbf{x}^{(1)}\!\cdot\bbf{x}^{(2)})^2 \big\rangle_{\!n,a} - a \E \big\langle (\bbf{x}^{(1)}\!\cdot\bbf{x}^{(3)} + \bbf{x}^{(1)}\!\cdot\bbf{x}^{(4)})(\bbf{x}^{(1)}\!\cdot\bbf{x}^{(2)})\big\rangle_{\!n,a} \\
		&= 2a \E \big\langle (\bbf{x}^{(1)}\!\cdot\bbf{x}^{(2)})^2 \big\rangle_{\!n,a} \,.
	\end{align*}
	Thus, by~\eqref{eq:gg_trick} we have for all $a \in [1,A]$
	$$
	\E \Big\langle \big(\bbf{x}^{(1)}\!\cdot\bbf{x}^{(2)} - \E\big\langle \bbf{x}^{(1)}\!\cdot\bbf{x}^{(2)} \big\rangle_{\!n,a}\big)^2 \Big\rangle_{\!n,a}
	\leq
	\frac{K^2}{2}  \E \left\langle \big| U(\bbf{x}) - \E \left\langle U(\bbf{x}) \right\rangle_{\!n,a} \! \big| \right\rangle_{\!n,a} ,
	$$
	and we conclude by integrating with respect to $a$ over $[1,A]$ and using Lemma~\ref{lem:concentration_energy}.
\end{proof}
\\

\begin{proof}[Proof of Lemma~\ref{lem:concentration_energy}]
	$\phi$ is twice differentiable on $(0,+\infty)$, and for $a > 0$
	\begin{align}
		\phi'(a) &= \langle U(\bbf{x}) \rangle_{n,a} \,, \\
		\phi''(a) &= n s_n \big(\langle U(\bbf{x})^2 \rangle_{n,a} -\langle U(\bbf{x}) \rangle_{n,a}^2\big)
		+ \frac{1}{n} \sum_{i=1}^n \Big\langle 2 x_i X_i - x_i^2 \Big\rangle_{n,a} \,. \label{eq:der_sec_phi}
	\end{align}
	Thus $\big\langle (U(\bbf{x}) - \langle U(\bbf{x})\rangle_{n,a})^2 \big\rangle_{n,a} \leq \frac{1}{n s_n} (\phi''(a) + 2 K^2)$ and
	\begin{align*}
		\int_{1}^A \E \big\langle (U(\bbf{x}) - \langle U(\bbf{x})\rangle_{n,a})^2 \big\rangle_{n,a} da \leq \frac{1}{n s_n} \left(\E \phi'(A) - \E \phi'(1) + 2 K^2 (A-1) \right) \leq \frac{C A}{n s_n} \,,
	\end{align*}
	for some constant $C>0$ (that only depend on $K$),
	because $\E \phi'(a) = 2 a \E \langle \bbf{x} \cdot\bbf{X} \rangle_{n,a}$.
	It remains to show that $\int_{1}^A \E \big| \langle U(\bbf{x}) \rangle_{n,a} - \E \langle U(\bbf{x})\rangle_{n,a}\big| da \leq C A \sqrt{v_n(s_n)}$ for some constant $C>0$ that only depends on $K$.
	\\

	We will use the following lemma on convex functions (from~\cite{panchenko2013SK}, Lemma 3.2).

	\begin{lemma}
		If $f$ and $g$ are two differentiable convex functions then, for any $b >0$
		$$
		|f'(a) - g'(a)| \leq g'(a+b) - g'(a-b) + \frac{d}{b} \,,
		$$
		where $d = |f(a+b) - g(a+b)| + |f(a-b) - g(a-b)| + |f(a) - g(a)|$.
	\end{lemma}

	We apply this lemma to $\lambda \mapsto \phi(\lambda) + \frac{3}{2} K^2 \lambda^2$ and $\lambda \mapsto \E \phi(\lambda) + \frac{3}{2} K^2 \lambda^2$ that are convex because of~\eqref{eq:der_sec_phi} and the bounded support assumption on $P_0$.
	Therefore, for all $a \geq 1$ and $b \in (0,1/2)$ we have
	\begin{equation} \label{eq:appli_lem_convex}
		\E |\phi'(a) - \E \phi'(a)| \leq \E \phi'(a+b) - \E \phi'(a-b) + 6 K^2 b + \frac{3 v_n(s_n)}{b} \,.
	\end{equation}
	Notice that for all $a >0, \ |\E \phi'(a) | = | 2a \E \langle \bbf{x}\cdot\bbf{X} \rangle_{n,a} | \leq 2 a K^2$. Therefore, by the mean value theorem
	\begin{align*}
		\int_1^A \big(\E \phi'(a+b) - \E \phi'(a-b)\big) da 
		&= \big(\E \phi(b+A) - \E \phi(b+1)\big) - \big(\E \phi(A-b) - \E \phi(1-b)\big) \\
		&= \big(\E \phi(b+A) - \E \phi(A-b)\big) + \big(\E \phi(1-b) - \E \phi(1+b)\big) \\
		&\leq 4 K^2 b (A + 2)\,.
	\end{align*}
	Combining this with equation~\eqref{eq:appli_lem_convex}, we obtain
	\begin{equation} \label{eq:control_b}
		\forall b \in (0,1/2), \ \int_1^A \E | \phi'(a) - \E \phi'(a) | da  \leq C A \Big(b + \frac{v_n(s_n)}{b}\Big) \,.
	\end{equation}
	for some constant $C>0$ depending only on $K$.
	The minimum of the right-hand side is achieved for $b=\sqrt{v_n(s_n)} < 1/2$ for $n$ large enough. Then,~\eqref{eq:control_b} gives
	\begin{align*}
		\int_1^A \E \big| \langle U(\bbf{x}) \rangle_{n,a} - \E \langle U(\bbf{x}) \rangle_{n,a} \big| da
		= \int_1^A \E | \phi'(a) - \E \phi'(a) | da
		\leq 2 C A \sqrt{v_n(s_n)}.
	\end{align*}
\end{proof}

%% file: xx3.tex
\chapter{Low-rank symmetric matrix estimation}\label{sec:xx}

We consider in this chapter the spiked Wigner model~\eqref{eq:spiked_wigner}.
Let $P_0$ be a probability distribution on $\R$ that admits a finite second moment and consider the following observations:
\begin{equation}\label{eq:observation_xx}
	Y_{i,j} = \sqrt{\frac{\lambda}{n}} \, X_i X_j + Z_{i,j}, \quad \text{for } 1 \leq i<j \leq n \,,
\end{equation}
where $X_i \iid  P_0$ and $Z_{i,j} \iid  \mathcal{N}(0,1)$ are independent random variables. 
Note that we suppose here to only observe the coefficients of $\sqrt{\lambda/n} \bbf{X} \bbf{X}^{\sT} + \bbf{Z}$ that are above the diagonal. The case where all the coefficients are observed can be directly deduced from this case. 
In the following, $\E$ will denote the expectation with respect to the $\bbf{X}$ and $\bbf{Z}$ random variables.

Our main quantity of interest is the Minimum Mean Squared Error (MMSE) defined as:
\begin{align*}
	\MMSE_n(\lambda) 
	&= \min_{\widehat{\theta}} \frac{2}{n(n-1)} \sum_{1\leq i<j\leq n} \E\left[ \left(X_iX_j- \widehat{\theta}_{i,j}(\bbf{Y}) \right)^2\right] \label{eq:def_mmse_min_intro} \\
	&=\frac{2}{n(n-1)} \sum_{1\leq i<j\leq n} \E\left[ \left(X_iX_j-\E\left[X_iX_j|\bbf{Y} \right]\right)^2\right],
\end{align*}
where the minimum is taken over all estimators $\widehat{\theta}$ (i.e.\ measurable functions of the observations $\bbf{Y}$).
We have the trivial upper-bound
$$
\MMSE_n(\lambda) \leq \DMSE \defeq \E_{P_0}[X^2]^2 - \E_{P_0}[X]^4 \,,
$$
obtained by considering the ``dummy'' estimator $\widehat{\theta}_{i,j} = \E_{P_0} [X]^2$.
One can also compute the Mean Squared Error achieved by naive PCA.
Let $\widehat{\bbf{x}}$ be the leading eigenvector of $\bbf{Y}$ with norm $\|\widehat{\bbf{x}}\|^2 = n$.
If we take an estimator proportional to $\widehat{x}_i \widehat{x}_j$, i.e.\ $\widehat{\theta}_{i,j}=\delta \widehat{x}_i \widehat{x}_j$ for $\delta\geq 0$, we can compute explicitly (using the results from \cite{peche2006largest} presented in the introduction) the resulting $\MSE$ as a function of $\delta$ and minimize it. The optimal value for $\delta$ depends on $\lambda$, more precisely if $\lambda<\E_{P_0}[X^2]^{-2}$, then $\delta=0$ while for $\lambda\geq \E_{P_0}[X^2]^{-2}$, the optimal of value for $\delta$ is $\E_{P_0}[X^2] - \lambda^{-1} \E_{P_0}[X^2]^{-1}$, resulting in the following $\MSE$ for naive PCA:
\begin{equation} \label{eq:perfpca}
	\MSE^{\text{PCA}}_{n}(\lambda) \xrightarrow[n \to \infty]{} \left\{
		\begin{array}{ll}
			\E_{P_0}[X^2]^2 &\mbox{ if } \lambda\leq \E_{P_0}[X^2]^{-2},\\
			\lambda^{-1}\left(2-\lambda^{-1} \E_{P_0}[X^2]^{-2}\right)&\mbox{ otherwise.}
	\end{array}\right.
\end{equation}
We will see in Section~\ref{sec:interpretation_xx} that in the particular case of $P_0=\cN(0,1)$, PCA is optimal: $\lim\limits_{n\to\infty}\MSE^{\text{PCA}}_n = \lim\limits_{n \to \infty} \MMSE_n$.

\section{Information-theoretic limits}

In order to formulate our inference problem as a statistical physics problem we introduce
the random Hamiltonian 
\begin{equation}\label{eq:hamiltonian_xx}
H_n(\bbf{x}) = \sum_{i<j} \sqrt{\frac{\lambda}{n}} x_i x_j Z_{i,j} + \frac{\lambda}{n} X_i X_j x_i x_j - \frac{\lambda}{2n} x_i^2 x_j^2 \,.
\end{equation}
The posterior distribution of $\bbf{X}$ given $\bbf{Y}$ takes then the form
\begin{equation}
	\label{def:post2}
	dP(\bbf{x} \, | \, \bbf{Y}) = \frac{1}{\cZ_n(\lambda)} dP_0^{\otimes n}(\bbf{x}) \exp\Big( \sum_{i<j} x_i x_j \sqrt{\frac{\lambda}{n}} Y_{i,j} - \frac{\lambda}{2n} x_i^2 x_j^2 \Big) = \frac{1}{\cZ_n(\lambda)} dP^{\otimes n}_0(\bbf{x}) e^{H_n(\bbf{x})} \,,
\end{equation}
where $\cZ_n(\lambda)$ is the appropriate normalization.
The free energy is defined as
$$
F_n(\lambda) = \frac{1}{n} \E \Big[ \log \int dP^{\otimes n}_0(\bbf{x}) \ e^{H_n(\bbf{x})}\Big] = \frac{1}{n} \E \log \cZ_n(\lambda).
$$
We will first compute the limit of the free energy $F_n$ and then deduce the limit of $\MMSE_n$ by an I-MMSE (see Proposition~\ref{prop:i_mmse}) argument.
We express the limit of $F_n$ using the following function
\begin{equation}\label{eq:def_potential_xx}
	\mathcal{F}: (\lambda,q) \mapsto \psi_{P_0}(\lambda q) - \frac{\lambda}{4} q^2 = \E \log \left( \int dP_0(x) \exp\left(\sqrt{\lambda q}Zx + \lambda q x X - \frac{\lambda}{2}q x^2\right)\right)-\frac{\lambda}{4}q^2 ,
\end{equation}
where $Z \sim \mathcal{N}(0,1)$ and $X \sim P_0$ are independent random variables. Recall that $\psi_{P_0}$ denotes the free energy \eqref{eq:def_psi_p0} of the scalar channel~\eqref{eq:additive_scalar_channel}. The main result of this section is:
\begin{theorem}[Replica-Symmetric formula for the spiked Wigner model] \label{th:rs_formula}
	For all $\lambda > 0$,
	\begin{equation}\label{eq:rs_formula}
		F_n(\lambda) \xrightarrow[n \to \infty]{} \sup_{q \geq 0} \mathcal{F}(\lambda,q) \,.
	\end{equation}
\end{theorem}
Theorem~\ref{th:rs_formula} is proved in Section~\ref{sec:proof_rs_xx}.
In the case of Rademacher prior ($P_0 = \frac{1}{2}\delta_{-1} + \frac{1}{2} \delta_{+1}$), Theorem~\ref{th:rs_formula} was proved in~\cite{deshpande2016asymptotic}. 
The expression~\eqref{eq:rs_formula} for general priors was conjectured by~\cite{DBLP:conf/allerton/LesieurKZ15}.
For discrete priors $P_0$ for which the map $\cF(\lambda,\cdot)$ has not more than $3$ stationary points, the statement of Theorem~\ref{th:rs_formula} was obtained in~\cite{barbier2016mutual}.
The full version of Theorem~\ref{th:rs_formula} as well as its multidimensional generalization (where $\bbf{X}\in \R^{n \times k}$, $k$ fixed) was proved in~\cite{lelarge2016fundamental}.

Theorem~\ref{th:rs_formula} allows us to compute the limit of the mutual information between the signal $\bbf{X}$ and the observations $\bbf{Y}$. Indeed, by using~\eqref{eq:f_i}:
\begin{corollary} \label{cor:rs_information}
	$$
	\lim_{n \rightarrow + \infty} \frac{1}{n}  I(\bbf{X};\bbf{Y}) = \frac{\lambda \E_{P_0} (X^2)^2}{4} - \sup_{q \geq 0} \mathcal{F}(\lambda,q) \,.
	$$
\end{corollary}

We will now use Theorem~\ref{th:rs_formula} to obtain the limit of the Minimum Mean Squared Error $\MMSE_n$ by the I-MMSE relation of Proposition~\ref{prop:i_mmse}.
Let us define
$$
D = \left\{ \lambda > 0 \ \middle| \ \mathcal{F}(\lambda,\cdot) \ \text{has a unique maximizer} \ q^*(\lambda) \, \right\}.
$$
We start by computing the derivative of $\lim\limits_{n \to \infty}  F_n(\lambda)$ with respect to $\lambda$.

\begin{proposition} \label{prop:derivative_phi}
	$D$ is equal to $\R_{>0}$ minus some countable set and
	is precisely the set of $\lambda > 0$ at which the function $f: \lambda \mapsto \sup_{q \geq 0} \mathcal{F}(\lambda,q)$ is differentiable. Moreover,
	for all $\lambda \in D$
	$$
	f'(\lambda) = \frac{q^*(\lambda)^2}{4} \,.
	$$
\end{proposition}

\begin{proof}
	Let $\lambda > 0$ and compute
	$$
	\frac{\partial}{\partial q}\cF(\lambda,q) = \lambda \psi_{P_0}'(\lambda q) - \frac{\lambda q}{2} \leq \frac{\lambda}{2} \big(\E_{P_0}[X^2] - q \big) \,,
	$$
	because $\psi_{P_0}$ is $\frac{1}{2} \E_{P_0}[X^2]$-Lipschitz by Proposition~\ref{prop:i_mmse}. Consequently, the maximum of $\cF(\lambda, \cdot)$ is achieved on $[0,\E_{P_0}[X^2]]$. 
	If $q^*$ maximizes $\cF(\lambda, \cdot)$, the optimality condition gives $ q^*= 2 \psi'_{P_0}(\lambda q^*) $. Consequently
	$$
	\frac{\partial}{\partial \lambda} \mathcal{F}(\lambda,q^*) = q^* \psi'_{P_0}(\lambda q^*) - \frac{(q^*)^2}{4} = \frac{(q^*)^2}{4}\,.
	$$
	Now, Proposition~\ref{prop:envelope_compact} in Appendix~\ref{sec:appendix_envelope} gives that the $\lambda>0$ at which $f$ is differentiable is exactly the $\lambda > 0$ for which
	$$
	\left\{ \frac{\partial}{\partial \lambda} \mathcal{F}(\lambda,q^*)=\frac{1}{4}(q^*)^2  \, \middle| \, q^* \ \text{maximizer of} \ \cF(\lambda,\cdot) \right\}
	$$
	is a singleton. These $\lambda$ are precisely the elements of $D$. Moreover, Proposition~\ref{prop:envelope_compact} gives also that for all $\lambda \in D$, $f'(\lambda) = \frac{q^*(\lambda)^2}{4}$, which concludes the proof.
\end{proof}
\\

We deduce then the limit of $\MMSE_n$:
\begin{corollary} \label{cor:limit_mmse}
	For all $\lambda \in D$,
	\begin{equation}\label{eq:lim_mmse_xx}
	\MMSE_n(\lambda) \xrightarrow[n \to \infty]{} (\E_{P_0} X^2)^2 - q^*(\lambda)^2 \,.
\end{equation}
\end{corollary}

\begin{proof}
	By Proposition~\ref{prop:i_mmse}, $(F_n)_{n \geq 1}$ is a sequence of differentiable convex functions that converges pointwise on $\R_{>0}$ to $f$. By Proposition~\ref{prop:deriv_convex}, $F_n'(\lambda) \xrightarrow[n \to\infty]{} f'(\lambda)$ for every $\lambda > 0$ at which $f$ is differentiable, that is for all $\lambda \in D$. We conclude using the I-MMSE relation~\eqref{eq:i_mmse}:
	\begin{equation}
		\label{eq:rel_f_mmse}
		\frac{n-1}{4n} \big(
			\E_{P_0}[X^2]^2 - \MMSE_n(\lambda)
		\big) 
		=F_n'(\lambda) 
		\xrightarrow[n \to \infty]{} f'(\lambda) = \frac{q^*(\lambda)^2}{4}
		\,.
	\end{equation}
\end{proof}

Let us now define the information-theoretic threshold
\begin{equation}\label{eq:def_it_threshold_xx}
	\lambda_c = \inf \Big\{ \lambda \in D \, \Big| \, q^*(\lambda) > (\E_{P_0} X)^2 \Big\} \,.
\end{equation}
If the above set is empty, we define $\lambda_c = 0$.
By Corollary~\ref{cor:limit_mmse} we obtain that
\begin{itemize}
	\item if $\lambda > \lambda_c$, then $\lim\limits_{n \to \infty} \MMSE_n < \DMSE$: one can estimate the signal better than a random guess.
	\item if $\lambda < \lambda_c$, then $\lim\limits_{n \to \infty} \MMSE_n = \DMSE$: one can not estimate the signal better than a random guess.
\end{itemize}
Thus, there is no hope for reconstructing the signal below $\lambda_c$. Interestingly, one can not even detect if the measurements $\bbf{Y}$ contains some signal below $\lambda_c$. If one denotes by $Q_{\lambda}$ the distribution of $\bbf{Y}$ given by~\eqref{eq:observation_xx}, the work~\cite{alaoui2017finite} shows that for $\lambda <\lambda_c$ one can not asymptotically distinguish between $Q_{\lambda}$ and $Q_0$: both distributions are contiguous.




\section{Information-theoretic and algorithmic phase transitions}\label{sec:interpretation_xx}

\subsection{Approximate Message Passing (AMP) algorithms} \label{sec:amp_xx}

Approximate Message Passing (AMP) algorithms, introduced in~\cite{donoho2009message} for compressed sensing, have then be used for various other tasks.
Rigorous properties of AMP algorithms have been established in~\cite{bayati2011amp,javanmard2013state,bayati2015universality,berthier2017state}, following the seminal work of Bolthausen \cite{bolthausen2014iterative}.
In the context of low-rank matrix estimation an AMP algorithm has been proposed by~\cite{rangan2012iterative} for the rank-one case and then by~\cite{matsushita2013low} for finite-rank matrix estimation. 
For detailed review and developments about matrix factorization with message-passing algorithms, see~\cite{lesieur2017constrained}.
We will only give a brief description of AMP here and we let the reader refer to~\cite{rangan2012iterative,deshpande2014information,DBLP:conf/allerton/LesieurKZ15,montanari2017estimation}. 
In this section, we follow \cite{montanari2017estimation} who provides the most advanced results for our problem \eqref{eq:spiked_wigner}. For simplicity, we assume here that $P_0$ has a unit second moment: $\int x^2 dP_0(x) = 1$.
\\

Starting from an initialization $\bbf{x}^0$, the AMP algorithm produces vectors $\bbf{x}^1, \dots, \bbf{x}^t$ according to the following recursion:
\begin{equation}\label{eq:amp_xx}
	\bbf{x}^{t+1} = (\bbf{Y}/\sqrt{n}) f_t(\bbf{x}^t) - b_t f_{t-1}(\bbf{x}^{t-1}),
\end{equation}
where $b_t = \frac{1}{n} \sum_{i=1}^n f_t'(x_i^t)$ and where the functions $f_t$ act componentwise on vectors.
After $t$ iterations of \eqref{eq:amp_xx}, the AMP estimate of $\bbf{X}$ is defined by $\what{\bbf{x}}^t = f_t(\bbf{x}^t)$.

A natural choice for the initialization is to take $\bbf{x}^0$ proportional to $\bbf{\varphi}_1$, the leading unit eigenvector of $\bbf{Y}$:
$$
\bbf{x}^0 = \sqrt{n (\lambda^2 -1)} \bbf{\varphi}_1.
$$
We need now to specify the ``denoisers'' $(f_t)_{t \geq 1}$. Let us consider the following one-dimensional recursion:
\begin{equation}\label{eq:se_xx}
	\begin{cases}
		q_0 &= (1 - \lambda^{-1})_+ , \\
		q_{t+1} &= 2 \psi_{P_0}'(\lambda q_t) = 1 - \MMSE_{P_0}(\lambda q_t).
	\end{cases}
\end{equation}

Recall the additive Gaussian scalar channel from Section~\ref{sec:i_mmse}: $Y_0 = \sqrt{\gamma} X_0 + Z_0$. Let us define $g_{P_0}(y,\gamma) = \E[X_0|\sqrt{\gamma}X_0 + Z_0 = y]$.
We define then
\begin{equation}\label{eq:amp_xx_denoiser}
	f_t(x) = g_{P_0}\big(x/ \sqrt{q_t},\lambda q_t\big).
\end{equation}
The next theorem is a consequence of the more general results of \cite{montanari2017estimation}, specified to our setting.

\begin{theorem}\label{th:amp_xx}
	For all $t \geq 0$,
	$$
	\lim_{n \to \infty} \frac{ | \langle \what{\bbf{x}}^t, \bbf{X} \rangle | }{ \|\what{\bbf{x}}^t\| \| \bbf{X} \| } = \lim_{n \to \infty} \|\what{\bbf{x}}^t \| = \sqrt{q_t}.
	$$
\end{theorem}

Consequently,
\begin{equation}\label{eq:perfamp}
	\MSE^{\rm AMP}_t \defeq
	\lim_{n \to \infty} \frac{1}{n^2}
	\E \| \bbf{X}\bbf{X}^{\sT} - \what{\bbf{x}}^t (\what{\bbf{x}}^t)^{\sT} \|^2
	= 1 - q_t^2.
\end{equation}

By Proposition~\ref{prop:i_mmse}, the function $\psi_{P_0}'$ is increasing and bounded. The sequence $(q_t)_{t \geq 0}$ converges therefore to a point $q_{\infty} \geq 0$ that verifies $q_{\infty} = 2 \psi_{P_0}'(\lambda q_{\infty})$.
$q_{\infty}$ is therefore a critical point of $\cF(\lambda,\cdot)$. 
In the case where $q_{\infty}$ is the global minimizer of $\cF(\lambda,\cdot)$, i.e.\ $q_{\infty} = q^*(\lambda)$, we see using Corollary~\ref{cor:limit_mmse} that $\lim_{t \to \infty}\MSE^{\rm AMP}_t = \MMSE(\lambda)$: AMP achieves the Bayes-optimal accuracy.
\\

In the case where $q_{\infty} \neq q^*(\lambda)$, AMP does not reach the information-theoretically optimal performance.
However, AMP is conjectured (see for instance~\cite{zdeborova2016statistical,antenucci2018glassy}) to be optimal among polynomial-time algorithms, i.e.\ $\lim_{t \to \infty} \MSE^{\rm AMP}_t$ is conjectured to be the best Mean Squared Error achievable by any polynomial-time algorithm.

\subsection{Examples of phase transitions}\label{sec:examples_phase}
We give here some illustrations and interpretations of the results presented in the previous sections.
Let us first study the case where $P_0 = \cN(0,1)$ where the formulas~\eqref{eq:rs_formula} and~\eqref{eq:lim_mmse_xx} can be evaluated explicitly. Indeed, we saw in Example~\ref{ex:gaussian_psi} in Section~\ref{sec:i_mmse} that $\psi_{\cN(0,1)}(q) = \frac{1}{2} \big(q - \log(1+q) \big)$. We can then compute $q^*(\lambda) = (1 - \lambda^{-1} )_{+}$ which gives
$$
\lim_{n \to \infty} \MMSE_n(\lambda) = 
\begin{cases}
	0 & \text{if} \ \lambda \leq 1 \,, \\
	\frac{1}{\lambda}\big(2 - \frac{1}{\lambda}\big) & \text{if} \ \lambda \geq 1 \,.
\end{cases}
$$
Comparing the limit above with the performance of (naive) PCA given by~\eqref{eq:perfpca} we see that in the case $P_0 = \cN(0,1)$, PCA is information-theoretically optimal.
\\

However, as we see on~\eqref{eq:perfpca}, the MSE of PCA only depends on the second moment of $P_0$: naive PCA is not able to exploit additional properties of the signal.
We compare on Figure~\ref{fig:mmse_xx} the asymptotic performance of the naive PCA~\eqref{eq:perfpca} and the Approximate Message Passing (AMP) algorithm~\eqref{eq:perfamp} to the asymptotic Minimum Mean Squared Error for the prior
\begin{equation}\label{eq:prior_sbm}
	P_0 = 
	p \, \delta_{\!\sqrt{\frac{1-p}{p}}}
	+(1-p) \, \delta_{\!-\sqrt{\frac{p}{1-p}}} \,,
\end{equation}
where $p \in (0,1)$. This is a two-points distribution with zero mean and unit variance. 
It is of particular interest because it is related with the community detection problem in the (dense) Stochastic Block Model \cite{deshpande2014information,lelarge2016fundamental}.
\begin{figure*}[h!]
	\centering
	\includegraphics[width=0.6\linewidth]{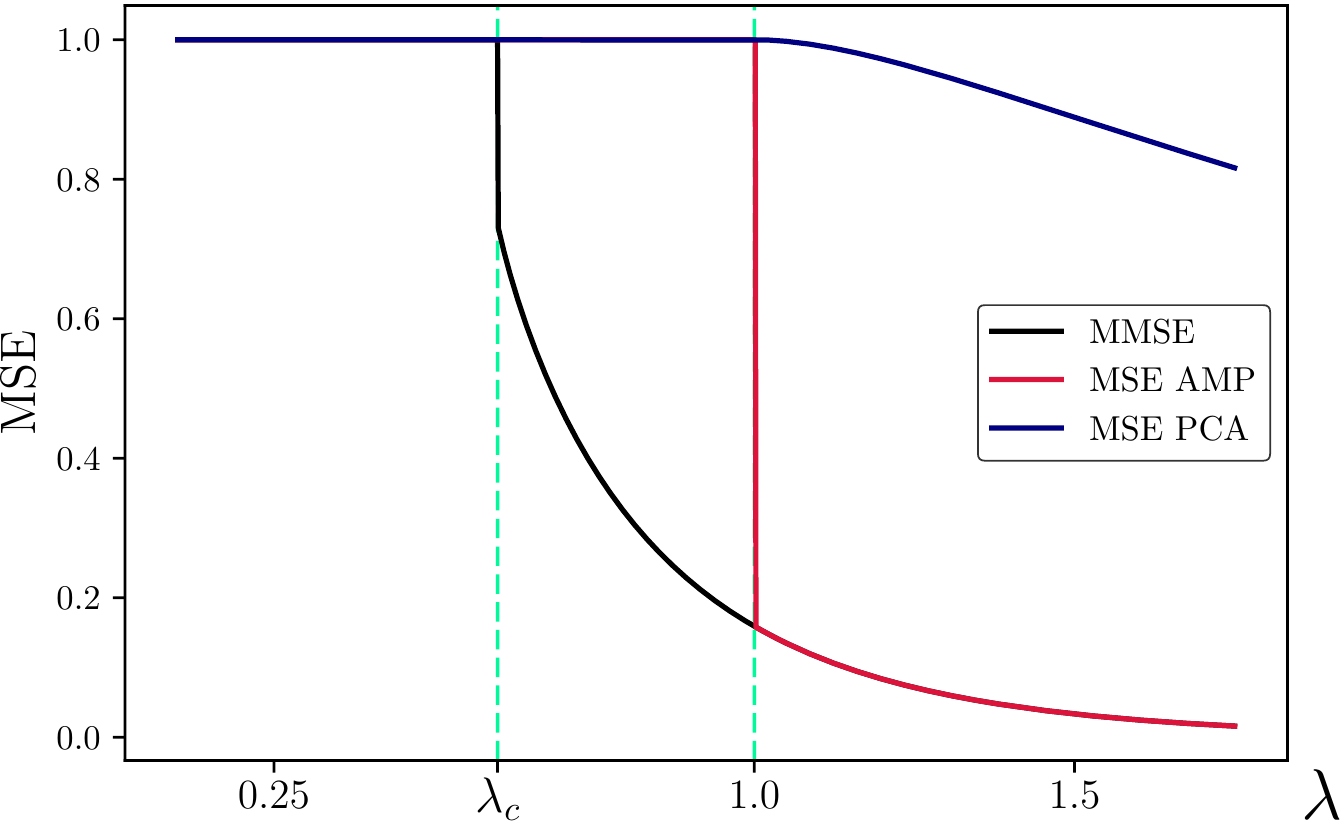}
	\caption{Mean Squared Errors for the Spiked Wigner model with prior $P_0$ given by~\eqref{eq:prior_sbm} with $p=0.05$.}
	\label{fig:mmse_xx}
\end{figure*}
We see on Figure~\ref{fig:mmse_xx} that the MMSE is equal to $1$ for $\lambda$ below the information-theoretic threshold $\lambda_c \simeq 0.6$. One can not asymptotically recover the signal better than a random guess in this region: we call this region the ``impossible'' phase.
For $\lambda > 1$ we see that spectral methods and AMP perform better than random guessing. This region is therefore called the ``easy'' phase, because non-trivial estimation is here possible using efficient algorithms. Notice also that AMP achieves the Minimum Mean Squared Error for $\lambda > 1$, as proved in~\cite{montanari2017estimation}.
The region $\lambda_c < \lambda < 1$ is more intriguing. It is still possible to build a non-trivial estimator (for instance by computing the posterior mean), but our two polynomial-time algorithms fail. This region is thus denoted as the ``hard'' phase because it is conjectured that polynomial-time algorithms can only provide trivial estimates (based on the belief that AMP is here optimal among polynomial-time algorithms).
\\

Quite surprisingly, one can guess in which phase (easy-hard-impossible) we are, simply by plotting the ``potential'' $q \mapsto -\cF(\lambda, q)$. This is done in Figure~\ref{fig:free_energy_landscape}.
\begin{figure*}[h!]
	\hspace{-1cm}
	\includegraphics[width=1.1\linewidth]{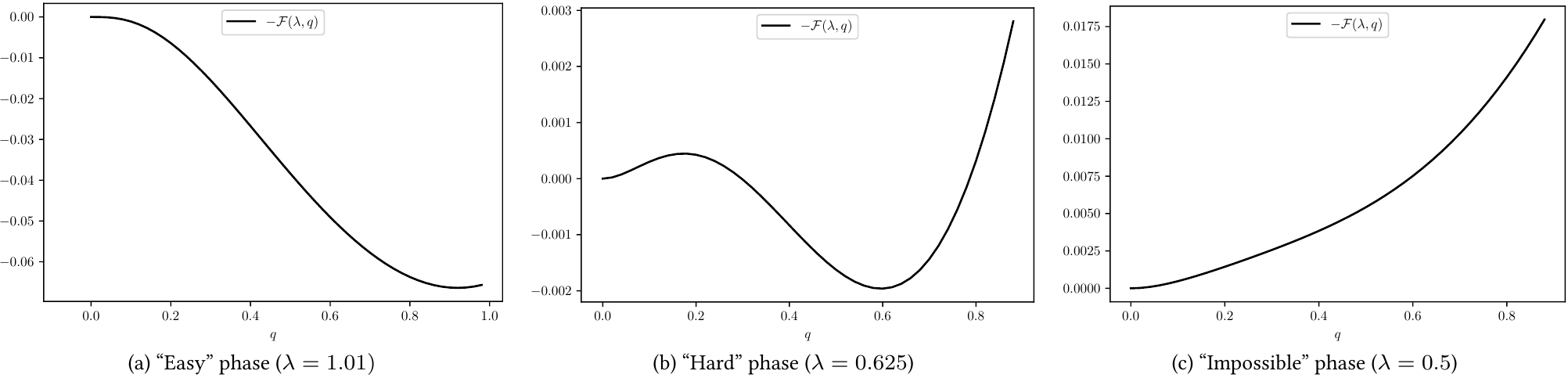}
	\caption{Plots of $q \mapsto - \cF(\lambda,q)$ for different values of $\lambda$ and $P_0$ given by~\eqref{eq:prior_sbm} with $p=0.05$.}
	\label{fig:free_energy_landscape}
\end{figure*}
By Corollary~\ref{cor:limit_mmse} we know that the limit of the MMSE is equal to $1-q^*(\lambda)^2$ where $q^*(\lambda)$ is the minimizer of $-\cF(\lambda,\cdot)$. Thus when $-\cF(\lambda,\cdot)$ is minimal at $q = 0$, we are in the impossible phase. 

When $q^*(\lambda) > 0$, the shape of $-\cF(\lambda,\cdot)$ indicates whether we are in the easy or hard phase.
If the $q=0$ is a local maximum, then we are in the easy phase, whereas when it is a local minimum we are in a hard phase.
The shape of $-\cF(\lambda,\cdot)$ could be interpreted as a simplified ``free energy landscape'': the hard phase appears when the ``informative'' minimum $q^*(\lambda) > 0$ is separated from the non-informative critical point $q=0$ by a ``free energy barrier'' as in Figure~\ref{fig:free_energy_landscape} (b).

\begin{figure}[h!]
	\centering
	\includegraphics[width=10cm]{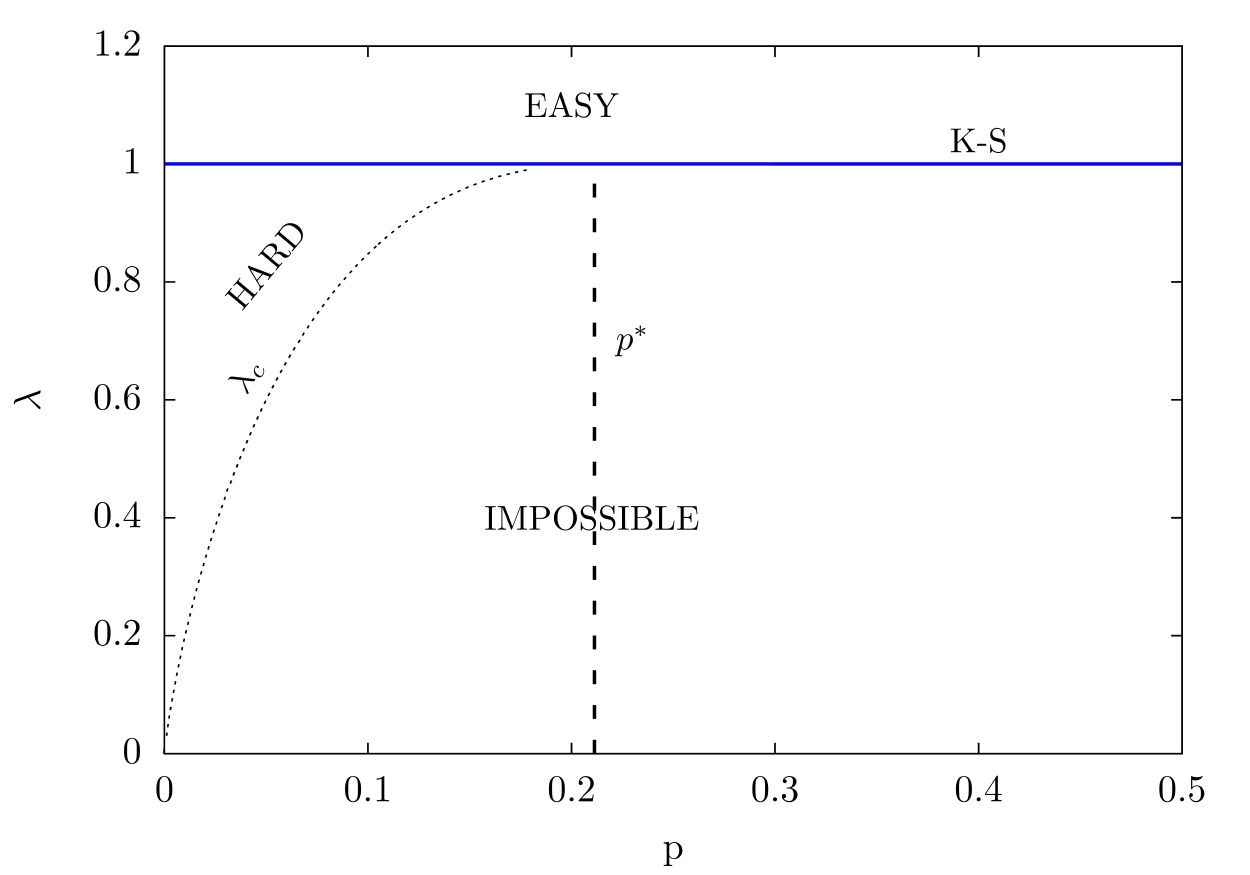}
	\caption{Phase diagram for the spiked Wigner model with prior~\eqref{eq:prior_sbm}.}
	\label{fig:phase_diagram}
\end{figure}

The phase diagram from Figure~\ref{fig:phase_diagram} displays the three phases on the $(p,\lambda)$-plane. One observes that the hard phase only appears when the prior is sufficiently asymmetric, i.e.\ for $p < p^* = \frac{1}{2} - \frac{1}{2\sqrt{3}}$, as computed in~\cite{barbier2016mutual,caltagirone2016asymmetric}. 
For a more detailed analysis of the phase transitions in the spiked Wigner model, see~\cite{lesieur2017constrained} where many other priors are considered.

\section{Proof of the Replica-Symmetric formula (Theorem~\ref{th:rs_formula})} \label{sec:proof_rs_xx}

We prove Theorem~\ref{th:rs_formula} in this section, following~\cite{lelarge2016fundamental}. We have to mention that other proofs of Theorem~\ref{th:rs_formula} have appeared since then: see~\cite{barbier2017stochastic,alaoui2018estimation,mourrat2018hamilton}.

Because of an approximation argument presented in Section~\ref{sec:approximation} it suffices to prove Theorem~\ref{th:rs_formula} for priors $P_0$ with finite (and thus bounded) support $S \subset [-K,K]$, for some $K > 0$. From now, we assume to be in that case.

\subsection{The lower bound: Guerra's interpolation method}\label{sec:guerraton}

The following result comes from~\cite{krzakala2016mutual}. It adapts arguments from the study of the gauge symmetric $p$-spin glass model of~\cite{korada2009exact} to the inference model~\eqref{eq:observation_xx}. It is based on Guerra's interpolation technique for the Sherrington-Kirkpatrick model, see~\cite{guerra2003broken}. We reproduce the proof for completeness.

\begin{proposition} \label{prop:guerra_bound}
	\begin{equation} \label{eq:guerra_bound}
		\liminf_{n \to \infty} F_n(\lambda) \geq \sup_{q \geq 0} \mathcal{F}(\lambda,q) \,.
	\end{equation}
\end{proposition}
\begin{proof}
	Let $q \geq 0$. For $t \in [0,1]$ we define
	$$
	H_{n,t}(\bbf{x}) = \sum_{i<j} \sqrt{\frac{\lambda t}{n}} Z_{i,j} x_i x_j + \frac{\lambda t}{n} x_i x_j X_i X_j - \frac{\lambda t}{2n} x_i^2 x_j^2
	+ \sum_{i=1}^n \sqrt{(1-t)\lambda q} Z_i' x_i + (1-t)\lambda q x_i X_i - \frac{(1-t) \lambda q}{2} x_i^2 \,.
	$$
	Let $\langle \cdot \rangle_{n,t}$ denote the Gibbs measure associated with the Hamiltonian $H_{n,t}(\bbf{x})$:
	$$
	\big\langle f(\bbf{x}) \big\rangle_{n,t} = \frac{ \sum_{\bbf{x} \in S^n} P_0^{\otimes n}(\bbf{x}) f(\bbf{x}) e^{H_{n,t}(\bbf{x})} }{ \sum_{\bbf{x} \in S^n} P^{\otimes n}_0(\bbf{x}) e^{H_{n,t}(\bbf{x})}} \,,
	$$
	for any function $f$ on $S^n$.
	The Gibbs measure $\langle \cdot \rangle_{n,t}$ corresponds to the distribution of $\bbf{X}$ given $\bbf{Y}$ and $\bbf{Y}'$ in the following inference channel:
	$$
	\begin{cases}
		Y_{i,j} = \sqrt{\frac{\lambda t}{n}} X_i X_j + Z_{i,j} & \ \text{for } 1 \leq i < j \leq n, \\
		\, Y'_{i} \ = \sqrt{(1-t) \lambda q} X_i + Z'_{i} & \ \text{for } 1 \leq i \leq n,
	\end{cases}
	$$
	where $X_i \iid  P_0$ and $Z_{i,j},\ Z'_i \iid  \mathcal{N}(0,1)$ are independent random variables. We will therefore be able to apply the Nishimori property (Proposition~\ref{prop:nishimori}) to the Gibbs measure $\langle \cdot \rangle_{n,t}$. Let us define
	$$
	\psi: t \in [0,1] \mapsto \frac{1}{n} \E \log \sum_{\bbf{x} \in S^n} P^{\otimes n}_0(\bbf{x}) e^{H_{n,t}(\bbf{x})} \,.
	\vspace{-0.3cm}
	$$
	We have $\psi(1)=F_n(\lambda)$ and 
	\vspace{-0.2cm}
	\begin{align*}
		\psi(0) 
		&= \frac{1}{n} \E \log \sum_{\bbf{x} \in S^n} P^{\otimes n}_0(\bbf{x}) 
		\exp\left( \sum_{i=1}^n \sqrt{\lambda q} Z_i' x_i + \lambda q x_i X_i - \frac{\lambda q}{2} x_i^2 \right) \\
		&= \frac{1}{n} \E \log \prod_{i=1}^n \left( \sum_{x_i \in S} P_0(x_i) 
	\exp\left( \sqrt{\lambda q} Z_i' x_i + \lambda q x_i X_i - \frac{\lambda q}{2} x_i^2\right) \right) \\
	&= \mathcal{F}(\lambda,q) + \frac{\lambda q^2}{4} \,.
\end{align*}
$\psi$ is continuous on $[0,1]$, differentiable on $(0,1)$. For $0<t<1$,
\begin{align} \label{eq:deriv_psi}
	\psi'(t) = \frac{1}{n} \E \left\langle 
		\sum_{i<j} \frac{\sqrt{\lambda}}{2\sqrt{nt}} Z_{i,j} x_i x_j + \frac{\lambda}{n} x_i x_j X_i X_j - \frac{\lambda}{2n} x_i^2 x_j^2
		- \sum_{i=1}^n \frac{\sqrt{\lambda q}}{2 \sqrt{1-t}} Z_i' x_i - \lambda q x_i X_i + \frac{\lambda q}{2} x_i^2
	\right\rangle_{\!\! n,t},
\end{align}
where $\bbf{x}$ is a sample from the Gibbs measure $\langle \cdot \rangle_{n,t}$, independently of everything else.
For $1 \leq i < j \leq n$ we have, by Gaussian integration by parts and by the Nishimori property
\begin{align*}
	\E \!\left[Z_{i,j} \Big\langle \frac{\sqrt{\lambda}}{2\sqrt{nt}} x_i x_j \Big\rangle_{\!\! n,t}\right]
	&=
	\frac{\lambda}{2n} \Big( \E \langle x_i^2 x_j^2 \rangle_{n,t} - \E \langle x_i x_j \rangle_{n,t}^2 \Big)
	=
	\frac{\lambda}{2n} \Big( \E \langle x_i^2 x_j^2 \rangle_{n,t} - \E \langle x^{(1)}_i x^{(1)}_j x^{(2)}_i x^{(2)}_j \rangle_{n,t} \Big)
	\\
	&=
	\frac{\lambda}{2n} \Big( \E \langle x_i^2 x_j^2 \rangle_{n,t} - \E \langle x_i x_j X_i X_j \rangle_{n,t} \Big),
\end{align*}
where $\bbf{x}^{(1)}$ and $\bbf{x}^{(2)}$ are two independent samples from the Gibbs measure $\langle \cdot \rangle_{n,t}$, independently of everything else.
Similarly, we have for $1 \leq i \leq n$
$$
\E \left\langle \frac{\sqrt{\lambda q}}{2 \sqrt{1-t}} Z_i' x_i \right\rangle_{\!\! n,t} =\, 
\frac{\lambda q}{2} \left( \E \langle x_i^2 \rangle_{n,t} - \E \langle x_i X_i \rangle_{n,t} \right)\,.
$$
Therefore~\eqref{eq:deriv_psi} simplifies
\begin{align}
	\psi'(t) &= \frac{1}{n} \E \Big\langle 
		\sum_{i<j} \frac{\lambda}{2n} x_i x_j X_i X_j 
		- \sum_{i=1}^n  \frac{\lambda q}{2} x_i X_i 
	\Big\rangle_{\!\! n,t}
	= \frac{\lambda}{4} \E \Big\langle (\bbf{x} \cdot \bbf{X})^2 - 2 q \, \bbf{x}\cdot\bbf{X} \Big\rangle_{\!\!n,t} + o_n(1) \nonumber
	\\
	&=\frac{\lambda}{4} \E \Big\langle (\bbf{x}\cdot\bbf{X} - q)^2 \Big\rangle_{\!\!n,t} - \frac{\lambda q^2}{4} + o_n(1) \geq -\frac{\lambda q^2}{4} + o_n(1)\,, \label{eq:convex_guerra}
\end{align}
where $o_n(1)$ denotes a quantity that goes to $0$ uniformly in $t \in (0,1)$. Then
$$
F_n(\lambda) - \mathcal{F}(\lambda,q) -\frac{\lambda}{4} q^2 = \psi(1) - \psi(0)
= \int_0^1 \psi'(t) dt 
\geq - \frac{\lambda}{4} q^2 + o_n(1) \,.
$$
Thus $\liminf\limits_{n \to \infty} F_n(\lambda) \geq \mathcal{F}(\lambda,q)$, for all $q \geq 0$.
\end{proof}

\subsection{Adding a small perturbation} \label{sec:small_perturbation}

It remains to prove the converse bound of~\eqref{eq:guerra_bound}. 
For this purpose, we need to show that the overlap $\bbf{x} \cdot \bbf{X}$ (where $\bbf{x}$ is a sample from the posterior distribution of $\bbf{X}$ given $\bbf{Y}$, independently of everything else) concentrates around its mean. To obtain such a result, we follow the ideas of Section~\ref{sec:overlap_concentration_montanari} that states that giving a small amount of side information to the statistician forces the overlap to concentrate, while keeping the free energy almost unchanged.

Let us fix $\epsilon\in [0,1]$, and suppose that we have access, in addition of $\bbf{Y}$, to the additional information, for $1 \leq i \leq n$
\begin{equation}\label{eq:perturbation_xx}
	Y'_i =
	\begin{cases}
		X_i &\text{if } L_i = 1, \\
		* &\text{if } L_i = 0,
	\end{cases} 
\end{equation}
where $L_i \iid  \Ber(\epsilon)$ and $*$ is a value that does not belong to $S$. 
Recall the free energy that corresponds to this perturbed inference channel is
$$
F_{n,\epsilon} = 
\frac{1}{n} \E \Big[ \log \sum_{\bbf{x} \in S^n} P^{\otimes n}_0(\bbf{x}) \exp(H_{n}(\bar{\bbf{x}}))\Big] \,,
$$
where
\begin{equation}\label{eq:def_x_bar_xx}
	\bar{\bbf{x}} = (\bar{x}_1, \dots, \bar{x}_n) = (L_1 X_1 + (1-L_1) x_1, \dots, L_n X_n + (1-L_n)x_n)\,.
\end{equation}
From now we suppose $\epsilon_0 \in (0,1]$ to be fixed and consider $\epsilon \in [0,\epsilon_0]$. We will compute the limit of $F_{n,\epsilon}$ as $n \to \infty$ and then let $\epsilon \to 0$ to deduce the limit of $F_n$, because by Proposition~\ref{prop:approximation_f_n_epsilon}
$$
|F_{n,\epsilon} - F_n | \leq H(P_0) \epsilon \,.
$$

\subsection{Aizenman-Sims-Starr scheme} \label{sec:aizenman}

The Aizenman-Sims-Starr scheme was introduced in~\cite{aizenman2003extended} in the context of the SK model.
This is what physicists call a ``cavity computation'': one compare the system with $n+1$ variables to the system with $n$ variables and see what happen to the $(n+1)^{\text{th}}$ variable we add.
\\

\noindent With the convention $F_{0,\epsilon} = 0$, we have
$F_{n,\epsilon} = \frac{1}{n} \sum\limits_{k=0}^{n-1} A_{k,\epsilon}^{(0)}$ where
\vspace{-0.2cm}
$$
A^{(0)}_{k,\epsilon} = (k+1) F_{k+1, \epsilon} - k F_{k,\epsilon} = \E [ \log(\cZ_{k+1,\epsilon}) ] - \E [ \log(\cZ_{k,\epsilon}) ] \,.
$$
We recall that $\cZ_{n,\epsilon} = \sum_{\bbf{x} \in S^n} P^{\otimes n}_0(\bbf{x}) e^{H_{n}(\bar{\bbf{x}})}$ where the notation $\bar{\bbf{x}}$ is defined by equation~\eqref{eq:def_x_bar_xx}. Consequently
\begin{equation} \label{eq:limsup1}
	\limsup_{n \to \infty} \int_0^{\epsilon_0} \! d\epsilon \, F_{n,\epsilon}
	\leq
	\limsup_{n \to \infty} \int_0^{\epsilon_0} \! d\epsilon \,A_{n,\epsilon}^{(0)} \,.
\end{equation}
We now compare $H_{n+1}$ with $H_{n}$. Let $\bbf{x} \in S^n$ and $\sigma \in S$. $\sigma$ plays the role of the $(n+1)^{\text{th}}$ variable. 
We decompose $H_{n+1}(\bbf{x},\sigma) = H_{n}'(\bbf{x}) + \sigma z_0(\bbf{x}) + \sigma^2 s_0(\bbf{x})$, where
\vspace{-0.1cm}
\begin{align*}
	H_{n}'(\bbf{x}) &= \sum_{1 \leq i<j \leq n} \sqrt{\frac{\lambda}{n+1}} Z_{i,j} x_i x_j + \frac{\lambda}{n+1} X_i X_j x_i x_j  - \frac{\lambda}{2(n+1)} x_i^2 x_j^2 \,,
	\\
	z_0(\bbf{x}) &= \sum_{i=1}^n \sqrt{\frac{\lambda}{n+1}} Z_{i,n+1} x_i  + \frac{\lambda}{n+1} X_i X_{n+1} x_i \,, \\
	s_0(\bbf{x}) &= -\frac{\lambda}{2(n+1)} \sum_{i=1}^n x_i^2  \,.
\end{align*}
Let $(\widetilde{Z}_{i,j})_{1\leq i < j \leq n}$ be independent, standard Gaussian random variables, independent of all other random variables. We have then $H_{n}(\bbf{x}) = H_{n}'(\bbf{x}) + y_0(\bbf{x})$ in law, where
$$
y_0(\bbf{x}) = \sum_{1 \leq i<j \leq n} \frac{\sqrt{\lambda}}{\sqrt{n(n+1)}} \widetilde{Z}_{i,j} x_i x_j + \frac{\lambda}{n(n+1)} X_i X_j x_i x_j - \frac{\lambda}{2(n+1)n} x_i^2 x_j^2 \,.
$$
We define the Gibbs measure $\langle \cdot \rangle_{n,\epsilon}$ by
\begin{equation} \label{eq:def_gibbs}
	\langle f(\bbf{x}) \rangle_{n,\epsilon} = \frac{1}{\cZ_{n,\epsilon}} \sum_{\bbf{x} \in S^n} P_0(\bbf{x}) f(\bar{\bbf{x}}) \exp(H_n'(\bar{\bbf{x}})) \,,
\end{equation}
for any function $f$ on $S^n$.
The Gibbs measure $\langle \cdot \rangle_{n,\epsilon}$ corresponds to the posterior distribution of $\bbf{X}$ given $(\sqrt{\lambda/(n+1)}X_i X_j + Z_{i,j})_{1 \leq i < j \leq n}$ and $\bbf{Y'}$ from~\eqref{eq:perturbation_xx}. We will therefore be able to apply the Nishimori identity (Proposition~\ref{prop:nishimori}) and Proposition~\ref{prop:overlap_concentration_montanari} to the Gibbs measure $\langle \cdot \rangle_{n,\epsilon}$.
Let us define $\bar{\sigma} = (1-L_{n+1}) \sigma + L_{n+1} X_{n+1}$. We can rewrite
$
\cZ_{n+1,\epsilon} = \sum_{\bbf{x} \in S^n} P_0^{\otimes n}(\bbf{x}) e^{H_{n}'(\bar{\bbf{x}})} \Big(\sum_{\sigma \in S} P_0(\sigma) \exp(\bar{\sigma} z_0(\bar{\bbf{x}}) + \bar{\sigma}^2 s_0(\bar{\bbf{x}}) ) \Big)
$
and 
$
\cZ_{n,\epsilon} = \sum_{\bbf{x} \in S^n} P_0^{\otimes n}(\bbf{x}) e^{H_{n}'(\bar{\bbf{x}})} e^{y_0(\bar{\bbf{x}})}
$
. Thus
\begin{align*}
A^{(0)}_{n,\epsilon} = \E \log \Big\langle  \sum_{\sigma \in S} P_0(\sigma) \exp\big(\bar{\sigma} z_0(\bbf{x}) + \bar{\sigma}^2 s_0(\bbf{x}) \big) \Big\rangle_{\!\! n,\epsilon}
- \E  \log \Big\langle \exp(y_0(\bbf{x})) \Big\rangle_{\!\! n,\epsilon}  \,.
\end{align*}
In the sequel, it will be more convenient to use slightly simplified versions of $z_0, s_0$ and $y_0$ in order to obtain nicer expressions in the sequel. We define
\begin{align*}
	z(\bbf{x}) &= \sum_{i=1}^n \sqrt{\frac{\lambda}{n}} Z_{i,n+1} x_i  + \frac{\lambda}{n} X_i X_{n+1} x_i = \sqrt{\frac{\lambda}{n}} \sum_{i=1}^n x_i Z_{i,n+1} + \lambda (\bbf{x}\cdot\bbf{X}) X_{n+1} \,,\\
	s(\bbf{x}) &= -\frac{\lambda}{2n} \sum_{i=1}^n x_i^2 = -\frac{\lambda}{2} \bbf{x}\cdot\bbf{x} \,, \\
	y(\bbf{x}) &= \frac{\sqrt{\lambda}}{\sqrt{2}n} \sum_{i=1}^n Z_i'' x_i^2 
	+\frac{\lambda}{2 n^2} \sum_{i=1}^n \left(x_i^2 X_i^2 - \frac{x_i^4}{2}\right)
	+ \frac{\sqrt{\lambda}}{n} \sum_{1 \leq i<j \leq n} x_i x_j \left(\widetilde{Z}_{i,j} 
	+ \frac{\sqrt{\lambda}}{n} X_i X_j \right) - \frac{\lambda}{2n^2} x_i^2 x_j^2
	\\
	&= \frac{\sqrt{\lambda}}{\sqrt{2}n} \sum_{i=1}^n Z_i'' x_i^2 
	+ \frac{\sqrt{\lambda}}{n} \sum_{1 \leq i<j \leq n} x_i x_j \widetilde{Z}_{i,j} 
	+ \frac{\lambda}{2} \left((\bbf{x}\cdot\bbf{X})^2 - \frac{1}{2} (\bbf{x}\cdot\bbf{x})^2\right) ,
\end{align*}
where $Z_i'' \iid  \mathcal{N}(0,1)$ independently of any other random variables. Define now
\begin{align*}
	A_{n,\epsilon} = \E \log \left\langle  \sum_{\sigma \in S} P_0(\sigma)\exp(\bar{\sigma} z(\bbf{x}) + \bar{\sigma}^2 s(\bbf{x}) ) \right\rangle_{\!\! n,\epsilon}
	- \E  \log \left\langle \exp(y(\bbf{x})) \right\rangle_{\! n,\epsilon}.
\end{align*}
Using Gaussian interpolation techniques, it is not difficult to show that $\int_0^{\epsilon_0} \! d\epsilon \, (A_{n,\epsilon} - A^{(0)}_{n,\epsilon} )\xrightarrow[n \to \infty]{} 0$
because the modifications made in $z_0, s_0$ and $y_0$ are of negligible order.
Using~\eqref{eq:limsup1} we conclude
\begin{equation} \label{eq:limsup2}
	\limsup_{n \rightarrow \infty}\int_0^{\epsilon_0} \! d\epsilon \, F_{n,\epsilon}
	\leq
	\limsup_{n \rightarrow \infty} \int_0^{\epsilon_0} \! d\epsilon \, A_{n,\epsilon} \,.
\end{equation}

\subsection{Overlap concentration}

Proposition~\ref{prop:overlap_concentration_montanari} implies that the overlap between two replicas, i.e.\ two independent samples $\bbf{x}^{(1)}$ and $\bbf{x}^{(2)}$ from the Gibbs distribution $\langle \cdot \rangle_{n,\epsilon}$, concentrates. Let us define the random variables
$$
Q = \Big\langle \frac{1}{n} \sum_{i=1}^n x^{(1)}_i x^{(2)}_i \Big\rangle_{\!\!n,\epsilon}
\qquad \text{and} \qquad
b_i = \langle x_i \rangle_{n,\epsilon} \,.
$$
Notice that $Q = \frac{1}{n} \sum_i b_i^2 \geq 0$. 
By Proposition~\ref{prop:overlap_concentration_montanari} we know that
\begin{equation}\label{eq:concentration_xx_1}
	\int_0^{\epsilon_0} d\epsilon \E \left\langle (\bbf{x}^{(1)} \cdot \bbf{x}^{(2)} - Q)^2 \right\rangle_{\!n,\epsilon} \xrightarrow[n \to \infty]{} 0 \,.
\end{equation}
Thus, using the Nishimori property (Proposition~\ref{prop:nishimori}) we deduce:
\begin{equation}\label{eq:concentration_xx_2}
	\int_0^{\epsilon_0} d\epsilon \E \left\langle (\bbf{x} \cdot \bbf{X} - Q)^2 \right\rangle_{n,\epsilon} \xrightarrow[n \to \infty]{} 0
	\qquad \text{and} \qquad
	\int_0^{\epsilon_0} d\epsilon \E \left\langle (\bbf{x} \cdot \bbf{b} - Q)^2 \right\rangle_{n,\epsilon} \xrightarrow[n \to \infty]{} 0 \,.
\end{equation}

\subsection{The main estimate}

Let us denote, for $\epsilon \in [0,1]$,
$$
\mathcal{F}_{\epsilon}: (\lambda,q) \mapsto -\frac{\lambda}{4}q^2 + \epsilon (\E_{P_0} X^2) \frac{\lambda q}{2} + (1-\epsilon)\E \Big[ \log \sum_{x \in S} P_0(x) \exp\Big(\sqrt{\lambda q}Zx + \lambda q x X - \frac{\lambda}{2}q x^2\Big)\Big]
$$
where the expectation $\E$ is taken with respect to the independent random variables $X \sim P_0$ and $Z \sim \mathcal{N}(0,1)$.
The following proposition is one of the key steps of the proof.
\begin{proposition} \label{prop:main_estimate}
	For all $\epsilon_0 \in [0,1]$,
	$$
	\int_0^{\epsilon_0}\! d\epsilon \left( A_{n,\epsilon} - \E [\mathcal{F}_{\epsilon}(\lambda,Q)]\right) \xrightarrow[n \to \infty]{} 0 \,.
	$$
\end{proposition}
The proof of Proposition~\ref{prop:main_estimate} is deferred to Section~\ref{sec:proof_main_estimate}. We deduce here Theorem~\ref{th:rs_formula} from Proposition~\ref{prop:main_estimate} and the results of the previous sections. Because of Proposition~\ref{prop:guerra_bound}, we only have to show that $\limsup\limits_{n \to \infty} F_n \leq \sup\limits_{q \geq 0} \mathcal{F}(\lambda,q)$.
\\

By Proposition~\ref{prop:approximation_f_n_epsilon} we have
$$
	\epsilon_0  F_n \leq 
	 \int_0^{\epsilon_0} \! d\epsilon F_{n,\epsilon} + \frac{1}{2}H(P_0) \epsilon_0^2 \,.
$$
Therefore by equation~\eqref{eq:limsup2} and Proposition~\ref{prop:main_estimate}
\begin{equation}
	\epsilon_0 \limsup_{n \to \infty} F_n 
	\leq 
	\limsup_{n \to \infty} \int_0^{\epsilon_0} \! d\epsilon A_{n,\epsilon} + \frac{1}{2} H(P_0) \epsilon_0^2
	\leq \limsup_{n \to \infty} \int_0^{\epsilon_0} \! d\epsilon \, \E \mathcal{F}_{\epsilon}(\lambda,Q) + \frac{1}{2} H(P_0) \epsilon_0^2 \,.
	\label{eq:main_ineq}
\end{equation}
It remains then to show that $\limsup\limits_{n \to \infty} \int d\epsilon \, \E \mathcal{F}_{\epsilon}(\lambda,Q) \leq \epsilon_0 \sup\limits_{q \geq 0} \mathcal{F}(\lambda,q) + O(\epsilon_0^2)$.
We have for $\epsilon \in [0,1]$,
\vspace{-0.2cm}
\begin{align*}
	\sup_{q \in [0,K^2]} \!\!\left| \mathcal{F}_{\epsilon}(\lambda,q) - \mathcal{F}(\lambda,q) \right| 
	&\leq \epsilon \!\!\sup_{q \in [0,K^2]} \!\!\left\{
\frac{\lambda q}{2} \E_{P_0} [X^2]  + \Big| \E \log \sum_{x \in S}\! P_0(x) \exp(\sqrt{\lambda q}Zx + \lambda q x X - \frac{\lambda}{2}q x^2)\Big| \right\}
\\
&\leq C \epsilon \,,
\end{align*}
for some constant $C$ that only depends on $\lambda$ and $P_0$. Noticing that $Q \in [0, K^2]$ a.s., we have then $| \E \mathcal{F}_{\epsilon}(\lambda,Q) - \E \mathcal{F}(\lambda,Q) | \leq C \epsilon_0$, for all $\epsilon \in [0,\epsilon_0]$ and therefore
$$
\int_0^{\epsilon_0}\!d\epsilon \, \E \mathcal{F}_{\epsilon}(\lambda,Q)
\leq \epsilon_0 \sup_{q \geq 0}\mathcal{F}(\lambda,q) + \frac{1}{2}C \epsilon_0^2 \,.
$$
Combined with~\eqref{eq:main_ineq}, this implies $\limsup\limits_{n \to \infty} F_n \leq \sup\limits_{q \geq 0} \mathcal{F}(\lambda,q) + \frac{1}{2}H(P_0) \epsilon_0 + \frac{1}{2}C \epsilon_0$, for all $\epsilon_0 \in (0,1]$. Theorem~\ref{th:rs_formula} is proved.

\subsection{Proof of Proposition~\ref{prop:main_estimate}} \label{sec:proof_main_estimate}

In this section, we prove Proposition~\ref{prop:main_estimate} which is a consequence of Lemmas~\ref{lem:part1} and~\ref{lem:part2} below.
In order to lighten the formulas, we will use the following notations
\begin{align*}
	X' = X_{n+1} \ \ \ \text{and } \ \ \
	Z'_i = Z_{i,n+1}.
\end{align*}
Recall
\begin{equation} \label{eq:a_n_proof}
	A_{n,\epsilon} = \E \log \Big\langle  \sum_{\sigma \in S} P_0(\sigma)\exp(\bar{\sigma} z(\bbf{x}) + \bar{\sigma}^2 s(\bbf{x})) \Big\rangle_{\!\! n,\epsilon}
	- \E  \log \big\langle \exp(y(\bbf{x})) \big\rangle_{\! n,\epsilon} \,,
\end{equation}
where for $\sigma \in S$, $\bar{\sigma}= (1-L_{n+1}) \sigma + L_{n+1} X'$. 
We recall that $\langle \cdot \rangle_{n,\epsilon}$ denotes the expectation with respect to $\bbf{x}$ sampled from the Gibbs measure defined by \eqref{eq:def_gibbs}.
The computations here are closely related to the cavity computations in the SK model, see for instance~\cite{talagrand2010meanfield1}.

\begin{lemma} \label{lem:part1}
	\begin{align*}
		&\int_0^{\epsilon_0} \! d\epsilon \, \Big| \E \log \Big\langle  \sum_{\sigma \in S} P_0(\sigma)\exp(\bar{\sigma} z(\bbf{x}) + \bar{\sigma}^2 s(\bbf{x}) ) \Big\rangle_{\!\! n,\epsilon}
		\\
		&- \Big( \epsilon (\E_{P_0} X^2) \E \frac{\lambda Q}{2} + (1-\epsilon) \E \log \sum_{\sigma \in S} P_0(\sigma) \exp \big( \sqrt{\lambda Q} \sigma Z_0 + \lambda Q \sigma X' - \frac{\lambda \sigma^2}{2} Q \big) \Big) \Big| \xrightarrow[n \to \infty]{} 0 \,,
	\end{align*}
	where $Z_0 \sim \mathcal{N}(0,1)$ is independent of all other random variables.
\end{lemma}

\begin{lemma} \label{lem:part2}
	$$
	\int_0^{\epsilon_0} \!\! d\epsilon \, \left| \E \log \big\langle \!\exp(y(\bbf{x})) \big\rangle_{n,\epsilon} - \frac{\lambda}{4} \E Q^2 \right| \xrightarrow[n \to  \infty]{} 0 \,.
	$$
\end{lemma}

We will only prove Lemma~\ref{lem:part1} here since Lemma~\ref{lem:part2} follows from the same kind of arguments (the full proof can be found in~\cite{lelarge2016fundamental}).
The remaining of the section is thus devoted to the proof of Lemma~\ref{lem:part1}.
\\

Let us write $f(z,s) = \sum\limits_{\sigma \in S} P_0(\sigma) e^{\bar{\sigma} z + \bar{\sigma}^2 s}$ and we define:
\begin{align*}
	U &= \big\langle f(z(\bbf{x}),s(\bbf{x})) \big\rangle_{n,\epsilon} \,, \\
	V 
	&= \sum_{\sigma \in S}P_0(\sigma) \exp\left(\bar{\sigma} \sqrt{\frac{\lambda}{n}} \sum_{i=1}^n b_i Z_i' + \lambda Q X' \bar{\sigma} - \frac{\lambda Q}{2} \bar{\sigma}^2 \right) \,.
\end{align*}

\begin{lemma}
	$$
	\int_0^{\epsilon_0} \! d\epsilon \, \E \Big[ (U-V)^2 \Big] \xrightarrow[n \to \infty]{} 0 \,.
	$$
\end{lemma}

\begin{proof}
	It suffices to show that $\int \! d\epsilon \, |\E U^2 -\E V^2| \xrightarrow[n \to \infty]{} 0$ and $\int \! d\epsilon \, |\E U V -\E V^2| \xrightarrow[n \to \infty]{} 0$. 
	\\
Let $\E_{\bbf{Z}'}$ denote the expectation with respect to $\bbf{Z}'=(Z_{i,n+1})_{1 \leq i \leq n}$ only.
Compute
	\begin{align}
		\E_{\bbf{Z'}} V^2 &= \E_{\bbf{Z'}} \sum_{\sigma_1,\sigma_2 \in S} P_0(\sigma_1,\sigma_2) \exp \Big(
			(\bar{\sigma}_1 + \bar{\sigma}_2) \sqrt{\frac{\lambda}{n}} \sum_{i=1}^n b_i Z_i' + \lambda Q X' (\bar{\sigma}_1 + \bar{\sigma}_2) - \frac{\lambda Q}{2} (\bar{\sigma}_1^2 + \bar{\sigma}_2^2) 
		\Big) \nonumber \\
		&= \sum_{\sigma_1,\sigma_2 \in S} P_0(\sigma_1,\sigma_2) \exp \Big(
		(\bar{\sigma}_1 + \bar{\sigma}_2)^2 \frac{\lambda}{2} Q + \lambda Q X' (\bar{\sigma}_1 + \bar{\sigma}_2) - \frac{\lambda Q}{2} (\bar{\sigma}_1^2 + \bar{\sigma}_2^2) 
	\Big) \nonumber \\
	&= \sum_{\sigma_1,\sigma_2 \in S} P_0(\sigma_1,\sigma_2) \exp \big(
	\bar{\sigma}_1 \bar{\sigma}_2 \lambda Q + \lambda Q X' (\bar{\sigma}_1 + \bar{\sigma}_2) 
\big) \label{eq:esp_v2}
	\end{align}
	where we write for $i=1,2$, $\bar{\sigma}_i = (1-L_{n+1})\sigma_i + L_{n+1}X'$, as before.
	\\

	Let us show that $\int \! d\epsilon \, | \E U^2 -\E V^2 | \xrightarrow[n \to \infty]{} 0$.
	\begin{align*}
		\E_{\bbf{Z'}} U^2 &= \E_{\bbf{Z'}} \big\langle f(z(\bbf{x}),s(\bbf{x})) \big\rangle_{n,\epsilon}^2 \\
							&= \E_{\bbf{Z'}} \big\langle f(z(\bbf{x}^{(1)}),s(\bbf{x}^{(1)})) f(z(\bbf{x}^{(2)}),s(\bbf{x}^{(2)})) \big\rangle_{n,\epsilon} \ \ \ \text{(} \bbf{x}^{(1)} \ \text{and } \bbf{x}^{(2)} \ \text{are indep.\ samples from } \langle \cdot \rangle_{n,\epsilon} \text{)} \\
							&= \left\langle \E_{\bbf{Z'}} f(z(\bbf{x}^{(1)}),s(\bbf{x}^{(1)})) f(z(\bbf{x}^{(2)}),s(\bbf{x}^{(2)})) \right\rangle_{\! n,\epsilon} \\
							&= \left\langle 
		\sum_{\sigma_1,\sigma_2 \in S} P_0(\sigma_1,\sigma_2) \E_{\bbf{Z'}} \exp\Big(\bar{\sigma}_1 z(\bbf{x}^{(1)}) + \bar{\sigma}_1^2 s(\bbf{x}^{(1)}) 
		+ \bar{\sigma}_2 z(\bbf{x}^{(2)}) + \bar{\sigma}_2^2 s(\bbf{x}^{(2)}) \Big)
	\right\rangle_{\!\! n,\epsilon}.
\end{align*}
The next lemma follows from the simple fact that for $N \sim \mathcal{N}(0,1)$ and $t \in \R$, $\E e^{tN} = \exp(\frac{t^2}{2})$.
\begin{lemma} \label{lem:esp_comp1}
	Let $\bbf{x}^{(1)}, \bbf{x}^{(2)} \in S^n$ and $\sigma_1, \sigma_2 \in S$ be fixed. Then
	\begin{align*}
		\E_{\bbf{Z'}} \exp\!\left( \! \sigma_1 \sqrt{\frac{\lambda}{n}} \sum_{i=1}^n x^{(1)}_i Z_i' + \sigma_2 \sqrt{\frac{\lambda}{n}} \sum_{i=1}^n x^{(2)}_i Z_i'  \right) 
		\!=
		\exp\!\left(\!
			\lambda \sigma_1 \sigma_2 \, \bbf{x}^{(1)}\!\cdot\bbf{x}^{(2)} 
			+\frac{\lambda \sigma_1^2}{2n}  \|\bbf{x}^{(1)}\|^2
			+\frac{\lambda \sigma_2^2}{2n}  \|\bbf{x}^{(2)}\|^2
		\!\right).
	\end{align*}
\end{lemma}
Thus, for all $\bbf{x}^{(1)}, \bbf{x}^{(2)} \!\in\! S^n$ and $\sigma_1, \sigma_2 \! \in \! S$
\begin{align*}
	\E_{\bbf{Z'}} e^{\bar{\sigma}_1 z(\bbf{x}^{(1)}) + \bar{\sigma}_1^2 s(\bbf{x}^{(1)})
	+ \bar{\sigma}_2 z(\bbf{x}^{(2)}) + \bar{\sigma}_2^2 s(\bbf{x}^{(2)})} 
	= e^{\lambda \bar{\sigma}_1 \bar{\sigma}_2 \bbf{x}^{(1)}\cdot\bbf{x}^{(2)} \!+ \lambda X' (\bar{\sigma}_1 (\bbf{x}^{(1)}\cdot\bbf{X}) + \bar{\sigma}_2 (\bbf{x}^{(2)}\cdot\bbf{X}))}\,,
\end{align*}
where we used the fact that $s(\bbf{x}) \!=\! -\frac{\lambda}{2n} \|\bbf{x}\|^2$ for all $\bbf{x} \!\in\! S^n$.
We have therefore
$$
\E_{\bbf{Z'}} U^2 = \left\langle \sum_{\sigma_1,\sigma_2 \in S} P_0(\sigma_1,\sigma_2) \exp\left(\lambda \bar{\sigma}_1 \bar{\sigma}_2 \bbf{x}^{(1)}\cdot\bbf{x}^{(2)} + \lambda X' \left(\bar{\sigma}_1 (\bbf{x}^{(1)}\cdot\bbf{X}) + \bar{\sigma}_2 (\bbf{x}^{(2)}\cdot\bbf{X})\right) \right) \right\rangle_{\!\! n,\epsilon}.
$$
Define
\begin{align*}
	g:(s,r_1,r_2) \in [-K^2,K^2]^3 \mapsto &
	\sum_{\sigma_1,\sigma_2 \in S}P_0(\sigma_1,\sigma_2) \exp\Big(\lambda \bar{\sigma}_1 \bar{\sigma}_2 s
	+ \lambda X'(\bar{\sigma}_1 r_1 + \bar{\sigma}_2 r_2) \Big).
\end{align*}
We have $\E_{\bbf{Z'}} U^2 = \big\langle g(\bbf{x}^{(1)}\!\cdot\bbf{x}^{(2)},\bbf{x}^{(1)}\!\cdot\bbf{X},\bbf{x}^{(2)}\!\cdot\bbf{X}) \big\rangle_{n,\epsilon}$ and by~\eqref{eq:esp_v2}, $\E_{\bbf{Z'}} V^2 = g(Q,Q,Q)$.

\begin{lemma} \label{lem:f_1_lipschitz}
	There exists a constant $M$ that only depends on $\lambda$ and $K$, such that $g$ is almost surely $M$-Lipschitz.
\end{lemma}
\begin{proof}
	$g$ is a random function that depends only on the random variables $X'$ and $L_{n+1}$ (because of $\bar{\sigma}_1$ and $\bar{\sigma}_2$). $g$ is $\mathcal{C}^1$ on the compact $[-K^2, K^2]^3$. An easy computation show that
	$$
	\forall (s,r_1,r_2) \in [-K^2,K^2]^3, \ \|\nabla g(s,r_1,r_2) \| \leq 3 \lambda K^4 \exp(3 \lambda K^4).
	$$
	$g$ is thus $M$-Lipschitz with $M = 3 \lambda K^4 \exp(3 \lambda K^4)$.
\end{proof}
\\

\noindent Using Lemma~\ref{lem:f_1_lipschitz} we obtain
\begin{align*}
	\Big\langle | g(\bbf{x}^{(1)}\!\cdot\bbf{x}^{(2)},\bbf{x}^{(1)}\!\cdot\bbf{X},\bbf{x}^{(2)}\!\cdot\bbf{X}) &- g(Q,Q,Q) | \Big\rangle_{\!\!n,\epsilon}
	\\
	&\leq 
	M \Big\langle \!\sqrt{(\bbf{x}^{(1)}\!\cdot\bbf{x}^{(2)} - Q)^2 + (\bbf{x}^{(1)}\!\cdot\bbf{X} - Q)^2 + (\bbf{x}^{(2)}\!\cdot\bbf{X} - Q)^2} \Big\rangle_{\!\!n,\epsilon}\,.
\end{align*}
We recall equation~\eqref{eq:esp_v2} to notice that $g(Q,Q,Q) = \E_{\bbf{Z'}} V^2$. Thus, using~\eqref{eq:concentration_xx_1} and~\eqref{eq:concentration_xx_2}
\begin{align*}
	\int_0^{\epsilon_0} \!\!  \! d\epsilon \, \E | \E_{\bbf{Z'}} U^2 - \E_{\bbf{Z'}} V^2 |
	\leq
	M \int_0^{\epsilon_0} \!\!  \! d\epsilon \, \E \Big\langle \sqrt{(\bbf{x}^{(1)}\cdot\bbf{x}^{(2)} - Q)^2 + (\bbf{x}^{(1)}\cdot\bbf{X} - Q)^2 + (\bbf{x}^{(2)}\cdot\bbf{X} - Q)^2 } \Big\rangle_{\!\!n,\epsilon}
	\,,
\end{align*}
and the right-hand side goes to $0$ by (\ref{eq:concentration_xx_1}-\ref{eq:concentration_xx_2}).

\vspace{0.5cm}

Showing that
$\int \! d\epsilon \, | \E U V -\E V^2 | \xrightarrow[n \to \infty]{} 0$
goes exactly the same way. We thus omit this part here for the sake of brevity, but the reader can refer to~\cite{lelarge2016fundamental} where all details are presented.
\end{proof}
\\

Using the fact that $|\log U - \log V | \leq \max(U^{-1},V^{-1}) | U - V|$ and the Cauchy-Schwarz inequality, we have
$$
\E | \log U - \log V | \leq \sqrt{\E U^{-2} + \E V^{-2}} \sqrt{\E (U - V)^2} \,.
$$

\begin{lemma} \label{lem:u_v_moment}
	There exists a constant $C$ that depends only on $\lambda$ and $K$ such that
	$$
	\E U^{-2} + \E V^{-2} \leq C \,.
	$$
\end{lemma}
\begin{proof}
	Using Jensen inequality, we have $U \geq f\big(\langle z(\bbf{x}) \rangle_{n,\epsilon}, \langle s(\bbf{x}) \rangle_{n,\epsilon}\big)$. Then
	$$
	U^{-2} \leq f\big(\langle z(\bbf{x}) \rangle_{n,\epsilon}, \langle s(\bbf{x}) \rangle_{n,\epsilon}\big)^{-2} 
	\leq \sum_{\sigma \in S} P_0(\sigma) \exp\big(-2 \bar{\sigma} \langle z(\bbf{x}) \rangle_{n,\epsilon} - 2 \bar{\sigma}^2 \langle s(\bbf{x}) \rangle_{n,\epsilon} \big) \,.
	$$
	It remains to bound $\E \exp(-2 \bar{\sigma} \langle z(\bbf{x}) \rangle_{n,\epsilon} - 2 \bar{\sigma}^2 \langle s(\bbf{x}) \rangle_{n,\epsilon})$. $P_0$ has a bounded support, therefore
	$$
	\E \exp\big(-2 \bar{\sigma} \langle z(\bbf{x}) \rangle_{n,\epsilon} - 2 \bar{\sigma}^2 \langle s(\bbf{x}) \rangle_{n,\epsilon}\big) \leq C_0 \E \exp\left(-2 \bar{\sigma} \sum_{i=1}^n \sqrt{\frac{\lambda}{n}} \langle x_i \rangle_{n,\epsilon} Z_i'\right) = C_0 \E \exp(2 \lambda Q \bar{\sigma}^2) \leq C_1 \,,
	$$
	for some constant $C_0,C_1$ depending only on $\lambda$ and $K$. Similar arguments show that $\E V^{-2}$ is upper-bounded by a constant.
\end{proof}
\\

\noindent Using the previous lemma we obtain
$
\int \! d\epsilon \, \E | \log U - \log V | \xrightarrow[n \to \infty]{} 0
$. We now compute $\E \log V$ explicitly.
\begin{lemma} \label{lem:esp_log_v1}
	$$
	\E \log V = \epsilon (\E_{P_0} X^2) \E \frac{\lambda Q}{2} + (1-\epsilon) \E \log \sum_{\sigma \in S} P_0(\sigma) \exp \left( \sigma \sqrt{\frac{\lambda}{n}} \sum_{i=1}^n b_i Z'_i + \lambda Q \sigma X' - \frac{\lambda \sigma^2}{2} Q \right).
	$$
\end{lemma}
\begin{proof}
	It suffices to distinguish the cases $L_{n+1}\!=\!0$ and $L_{n+1}\!=\!1$. If $L_{n+1}\!=\!1$ then for all $\sigma \in S$, $\bar{\sigma}=X'$ and
	\begin{align*}
		\log V &= \log \left( \exp\Big(X' \sqrt{\frac{\lambda}{n}} \sum_{i=1}^n b_i Z'_i + \lambda Q X'^2 - \frac{\lambda X'^2}{2} Q\Big) \right)
		= X' \sqrt{\frac{\lambda}{n}} \sum_{i=1}^n b_i Z'_i + \frac{\lambda X'^2}{2} Q  \,.
	\end{align*}
	$L_{n+1}$ is independent of all other random variables, thus
	$$
	\E \Big[ \1(L_{n+1}=1) \log V \Big] = \epsilon (\E_{P_0} X^2) \frac{\lambda}{2} \E Q \,,
	$$
	because the $Z_i'$ are centered, independent from $X'$ and because $X'$ is independent from $Q$. The case $L_{n+1}=0$ is obvious.
\end{proof}
\\

\noindent The variables $(b_i)_{1 \leq i \leq n}$ and $(Z_i')_{1 \leq i \leq n}$ are independent. Recall that $Q = \frac{1}{n} \sum\limits_{i=1}^n b_i^2$. Therefore, 
\vspace{-0.4cm}
$$
\left(X',Q, \frac{1}{\sqrt{n}} \sum_{i=1}^n b_i Z_i'\right) = \left(X', Q, \sqrt{Q} Z_0\right)  \ \ \text{in law, }
\vspace{-0.2cm}
$$
where $Z_0 \sim \mathcal{N}(0,1)$ is independent of $Q,X'$. The expression of $\E \log V$ from Lemma~\ref{lem:esp_log_v1} simplifies
$$
\E \log V = \epsilon (\E_{P_0} X^2) \E \frac{\lambda Q}{2} + (1-\epsilon) \E \log \sum_{\sigma \in S} P_0(\sigma) \exp \left( \sqrt{\lambda Q} \sigma Z_0 + \lambda Q \sigma X' - \frac{\lambda \sigma^2}{2} Q \right),
\vspace{-0.4cm}
$$
thus
$$
\int_0^{\epsilon_0} \!\! d\epsilon \, \Big| \E \log U - \Big( \epsilon (\E_{P_0} X^2) \E \frac{\lambda Q}{2} + (1-\epsilon) \E \log \sum_{\sigma \in S} P_0(\sigma) \exp \big( \sqrt{\lambda Q} \sigma Z_0 + \lambda Q \sigma X' - \frac{\lambda \sigma^2}{2} Q \big) \Big) \Big| \xrightarrow[n \to \infty]{} 0 \,,
$$
which is precisely the statement of Lemma~\ref{lem:part1}.

\subsection{Reduction to distribution with finite support}\label{sec:approximation}

We will show in this section that it suffices to prove Theorems~\ref{th:rs_formula} for input distribution $P_0$ with finite support.
\\

Suppose the Theorem~\ref{th:rs_formula} holds for all prior distributions over $\R$ with finite support. Let $P_0$ be a probability distribution that admits a finite second moment: $\E_{P_0} X^2 < \infty$. We are going to approach $P_0$ with distributions with finite supports.
\\
Let $0 < \epsilon \leq 1$.
Let $K > 0$ such that $\E_{P_0}[X^2 \1(|X| \geq K)] \leq \epsilon^2$.
Let $m \in \N^*$ such that $\frac{K}{m} \leq \epsilon$. For $x \in \R$ we will use the notation
$$
\bar{x} = 
\begin{cases}
	\frac{K}{m} \Big\lfloor \frac{x m}{K} \Big\rfloor & \ \text{if} \ x \in [-K,K]\,, \\
	0 & \ \text{otherwise}.
\end{cases}
$$
Consequently if $x \in [-K,K]$, $\bar{x} \leq x < \bar{x} + \frac{K}{m} \leq \bar{x} + \epsilon$. 
We define $\bar{P}_0$ the image distribution of $P_0$ through the application $x \mapsto \bar{x}$. Let $n \geq 1$. We will note $\bar{F}_n$ the free energy corresponding to the distribution $\bar{P}_0$ and $\bar{\mathcal{F}}$ the function $\mathcal{F}$ from~\eqref{eq:def_potential_xx} corresponding to the distribution $\bar{P}_0$.
$\bar{P}_0$ has a finite support, we have then by assumptions
\begin{equation} \label{eq:limit_discrete}
	\bar{F}_n(\lambda) \xrightarrow[n \to \infty]{} \sup_{q \geq 0} \bar{\mathcal{F}}(\lambda,q) \,.
\end{equation}
By construction we have for all $1 \leq i \leq n$, $\E (X_i - \bar{X}_i)^2 \leq \epsilon^2$.
Hence
$$
\E  \big\| (X_{i} X_{j})_{i<j} - (\bar{X}_{i} \bar{X}_{j})_{i<j}\big\|^2 \leq 2 (n-1)n \E_{P_0}[X^2] \epsilon^2 \,.
$$
Consequently, by ``pseudo-Lipschitz'' continuity of the free energy with respect to the Wasserstein metric (see Proposition~\ref{prop:free_wasserstein} in Appendix~\ref{sec:free_wasserstein}) there exist a constant $C > 0$ depending only on $P_0$, such that, for all $n \geq 1$ and all $\lambda \geq 0$,
\begin{equation} \label{eq:lipF}
	| F_n(\lambda) - \bar{F}_n(\lambda) | \leq \lambda C \epsilon\,.
\end{equation}

\begin{lemma} \label{lem:approx_F_discrete}
	There exists a constant $C' > 0$ that depends only on $P_0$, such that
	$$
	\Big| \sup_{q \geq 0} \mathcal{F}(\lambda,q)
	-
	\sup_{q \geq 0} \bar{\mathcal{F}}(\lambda,q) \Big| \leq \lambda C' \epsilon \,.
	$$
\end{lemma}
\begin{proof}
	First notice that both suprema are achieved over a common compact set $[0, \E_{P_0}[X^2 + \bar{X}^2]]$. Indeed, for $q \geq 0$,
	$$
	\frac{\partial}{\partial q}\cF(\lambda,q) = \lambda \psi_{P_0}'(\lambda q) - \frac{\lambda q}{2} \leq \frac{\lambda}{2} \big(\E_{P_0}[X^2] - q \big)
	$$
	because $\psi_{P_0}$ is $\frac{1}{2}\E_{P_0}[X^2]$-Lipschitz by Proposition~\ref{prop:i_mmse}. Consequently, the maximum of $\cF(\lambda, \cdot)$ is achieved on $[0,\E_{P_0}[X^2]]$ and similarly the supremum of $\bar{\mathcal{F}}(\lambda,\cdot)$ is achieved over $[0, \E_{P_0}[\bar{X}^2]]$.
	Using Proposition~\ref{prop:free_wasserstein} in Appendix~\ref{sec:free_wasserstein}, we obtain that there exists a constant $C'$ depending only on $P_0$ such that $\forall q \in [0, \E_{P_0}[X^2 + \bar{X}^2]], \ | \mathcal{F}(\lambda,q) - \bar{\mathcal{F}}(\lambda, q)| \leq \lambda C' \epsilon$. The lemma follows.
\end{proof}
\\

Combining Equation~\ref{eq:limit_discrete} and~\ref{eq:lipF} and Lemma~\ref{lem:approx_F_discrete}, we obtain that there exists $n_0 \geq 1$ such that for all $n \geq n_0$,
$$
| F_n - \sup_{q \geq 0} \mathcal{F}(\lambda,q) | \leq \lambda (C + C' +1) \epsilon \,,
$$
where $C$ and $C'$ are two constants that only depend on $P_0$. This proves Theorem~\ref{th:rs_formula}. 

%% file: uv3.tex
\chapter{Non-symmetric low-rank matrix estimation}\label{sec:uv}

We consider now the spiked Wishart model~\eqref{eq:spiked_wishart}.
Let $P_U$ and $P_V$ be two probability distributions on $\R$ with finite second moment. 
We assume that $\Var_{P_U}(U) , \Var_{P_V}(V) > 0$. 
Let $n,m \geq 1$, $\lambda > 0$  and consider $\bbf{U} = (U_1, \dots, U_n) \iid P_U$ and $\bbf{V} = (V_1,\dots, V_m) \iid P_V$, independent from each other.
Suppose that we observe
\begin{equation} \label{eq:model}
	Y_{i,j} = \sqrt{\frac{\lambda}{n}} \, U_i V_j + Z_{i,j}, \qquad \text{for} \quad 1 \leq i \leq n \quad \text{and} \quad 1 \leq j \leq m \,,
\end{equation}
where $(Z_{i,j})_{i,j}$ are i.i.d.\ standard normal random variables, independent from $\bbf{U}$ and $\bbf{V}$. In the following, $\E$ will denote the expectation with respect to the variables $(\bbf{U},\bbf{V})$ and $\bbf{Z}$. 
We define the Minimum Mean Squared Error (MMSE) for the estimation of the matrix $\bbf{U} \bbf{V}^{\sT}$ given the observation of the matrix $\bbf{Y}$:
\begin{align*}
	\MMSE_n(\lambda) 
	&= \min_{\widehat{\theta}} \left\{ \frac{1}{nm} \sum_{i=1}^n \sum_{j=1}^m \E\left[ \left(U_i V_j - \widehat{\theta}_{i,j}(\bbf{Y}) \right)^2\right] \right\} \label{eq:def_mmse_min_intro}
	=\frac{1}{nm}  \sum_{i=1}^n \sum_{j=1}^m\E \left[ \left(U_i V_j - \E\left[U_i V_j |\bbf{Y} \right]\right)^2\right],
\end{align*}
where the minimum is taken over all estimators $\widehat{\theta}$ (i.e.\ measurable functions of the observations $\bbf{Y}$).
In order to get an upper bound on the MMSE, let us consider the ``dummy estimator''
$\widehat{\theta}_{i,j} = \E[U_i V_j]$ for all $i,j$ which achieves a ``dummy'' matrix Mean Squared Error of:
$$
\DMSE =  \frac{1}{nm} \sum_{i=1}^n \sum_{j=1}^m \E \left[ \left(U_i V_j - \E [U_i V_j] \right)^2\right] 
= \E[U^2] \E[V^2] - (\E U)^2 (\E V)^2 \,.
$$

\section{Fundamental limits of estimation} \label{sec:rs_formula_uv}

As in Chapter~\ref{sec:xx}, we investigate the posterior distribution of $\bbf{U},\bbf{V}$ given $\bbf{Y}$. We define the Hamiltonian 
\begin{equation} \label{eq:hamiltonian}
	H_n(\bbf{u},\bbf{v}) = \sum_{i,j} \sqrt{\frac{\lambda}{n}} u_i v_j Z_{i,j} + \frac{\lambda}{n} u_i U_i v_j V_j - \frac{\lambda}{2n} u_i^2 v_j^2
	, \quad \text{for} \  (\bbf{u},\bbf{v}) \in \R^n \times \R^m.
\end{equation}
The posterior distribution of $(\bbf{U},\bbf{V})$ given $\bbf{Y}$ is then
\begin{equation} \label{eq:posterior}
	dP\big(\bbf{u},\bbf{v} \, \big| \, \bbf{Y} \big)
	= \frac{1}{\cZ_n(\lambda)}  e^{H_n(\bbf{u},\bbf{v})} d P_U^{\otimes n}(\bbf{u}) d P_V^{\otimes m}(\bbf{v}),
\end{equation}
where $\cZ_n(\lambda) = \!\!\!\int \!  e^{H_n(\bbf{u},\bbf{v})}d P_U^{\otimes n}(\bbf{u}) d P_V^{\otimes m}(\bbf{v})$ is the appropriate normalization. The corresponding free energy is
$$
F_n(\lambda) = \frac{1}{n} \E \log \cZ_n(\lambda) = \frac{1}{n} \E \log \left( \int
e^{H_n(\bbf{u},\bbf{v})} d P_U^{\otimes n}(\bbf{u}) d P_V^{\otimes m}(\bbf{v}) \right).
$$

We consider here the high-dimensional limit where $n,m \to \infty$, while $m / n \to \alpha > 0$.
We will be interested in the following fixed point equations, sometimes called ``state evolution equations''.

\begin{definition}
	We define the set $\Gamma(\lambda,\alpha)$ as
	\begin{equation} \label{eq:def_gamma}
		\Gamma(\lambda,\alpha) = \left\{
			(q_u,q_v) \in \R_{\geq 0}^2 \ \big| \
			q_u = 2 \psi'_{P_U}(\lambda \alpha q_v)  \ \text{and} \
			q_v = 2 \psi'_{P_V}(\lambda q_u)
		\right\}\,.
	\end{equation}
\end{definition}

First notice that $\Gamma(\lambda,\alpha)$ is not empty. 
The function $f: (q_u,q_v) \mapsto (2 \psi'_{P_U}(\lambda \alpha q_v),2 \psi'_{P_V}(\lambda q_u))$ is continuous from the convex compact set $[0, \E U^2] \times [0, \E V^2]$  into itself (see Proposition~\ref{prop:i_mmse}). Brouwer's Theorem gives the existence of a fixed point of $f$: $\Gamma(\lambda,\alpha) \neq \emptyset$.
\\

We will express the limit of $F_n$ using the following function
\begin{equation}\label{eq:def_potential_uv}
	\mathcal{F}:(\lambda,\alpha,q_u,q_v) \mapsto 
	\psi_{P_U}(\lambda \alpha q_v) + \alpha \psi_{P_V}(\lambda q_u)
	- \frac{\lambda \alpha}{2} q_u q_v \,.
\end{equation}
Recall that $\psi_{P_U}$ and $\psi_{P_V}$, defined by~\eqref{eq:def_psi_p0}, are the free energies of additive Gaussian scalar channels~\eqref{eq:additive_scalar_channel} with priors $P_U$ and $P_V$.
The Replica-Symmetric formula states that the free energy $F_n$ converges to the supremum of $\mathcal{F}$ over $\Gamma(\lambda,\alpha)$.

\begin{theorem}[Replica-Symmetric formula for the spiked Wishart model] \label{th:rs_formula_uv}
	\begin{equation} \label{eq:lim_fn}
		F_n(\lambda) \xrightarrow[n \to \infty]{} \sup_{(q_u,q_v) \in \Gamma(\lambda,\alpha)} \mathcal{F}(\lambda,\alpha,q_u,q_v) = \sup_{q_u \geq 0} \inf_{q_v \geq 0} \mathcal{F}(\lambda,\alpha,q_u,q_v) \,.
	\end{equation}
	Moreover, these extrema are achieved over the same couples $(q_u,q_v) \in \Gamma(\lambda,\alpha)$.
\end{theorem}
This result proved in~\cite{miolane2017fundamental} was conjectured by~\cite{DBLP:conf/allerton/LesieurKZ15}, in particular $\mathcal{F}$ corresponds to the ``Bethe free energy'' \cite[Equation 47]{DBLP:conf/allerton/LesieurKZ15}. 
Theorem~\ref{th:rs_formula_uv} is proved in Section~\ref{sec:proof_rs_uv}. For the rank-$k$ case (where $P_U$ and $P_V$ are probability distributions over $\R^k$), see~\cite{miolane2017fundamental}.
As in Chapter~\ref{sec:xx}, the Replica-Symmetric formula (Theorem~\ref{th:rs_formula}) allows to compute the limit of the $\MMSE$.

\begin{proposition}[Limit of the $\MMSE$] \label{prop:mmse_uv}
	Let 
	$$
	D_{\alpha} = \Big\{ \lambda > 0 \ \Big| \ \mathcal{F}(\lambda,\alpha, \cdot,\cdot) \ \text{has a unique maximizer $(q_u^*(\lambda,\alpha),q_v^*(\lambda,\alpha))$ over} \ \Gamma(\lambda,\alpha) \Big\}.
	$$
	Then $D_{\alpha}$ is equal to $(0,+\infty)$ minus a countable set and for all $\lambda \in D_{\alpha}$ (and thus almost every $\lambda>0$)
	\begin{align}
		&\MMSE_n(\lambda) \xrightarrow[n \to \infty]{} \E[U^2] \E[V^2] - q_u^*(\lambda,\alpha) q_v^*(\lambda,\alpha) \label{eq:lim_mmse} \,.
	\end{align}
\end{proposition}

Again, this was conjectured in~\cite{DBLP:conf/allerton/LesieurKZ15}: the performance of the Bayes-optimal estimator (i.e.\ the MMSE) corresponds to the fixed point of the state-evolution equations~\eqref{eq:def_gamma} which has the greatest Bethe free energy $\mathcal{F}$.
Proposition~\ref{prop:mmse_uv} follows from the same kind of arguments than Corollary~\ref{cor:limit_mmse} so we omit its proof for the sake of brevity.
\\

Proposition~\ref{prop:mmse_uv} allows to locate the information-theoretic threshold for our matrix estimation problem.
Let us define 
\begin{equation} \label{eq:lambda_c}
	\lambda_c(\alpha) = \inf \left\{\lambda \in D_{\alpha} \ | \ q_u^*(\lambda,\alpha) q_v^*(\lambda,\alpha) > (\E U)^2 (\E V)^2 \right\}.
\end{equation}
If the set of the left-hand side is empty, one defines $\lambda_c(\alpha) = 0$.
Proposition~\ref{prop:mmse_uv} gives that $\lambda_c(\alpha)$ is the information-theoretic threshold for the estimation of $\bbf{U} \bbf{V}^{\sT}$ given $\bbf{Y}$:
\begin{itemize}
	\item If $\lambda < \lambda_c(\alpha)$, then $\MMSE_n(\lambda) \xrightarrow[n \to \infty]{} \DMSE$. It is not possible to reconstruct the signal $\bbf{U} \bbf{V}^{\sT}$ better than a ``dummy'' estimator.
	\item If $\lambda > \lambda_c(\alpha)$, then $\lim\limits_{n \to \infty} \MMSE_n(\lambda) < \DMSE$. It is possible to reconstruct the signal $\bbf{U} \bbf{V}^{\sT}$ better than a ``dummy'' estimator.
\end{itemize}

Proposition~\ref{prop:mmse_uv} gives us the limit of the MMSE for the estimation of the matrix $\bbf{U} \bbf{V}^{\sT}$, but does not gives us the minimal error for the estimation of $\bbf{U}$ or $\bbf{V}$ separately. As we will see in the next section with the spiked covariance model, one can be interested in estimating $\bbf{U} \bbf{U}^{\sT}$ or $\bbf{V} \bbf{V}^{\sT}$, only. Let us define:
\begin{align*}
	\MMSE_n^{(u)}(\lambda) &= \frac{1}{n^2} \E \Big[\sum_{1 \leq i,j \leq n}\big(U_i U_j - \E[U_i U_j | \bbf{Y}]\big)^2\Big] \,,
\\
\MMSE_n^{(v)}(\lambda) &= \frac{1}{m^2} \E \Big[\sum_{1 \leq i,j \leq m}\big(V_i V_j - \E[V_i V_j | \bbf{Y}]\big)^2\Big] \,.
\end{align*}
\vspace{-4mm}
\begin{theorem}\label{th:limit_mmse_vv}
	For all $\alpha > 0$ and all $\lambda \in D_{\alpha}$
	$$
	\MMSE_n^{(u)}(\lambda) \xrightarrow[n \to \infty]{} \E_{P_U}[U^2]^2 - q_u^*(\lambda,\alpha)^2
	\quad \text{and} \quad
	\MMSE_n^{(v)}(\lambda) \xrightarrow[n \to \infty]{} \E_{P_V}[V^2]^2 - q_v^*(\lambda,\alpha)^2 \,.
	$$
\end{theorem}
Theorem~\ref{th:limit_mmse_vv} is proved in Section~\ref{sec:proof_mmse_vv}.

\section{Application to the spiked covariance model}

Let us consider now the so-called spiked covariance model. Let $\bbf{U} = (U_1,\dots,U_n) \iid P_U$, where $P_U$ is a distribution over $\R$ with finite second moment. Define the ``spiked covariance matrix''
\begin{equation}\label{eq:spiked_covariance_matrix}
\bbf{\Sigma} = \Id_n + \frac{\lambda}{n} \bbf{U} \bbf{U}^{\sT} \,,
\end{equation}
and suppose that we observe $\bbf{Y}_1,\dots, \bbf{Y}_m \iid \cN(\bbf{0},\bbf{\Sigma})$, conditionally on $\bbf{\Sigma}$. Given the matrix $\bbf{Y} = (\bbf{Y}_1 | \cdots | \bbf{Y}_m)$, one would like to estimate the ``spike'' $\bbf{U} \bbf{U}^{\sT}$. We deduce from Theorem~\ref{th:limit_mmse_vv} above the minimal mean squared error for this task, in the asymptotic regime where $n,m \to + \infty$ and $m/n \to \alpha >0$.

\begin{corollary}\label{cor:mmse_spiked_covariance}
	For all $\alpha >0$, the function
	$$
	q \mapsto \Big\{ \psi_{P_U}(\lambda \alpha q) + \frac{\alpha}{2}\big(q + \log(1-q)\big) \Big\}
	$$
	admits for almost all $\lambda>0$ a unique maximizer $q^*(\lambda,\alpha)$ on $[0, 1)$ and
	$$
\MMSE_n^{(u)}(\lambda) = \frac{1}{n^2} \E \Big[ \big\| \bbf{U}\bbf{U}^{\sT} - \E[\bbf{U}\bbf{U}^{\sT} | \bbf{Y}] \big\|^2\Big] 
\xrightarrow[n \to \infty]{} \E_{P_U}[U^2]^2 - \left(\frac{q^*(\lambda,\alpha)}{\lambda (1-q^*(\lambda,\alpha))}\right)^2 \,.
	$$
\end{corollary}
\begin{proof}
	There exists independent Gaussian random variables $\bbf{V} = (V_1, \dots, V_m) \iid \cN(0,1)$ and $Z_{i,j} \iid \cN(0,1)$, independent from $\bbf{U}$ such that
	$$
	\bbf{Y} = (\bbf{Y}_1 | \cdots | \bbf{Y}_m)=
	\sqrt{\frac{\lambda}{n}} \bbf{U} \bbf{V}^{\sT} + \bbf{Z} \,.
	$$
	Therefore, the limit of the MMSE for the estimation of $\bbf{U} \bbf{U}^{\sT}$ is given by Theorem~\ref{th:limit_mmse_vv} above. 
	It remains only to evaluate the formulas of Theorems~\ref{th:rs_formula_uv} and~\ref{th:limit_mmse_vv} in the case $P_V = \cN(0,1)$. 
	As computed in Example~\ref{ex:gaussian_psi}, $\psi_{\cN(0,1)}(q) = \frac{1}{2} (q -\log (1+q))$. Thus, the limit of the free energy~\eqref{eq:lim_fn} becomes (after evaluation of the supremum in $q_u$):
	$$
	\sup_{q_v \in [0,1)}\Big\{ \psi_{P_U}(\lambda \alpha q_v) + \frac{\alpha}{2}\big(q_v + \log(1-q_v)\big) \Big\} \,.
	$$
	By Theorem~\ref{th:limit_mmse_vv} for all $\alpha >0$ and almost all $\lambda >0$ this supremum admits a unique maximizer $q_v^*(\lambda,\alpha)$ and $\MMSE_n^{(u)}(\lambda) \to \E_{P_U}[U^2]^2 - q^*_u(\lambda,\alpha)^2$ where $q_u^*$ verifies (recall that $(q_u^*,q_v^*) \in \Gamma$):
	$$
	q^*_v(\lambda,\alpha) = 2 \psi_{\cN(0,1)}'(\lambda q_u^*(\lambda,\alpha)) = \frac{\lambda q_u^*(\lambda,\alpha)}{1 + \lambda q_u^*(\lambda,\alpha)} \,.
	$$
	We deduce from the equation above that $q_u^*(\lambda,\alpha) = \frac{q_v^*(\lambda,\alpha)}{\lambda (1 - q_v^*(\lambda,\alpha))}$, which concludes the proof.
\end{proof}
\\

We will now compare the MMSE given by Corollary~\ref{cor:mmse_spiked_covariance} to the mean squared errors achieved by PCA and Approximate Message Passing (AMP). 
\begin{figure*}[h!]
	\centering
	\includegraphics[width=1.0\linewidth]{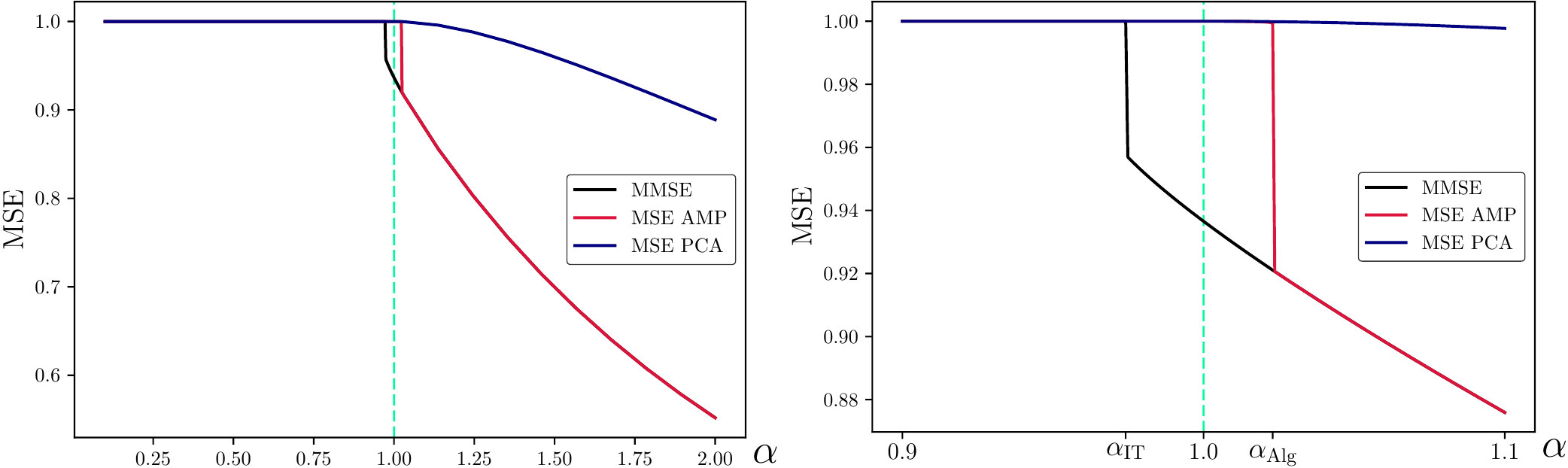}
	\caption{Mean Squared Errors for the spiked covariance model, where the spike is generated by~\eqref{eq:spiked_covariance_prior} with $s=0.15$, $\lambda = 1$. The right-hand side panel is a zoom of the left-hand side panel around $\alpha = 1$.}
	\label{fig:mmse_spiked_covariance}
\end{figure*}

Let $\widehat{\bbf{u}}$ be a singular vector of $\bbf{Y}/\sqrt{n}$ associated with $\sigma_1$, the top singular value of $\bbf{Y} / \sqrt{n}$, such that $\|\widehat{\bbf{u}}\| = \sqrt{n}$. Then results from~\cite{benaych2012singular,dobriban2016pca} give that almost surely:
$$
\lim_{n \to \infty} \big(\widehat{\bbf{u}} \cdot \bbf{U}\big)^2 
\ = \
\begin{cases}
	\frac{\lambda^2 \alpha - 1}{\lambda (\lambda \alpha + 1)} & \text{if} \ \lambda^2 \alpha \geq 1\,, \\
	0 & \text{otherwise,}
\end{cases}
\quad \text{and} \quad
\lim_{n \to \infty} \sigma_1
\ = \
\begin{cases}
	\sqrt{\frac{(1+\lambda)(\alpha^{-1} + \lambda)}{\lambda}} & \text{if} \ \lambda^2 \alpha \geq 1\,, \\
	1 + 1 / \sqrt{\alpha} & \text{otherwise.}
\end{cases}
$$
We are then going to estimate $\bbf{U} \bbf{U}^{\sT}$ using $\widehat{\theta}^{\rm PCA} = \delta \widehat{\bbf{u}} \widehat{\bbf{u}}^{\sT}$, where $\delta$ is chosen in order to minimize the mean squared error. 
The optimal choice of $\delta$ is $\delta^* = \big( \frac{\lambda^2 \alpha - 1}{\lambda (\lambda \alpha + 1)} \big)_{\! +}$, which can be estimated using $\sigma_1$.
We obtain the mean squared error of the spectral estimator $\widehat{\theta}^{\rm PCA}$:
$$
\lim_{n \to \infty} \MSE^{\rm PCA}_n 
\ = \
\begin{cases}
\frac{1+\lambda}{\lambda (\lambda \alpha + 1)} \Big(2 - \frac{1+\lambda}{\lambda (\lambda \alpha + 1)}\Big) & \text{if} \ \lambda^2 \alpha \geq 1\,, \\
	1 & \text{otherwise.}
\end{cases}
$$
As in the symmetric case (see Section~\ref{sec:amp_xx}) one can define an Approximate Message Passing (AMP) algorithm to estimate $\bbf{U} \bbf{U}^{\sT}$.
For a precise description of the algorithm, see \cite{rangan2012iterative,deshpande2014sparse,DBLP:conf/isit/LesieurKZ15}.
The MSE achieved by AMP after $t$ iterations is:
$$
\lim_{n \to \infty} \MSE^{\rm AMP}_n = 1 - \big(q_u^t\big)^2 \,,
$$
where $q_u^t$ is given by the recursion:
\begin{equation} \label{eq:state_evolution_uv}
	\begin{cases}
		q_u^{t} = 2 \psi'_{P_U}(\lambda \alpha q_v^t) \\
		q_v^{t+1} = 2 \psi'_{P_V}(\lambda q_u^t) \,,
	\end{cases}
\end{equation}
with initialization $(q_u^0, q_v^0) = (0,0)$. 
We know by Proposition~\ref{prop:i_mmse} that the functions $\psi_{P_U}'$ and $\psi_{P_V}'$ are both non-decreasing and bounded. This ensures that $(q_u^t,q_v^t)_{t \geq 0}$ converges as $t \to \infty$  to some fixed point $(q_u^{\rm AMP},q_v^{\rm AMP}) \in \Gamma$. If this fixed point turns out to be the one that maximizes $\mathcal{F}(\lambda,\alpha, \cdot, \cdot)$, i.e.\ that $(q_u^{\rm AMP},q_v^{\rm AMP}) = (q_u^*(\lambda,\alpha), q_v^*(\lambda,\alpha))$, then AMP achieves the minimal mean squared error!
\\

For the plots of Figure~\ref{fig:mmse_spiked_covariance}, we consider a case where the signal is sparse:
\begin{equation}\label{eq:spiked_covariance_prior}
	P_U = s \cN(0,1/s) + (1-s) \delta_0 \,,
\end{equation}
for some $s \in (0,1]$, so that $\E_{P_U}[U^2] = 1$.  
We plot the different MSE on Figure~\ref{fig:mmse_spiked_covariance}. We chose $\lambda = 1$ so the ``spectral threshold'' (the minimal value of $\alpha$ for which PCA performs better than a random guess) it at $\alpha = 1$ (green dashed line). This threshold corresponds also to the threshold for AMP: $\MSE^{\rm  AMP} = 1$ for $\alpha < 1$ while $\MSE^{\rm AMP} < 1$ for $\alpha >1$. The information-theoretic threshold $\alpha_{\rm IT}$ is however strictly less than $1$. For $\alpha \in (\alpha_{\rm IT},1)$ inference is ``hard'': it is information-theoretically possible to achieve a $\MSE$ strictly less than $1$, but PCA and AMP fail (and it is conjectured that any polynomial-time algorithm will also fail).

However, even for $\alpha > 1$, AMP does not always succeed to reach the MMSE. For $\alpha \in (1,\alpha_{\rm Alg})$, $\MSE^{\rm AMP}$ is strictly less than $1$ but is still very bad. So, the region $\alpha \in (1,\alpha_{\rm Alg})$ is also a ``hard region'' in the sense that achieving the $\MMSE$ seems impossible for polynomial-time algorithms (under the conjecture that AMP is optimal among polynomial-time algorithms). The scenario presented on Figure~\ref{fig:mmse_spiked_covariance} is not the only one possible: various cases have been studied in great details in~\cite{lesieur2017constrained}. See in particular Figure~6 from~\cite{lesieur2017constrained} and the phase diagrams of Figure~7 and~8.

\section{Proof of the Replica-Symmetric formula (Theorem~\ref{th:rs_formula_uv})}\label{sec:proof_rs_uv}

\subsection{Proof ideas} \label{sec:proof_techniques}

The proof of the Replica formula for the non-symmetric case is a little bit more involved compared to the symmetric case, because one can not use the convexity argument of Proposition~\ref{prop:guerra_bound} to obtain the lower bound. 
Indeed, a key step in the proof of Proposition~\ref{prop:guerra_bound} was the inequality~\eqref{eq:convex_guerra} that was obtained by saying that for every $q \geq 0$ 
\begin{equation} \label{eq:pos_xx}
	\E  \big[ ( \bbf{x} \cdot \bbf{X} - q )^2 \big] \geq 0 \,,
\end{equation}
where $\bbf{x}$ is a sample from the posterior distribution of $\bbf{X}$ given some observations (we omit the notation's details here in order to focus on the main ideas). 

However, if we apply the strategy of Proposition~\ref{prop:guerra_bound} to the non-symmetric case, one obtains
\begin{equation}\label{eq:non_pos_uv}
	\E \big[ (\bbf{u} \cdot \bbf{U} - q_u) ( \bbf{v} \cdot \bbf{V} - q_v) \big]
\end{equation}
where $(\bbf{u},\bbf{v})$ is a sample from the posterior distribution of $(\bbf{U},\bbf{V})$ given some observations,
instead of~\eqref{eq:pos_xx}. Now, it not obvious anymore that~\eqref{eq:non_pos_uv} is non-negative. In order to prove it, one has to investigate further the distributions of the overlaps $\bbf{u} \cdot \bbf{U}$ and $\bbf{v} \cdot \bbf{V}$. By following the approach used by Talagrand in~\cite{talagrand2010meanfield1} to prove the TAP equations (discovered by Thouless, Anderson and Palmer in~\cite{thouless1977solution}) for the Sherrington-Kirkpatrick model, one can show that the overlaps approximately satisfy (when $n$ and $m$ are large)
$$
\begin{cases}
	\bbf{u} \cdot \bbf{U} \, \simeq \, 2 \psi'_{P_U}(\lambda \alpha \bbf{v} \cdot \bbf{V}) \\
	\bbf{v} \cdot \bbf{V} \, \simeq \, 2 \psi'_{P_V}(\lambda \bbf{u} \cdot \bbf{U}) \,.
\end{cases}
$$
These are precisely the fixed point equations verified by $(q_u,q_v) \in \Gamma(\lambda,\alpha)$. Thus one has
\begin{equation}
	\E \big[ (\bbf{u} \cdot \bbf{U} - q_u) ( \bbf{v} \cdot \bbf{V} - q_v) \big]
	\simeq
	\E \big[ (2 \psi'_{P_U}(\lambda \alpha \bbf{v} \cdot \bbf{V}) - 2 \psi'_{P_U}(\lambda \alpha q_v)) ( \bbf{v} \cdot \bbf{V} - q_v) \big] \geq 0 \,,
\end{equation}
because by Proposition~\ref{prop:i_mmse}, $\psi'_{P_U}$ is non-decreasing. One obtain thus the analog of the lower-bound of Proposition~\ref{prop:guerra_bound} for the non-symmetric case. The converse upper-bound is proved following the Aizenman-Sims-Starr scheme, as in the symmetric case.
\\

In the following sections we will not, however, follow the proof strategy that we just described. This was done in~\cite{miolane2017fundamental}. We will instead provide a more straightforward proof from~\cite{barbier2017layered} that uses an evolution of Guerra's interpolation technique, see~\cite{barbier2017stochastic}.

\subsection{Interpolating inference model}

We prove Theorem~\ref{th:rs_formula_uv} in this section. First, notice that is suffices to prove Theorem~\ref{th:rs_formula_uv} for $\lambda = 1$, because the dependency in $\lambda$ can be ``incorporated'' in the prior $P_U$. We will thus consider in this section that $\lambda=1$ and consequently alleviate the notations by removing the dependencies in $\lambda$.
Second, it suffices to prove that 
\begin{equation}\label{eq:goal_proof_uv}
	F_n \xrightarrow[n \to \infty]{} \sup_{q_u \geq 0} \inf_{q_v \geq 0} \cF(\alpha,q_u,q_v)
\end{equation}
because the equality with $\sup_{(q_u,q_v)\in \Gamma(\lambda,\alpha)} \cF(\alpha,q_u,q_v)$ follows then from simple convex analysis arguments (Proposition~\ref{prop:sup_inf}) presented in Appendix~\ref{sec:app_convex}.

Third, by a straightforward adaptation of the approximation argument of Section~\ref{sec:approximation} to the non-symmetric case, it suffices to prove~\eqref{eq:goal_proof_uv} in the case where the priors $P_U$ and $P_V$ have bounded supports included in $[-K,K]$ for some $K >0$.
We suppose now that the above conditions are verified and we will show that~\eqref{eq:goal_proof_uv} holds.
\\

Let $q_1,q_2: [0,1] \to \R_{\geq 0}$ be two differentiable functions. For $0 \leq t \leq 1$ we consider the following observation channel
\begin{equation}\label{eq:interpolation_uv}
	\left\{
		\begin{array}{llcll}
			\bbf{Y}_t &=& \sqrt{(1-t) / n} \, \bbf{U} \bbf{V}^{\sT} &+& \bbf{Z} \\
			\bbf{Y}^{(u)}_t &=& \sqrt{\alpha q_1(t)} \, \bbf{U} &+& \bbf{Z}^{(u)} \\
			\bbf{Y}^{(v)}_t &=& \sqrt{q_2(t)} \, \bbf{V} &+& \bbf{Z}^{(v)} \,,
		\end{array}
	\right.
\end{equation}
where $Z^{(u)}_i,Z^{(v)}_j \iid \cN(0,1)$, are independent from everything else.
The observation channel~\eqref{eq:interpolation_uv} interpolates between 
the initial matrix estimation problem~\eqref{eq:model} ($t=0$, provided that $q_1(0)$ and $q_2(0)$ are small), and two decoupled inference channels on $\bbf{U}$ and $\bbf{V}$ ($t=1$).
For $r_1, r_2 \geq 0$, we define the Hamiltonian:
\begin{align*}
	H_{n,t}(\bbf{u},\bbf{v};r_1,r_2) = 
	&\sum_{i,j} \sqrt{\frac{(1-t)}{n}} u_i v_j Z_{i,j} + \frac{(1-t)}{n} u_i v_j U_i V_j - \frac{(1-t)}{2n} u_i^2 v_j^2
	\\
	+& \sum_{i=1}^n \sqrt{\alpha r_1} u_i Z_i^{(u)} + \alpha r_1 u_i U_i - \frac{\alpha r_1}{2} u_i^2 
	 + \sum_{j=1}^m \sqrt{r_2} \, v_j Z_j^{(v)} + r_2 \, v_j V_j - \frac{r_2}{2} \, v_j^2 \,.
\end{align*}
The posterior distribution of $(\bbf{U},\bbf{V})$ given $(\bbf{Y}_t,\bbf{Y}^{(u)}_t,\bbf{Y}^{(v)}_t)$ is then
\begin{equation} \label{eq:interpolation_posterior_uv}
	dP\big(\bbf{u},\bbf{v} \, \big| \, \bbf{Y}_t,\bbf{Y}^{(u)}_t,\bbf{Y}^{(v)}_t \big)
	= \frac{1}{\cZ_{n,t}}  e^{H_{n,t}(\bbf{u},\bbf{v};q_1(t),q_2(t))}d P_U^{\otimes n}(\bbf{u}) d P_V^{\otimes m}(\bbf{v}) \,,
\end{equation}
where $\cZ_{n,t}$ is the appropriate normalization. We will often drop the dependencies in $q_1(t),q_2(t)$ and write simply $H_{n,t}(\bbf{u},\bbf{v})$. The Gibbs bracket $\langle \cdot \rangle_{n,t}$ denotes the expectation with respect to samples $(\bbf{u},\bbf{v})$ from the posterior~\eqref{eq:interpolation_posterior_uv}:
\begin{equation} \label{eq:interpolation_gibbs_uv}
	\big\langle f(\bbf{u},\bbf{v}) \big\rangle_{n,t} = \frac{1}{\cZ_{n,t}} \int
	f(\bbf{u},\bbf{v}) e^{H_{n,t}(\bbf{u},\bbf{v};q_1(t),q_2(t))} dP_U^{\otimes n}(\bbf{u}) dP_V^{\otimes m}(\bbf{v}) \,,
\end{equation}
for all function $f$ for which the right-hand side is well defined.
The corresponding free energy is then
\begin{equation} \label{eq:interpolation_free_uv}
	f_n(t) = \frac{1}{n} \E \log \cZ_{n,t} = \frac{1}{n} \E \log \left( \int
e^{H_{n,t}(\bbf{u},\bbf{v})} d P_U^{\otimes n}(\bbf{u}) d P_V^{\otimes m}(\bbf{v})\right).
\end{equation}
Notice that
\begin{equation}
	\left\{
		\begin{array}{ccl}
			f_n(0) &=& F_n + O(q_1(0) + q_2(0)) \\
			f_n(1) &=& \psi_{P_U}(\alpha q_1(1)) + \frac{m}{n} \psi_{P_V}\left(q_2(1) \right) \,.
		\end{array}
	\right.
\end{equation}
$f_n(1)$ looks similar to the limiting expression $\mathcal{F}$ defined by~\eqref{eq:def_potential_uv}. We would therefore like to compare $f_n(1)$ and $F_n = f_n(0) + O(q_1(0) + q_2(0))$. We thus compute the derivative of $f_n$:

\begin{lemma}\label{lem:der_f_interpolation_uv}
	For all $t \in (0,1)$,
	\begin{equation}\label{eq:der_f_interpolation_uv}
		f_n'(t)
		=
		\frac{\alpha}{2} q_1'(t) q_2'(t) - \frac{1}{2} \E \left\langle
			\Big( \bbf{u} \cdot \bbf{U} - q'_2(t) \Big)
			\Big( \frac{m}{n} \bbf{v} \cdot \bbf{V} - \alpha q_1'(t) \Big)
		\right\rangle_{\! n,t} \,.
	\end{equation}
\end{lemma}
\begin{proof}
	Let $t \in (0,1)$. Compute
	$$
	f_n'(t)
	=
	\frac{1}{n}
	\E \left\langle
		\frac{\partial}{\partial t} H_{n,t}(\bbf{u},\bbf{v})
	\right\rangle_{\! n,t} \,.
	$$
	Using Gaussian integration by parts and the Nishimori property (Proposition~\ref{prop:nishimori}) as in the proof of Proposition~\ref{prop:i_mmse}, one obtains:
	$$
	\frac{1}{n}
	\E \left\langle
		\frac{\partial}{\partial t} H_{n,t}(\bbf{u},\bbf{v})
	\right\rangle_{\! n,t}
	=
	\frac{1}{2}\alpha q_1'(t)
	\E \left\langle
		\bbf{u}\cdot\bbf{U}
	\right\rangle_{\! n,t}
	+
	\frac{1}{2} q'_2(t)
	\E \left\langle
		\frac{m}{n}\bbf{v}\cdot\bbf{V}
	\right\rangle_{\! n,t}
	-
	\frac{1}{2}
	\E \left\langle
		(\bbf{u}\cdot\bbf{U}) (\frac{m}{n}\bbf{v}\cdot\bbf{V})
	\right\rangle_{\! n,t} \,,
	$$
	which leads to~\eqref{eq:der_f_interpolation_uv}.
\end{proof}
\\

Our goal now is to show that the expectation of the Gibbs measure in~\eqref{eq:der_f_interpolation_uv} vanishes. If this is the case, the relation $F_n \simeq f_n(0) = f_n(1) - \int_0^1 f_n'(t) dt$ would give us almost the formula that we want to prove.
The arguments can be summarized as follows:
\begin{itemize}
	\item First, we show that the overlap $\bbf{u} \cdot \bbf{U}$ concentrates around its mean $\E \langle \bbf{u} \cdot \bbf{U} \rangle_{n,t}$.
	\item Then, we chose $q_2$ to be solution of the differential equation $q'_2(t) = \E \langle \bbf{u} \cdot \bbf{U} \rangle_{n,t}$ in order to cancel the Gibbs average in~\eqref{eq:der_f_interpolation_uv}.
\end{itemize}

\subsection{Overlap concentration} \label{sec:overlap_concentration_uv}

Following the ideas of Section~\ref{sec:overlap_concentration_gg}, we show here that the overlap $\bbf{u} \cdot \bbf{U}$ concentrates around its mean, on average over small perturbations of our observation model.
\begin{proposition}\label{prop:overlap_uv}
	Let $R_1, R_2 : [0,1] \times (0,+\infty)^2 \to \R_{\geq 0}$ be two continuous, bounded functions that admits partial derivatives with respect to their second and third arguments, that are continuous and non-negative.
	Let $s_n = n^{-1/32}$.
	For $\epsilon \in [1,2]^2$, we let $q_1(\cdot,\epsilon),q_2(\cdot,\epsilon)$ be the unique solution of
	\begin{equation}\label{eq:cauchy}
	\begin{cases}
		q_1(0) = s_n \epsilon_1 \\
		q_2(0) = s_n \epsilon_2 
	\end{cases}
	\begin{cases}
		q_1'(t) = R_1(t,q_1(t),q_2(t)) \\
		q_2'(t) = R_2(t,q_1(t),q_2(t)).
	\end{cases}
	\end{equation}
	Then there exists a constant $C>0$ that only depends on $K$, $\alpha$, $\|R_1\|_{\infty}$ and $\|R_2\|_{\infty}$, such that for all $t \in [0,1]$,
	$$
	\int_1^2 \int_1^2 \E \Big\langle \big(\bbf{u} \cdot \bbf{U} - \E \langle \bbf{u} \cdot \bbf{U} \rangle_{n,t}\big)^2 \Big\rangle_{n,t} d\epsilon_1 d\epsilon_2
	\leq \frac{C}{n^{1/8}},
	$$
	where $\langle \cdot \rangle_{n,t}$ is the Gibbs measure \eqref{eq:interpolation_gibbs_uv} with $(q_1,q_2) = (q_1(\cdot, \epsilon), q_2(\cdot, \epsilon))$.
\end{proposition}
\begin{proof}
	The existence and uniqueness of the solution of the Cauchy problem \eqref{eq:cauchy} comes from the usual Cauchy-Lipschitz theorem (see for instance Theorem~3.1 in Chapter~V from \cite{hartmanordinary}). Let us fix $t \in [0,1]$
	The flow
	$$
	Q_t: \epsilon \mapsto (q_1(t,\epsilon), q_2(t,\epsilon))
	$$
	of \eqref{eq:cauchy} is a $\cC^1$-diffeomorphism. Its Jacobian is given by the Liouville formula (see for instance Corollary~3.1 in Chapter~V from \cite{hartmanordinary}):
	\begin{equation}\label{eq:ineq_j}
	J(\epsilon) \defeq {\rm det}\Big(\frac{\partial Q_t}{\partial \epsilon}(\epsilon) \Big)
	=s_n^2 \exp \Big( 
		\int_0^t \frac{\partial R_1}{\partial q_1}(s,Q_s(\epsilon)) ds 
		+
		\int_0^t \frac{\partial R_2}{\partial q_2}(s,Q_s(\epsilon)) ds 
	\Big)
	\geq s_n^2,
\end{equation}
	because the partial derivatives inside the exponential are both non-negative.
	The quantity
	$$
\E \Big\langle \big(\bbf{u} \cdot \bbf{U} - \E \langle \bbf{u} \cdot \bbf{U} \rangle_{n,t}\big)^2 \Big\rangle_{n,t}
	$$
	is a function of the signal-to-noise ratios $q_1$ and $q_2$, that we denote by $V$.
	Let us write $\Omega = Q([1,2]^2)/s_n$ and $M = \max(\|R_1\|_{\infty}, \|R_2\|_{\infty}) + 2$. Notice that $\Omega \subset [1,M/s_n]^2$ because $q_1,q_2$ are by \eqref{eq:cauchy} non-decreasing and $\max(\|R_1\|_{\infty}, \|R_2\|_{\infty})$-Lipschitz.
	By the change of variable $(r_1,r_2) = Q(\epsilon_1,\epsilon_2) / s_n$ we have
	\begin{align*}
	\int_1^2 \int_1^2 \E \Big\langle \big(\bbf{u} \cdot \bbf{U} - \E \langle \bbf{u} \cdot \bbf{U} \rangle_{n,t}\big)^2 \Big\rangle_{n,t} d\epsilon_1 d\epsilon_2
	&= 
	\int_1^2 \int_1^2 V(q_1(t,\epsilon_1),q_2(t,\epsilon_2)) d\epsilon_1 d\epsilon_2
	\\
	&= 
	\int_{\Omega}
	V(s_n r_1, s_n r_2)  \frac{s_n^2 dr_1dr_2}{J(Q_t^{-1}(s_n r))}
	\\
	&\leq
	\int_{1}^{M / s_n}
	\int_{1}^{M / s_n}
	V(s_n r_1, s_n r_2) dr_1 dr_2,
	\end{align*}
	where we used \eqref{eq:ineq_j} for the last inequality.
	By the change of variable $r_1 = a^2$, we have for all $r_2 \geq 0$:
	$$
	\int_1^{M / s_n} V(s_n r_1, s_n r_2) d r_1
	=
	\int_1^{\sqrt{M/s_n}} V(s_n a^2, s_n r_2) 2 a d a
	\leq 2 \sqrt{\frac{M}{s_n}}\int_1^{\sqrt{M/s_n}} V(s_n a^2, s_n r_2)d a.
	$$
	By definition of $V$, the quantity $V(s_n a^2, s_n r_2)$ is the variance of the overlap $\bbf{u} \cdot \bbf{U}$ where $\bbf{u}$ is sampled from the posterior distribution of $\bbf{U}$ given $\bbf{Y}_t$, $a \sqrt{\alpha s_n} \bbf{U} + \bbf{Z}^{(u)}$ and $\sqrt{s_n r_2} \bbf{V} + \bbf{Z}^{(v)}$.
	By Proposition~\ref{prop:overlap_concentration_gg} we have for all $1 \leq r_2 \leq M /s_n$
	$$
	\frac{1}{\sqrt{M/s_n} - 1} \int_1^{\sqrt{M/s_n}} V(s_n a^2, s_n r_2) da \leq C \Big(  \frac{1}{\sqrt{ns_n}} + \sqrt{v_n} \Big)
	$$
	where $C>0$ is a constant that depends only on $K,\alpha$,
	$$
	v_n = \sup_{t \in [0,1]} 
	\sup_{0 \leq r_1,r_2 \leq M/s_n}
	\E \big| \phi_t(r_1,r_2) - \E \phi_t(r_1,r_2) \big|
	$$
	and
	$$
	\phi_t : (r_1,r_2) \mapsto \frac{1}{n s_n} \log\Big(\int_{\bbf{u},\bbf{v}} dP_U^{\otimes n}(\bbf{u}) dP_V^{\otimes m}(\bbf{v}) e^{H_{n,t}(\bbf{u},\bbf{v}; s_n r_1, s_n r_2)}\Big).
	$$
	Consequently
	$$
	\int_1^2 \int_1^2 \E \Big\langle \big(\bbf{u} \cdot \bbf{U} - \E \langle \bbf{u} \cdot \bbf{U} \rangle_{n,t}\big)^2 \Big\rangle_{n,t} d\epsilon_1 d\epsilon_2
	\leq 2 \Big(\frac{M}{s_n} \Big)^{2} C \Big(\frac{1}{\sqrt{n s_n}} + \sqrt{v_n}\Big)
	\leq \frac{C'}{s_n^2} \Big(\frac{1}{\sqrt{ns_n}} + \sqrt{v_n}\Big),
	$$
	for some constant $C'>0$. We now use the following lemma to control $v_n$:
	\begin{lemma}\label{lem:v_n_uv}
		There exists a constant $C>0$ (that only depends on $K$, $M$ and $\alpha$) such that
		$$
		v_n \leq C n^{-1/2} s_n^{-1}.
		$$
	\end{lemma}
We delay the proof of Lemma~\ref{lem:v_n_uv} to Section~\ref{sec:technical_uv}.
We deduce that 
	$$
	\int_1^2 \int_1^2 \E \Big\langle \big(\bbf{u} \cdot \bbf{U} - \E \langle \bbf{u} \cdot \bbf{U} \rangle_{n,t}\big)^2 \Big\rangle_{n,t} d\epsilon_1 d\epsilon_2
	\leq
	\frac{2C'}{n^{1/8}}
$$
if we choose $s_n = n^{-1/32}$.
\end{proof}

\subsection{Lower and upper bounds}

From now we write $\E \langle \bbf{u} \cdot \bbf{U} \rangle_{n,t}$, as a function of $(t, q_1(t), q_2(t))$:
\begin{equation}
	\E \langle \bbf{u} \cdot \bbf{U} \rangle_{n,t} = Q(t, q_1(t), q_2(t)).
\end{equation}
Notice that $Q$ is continuous, non-negative on $[0,1] \times (0,+\infty)^2$, bounded by $K^2$ and admits partial derivatives with respect to its second and third argument. These derivatives are both continuous. Moreover, notice that
$$
\E_{P_U}[U^2] - Q(t,r_1,r_2) = \E \big\| \bbf{U} - \E \big[\bbf{U}\big| \bbf{Y}_t, \sqrt{\alpha r_1} \bbf{U} + \bbf{Z}^{(u)}, \sqrt{r_2} \bbf{V} + \bbf{Z}^{(v)} \big]\big\|^2
$$
is of course non-increasing with respect to $r_1$ and $r_2$. The partial derivatives of $Q$ with respect to its second and third argument are thus non-negative.

For simplicity we will now omit the dependencies on $\lambda$ and $\alpha$ in $\cF$.
The proof of~\eqref{eq:goal_proof_uv} will follow from the two matching lower- and upper-bounds below.
\begin{proposition}\label{prop:key_uv}
	In the setting of Proposition~\ref{prop:overlap_uv}, for $\epsilon \in [1,2]^2$
	we let $q_1(t,\epsilon), q_2(t,\epsilon)$ be the solution of \eqref{eq:cauchy}, with the choice $R_2 = Q$. For this choice of functions $q_1,q_2$, we have:
	\begin{align*}
		F_n = \int_{[1,2]^2} \int_0^1\Big(\psi_{P_U}(\alpha q_1(1,\epsilon)) + \alpha \psi_{P_V}(q_2(1,\epsilon)) - \frac{\alpha}{2} q_1'(t,\epsilon) q_2'(t,\epsilon) \Big) dt d\epsilon + o_n(1).
	\end{align*}
\end{proposition}
\begin{proof}
Let us fix $\epsilon \in [1,2]^2$. 
	With the choice $R_2 = Q$, we have for all $t \in [0,1]$:
	$$
	q_2'(t,\epsilon) = Q(t,q_1(t,\epsilon),q_2(t,\epsilon)) = \E \langle \bbf{u} \cdot \bbf{U} \rangle_{n,t}.
	$$
	The derivative of \eqref{eq:der_f_interpolation_uv} becomes then by Proposition~\ref{prop:overlap_uv}:
	$$
	f_n'(t)
		=
		\frac{\alpha}{2} q'_1(t) q_2'(t) - \frac{1}{2} \E \left\langle
			\Big( \bbf{u} \cdot \bbf{U} - \E\langle \bbf{u} \cdot \bbf{U} \rangle_{n,t} \Big)
			\Big( \frac{m}{n} \bbf{v} \cdot \bbf{V} - \alpha q_1'(t) \Big)
		\right\rangle_{\! n,t} 
		=
		\frac{\alpha}{2} q_1'(t) q_2'(t) + o_n(1)
	$$
	where $o_n(1)$ denotes a quantity that goes to $0$ as $n \to \infty$, uniformly in $t,\epsilon$. By Proposition~\ref{prop:free_wasserstein} we have $f_n(0)= F_n + O_n(s_n)$. We have also: $f_n(1)= \psi_{P_U}(\alpha q_1(1,\epsilon)) + \alpha \psi_{P_V}(q_2(1,\epsilon)) + o_n(1)$.
	We conclude by
	\begin{align*}
		F_n &= \int_{[1,2]^2} f_n(0) d\epsilon  + o_n(1) 
		= \int_{[1,2]^2} \Big(f_n(1) - \int_0^1 f_n'(t) dt \Big) d\epsilon + o_n(1)
	\\
			&= \int_{[1,2]^2} \int_0^1\Big(\psi_{P_U}(\alpha q_1(1,\epsilon)) + \alpha \psi_{P_V}(q_2(1,\epsilon)) - \frac{\alpha}{2} q_1'(t,\epsilon) q_2'(t,\epsilon) \Big) dt d\epsilon + o_n(1).
	\end{align*}
\end{proof}

\subsubsection{Lower bound}

One deduces from Proposition~\eqref{prop:key_uv} the following lower bound:

\begin{proposition}\label{prop:lower_bound_uv}
	$$
	\liminf_{n \to \infty} F_{n} \geq \sup_{q_1 \geq 0} \inf_{q_2 \geq 0} \cF(q_2,q_1) \,.
	$$
\end{proposition}
\begin{proof}
	We apply Proposition~\ref{prop:key_uv} with $R_1 = r$, for some $r \geq 0$. We get $q_1(t,\epsilon) = \epsilon_1 s_n + rt$, so that:
	\begin{align*}
		F_n &= \int_{[1,2]^2} \int_0^1\Big(\psi_{P_U}(\alpha (s_n \epsilon_1 + r)) + \alpha \psi_{P_V}(q_2(1,\epsilon)) - \frac{\alpha}{2} r  q_2'(t,\epsilon) \Big) dt d\epsilon + o_n(1).
		\\
		&= \int_{[1,2]^2} \Big(\psi_{P_U}(\alpha r) + \alpha \psi_{P_V}(q_2(1,\epsilon)) - \frac{\alpha}{2} r  q_2(1,\epsilon) \Big) d\epsilon + o_n(1).
		\\
		& \geq \inf_{q_2 \geq 0} \mathcal{F}(q_2,r) + o_n(1),
	\end{align*}
	where we used the fact that $\psi_{P_U}$ is $\frac{1}{2} K^2$-Lipschitz, and that $s_n \to 0$. This proves the proposition since the last inequality holds for all $r \geq 0$.
\end{proof}

\subsubsection{Upper bound}

We will now prove the converse upper bound.

\begin{proposition}\label{prop:upper_bound_uv}
	$$
	\limsup_{n \to \infty} F_{n} \leq \sup_{q_1 \geq 0} \inf_{q_2 \geq 0} \cF(q_2,q_1) \,.
	$$
\end{proposition}
\begin{proof}
	We apply Proposition~\ref{prop:key_uv} with $R_1 = 2 \alpha \psi_{P_V}' \circ Q$.
	$R_1$ verifies the conditions of Proposition~\ref{prop:overlap_uv} because $\psi_{P_V}$ is a $\cC^2$ convex Lipschitz function (Proposition~\ref{prop:i_mmse}).

	For simplicity we omit briefly the dependencies in $\epsilon$ of $q_1$ and $q_2$.
	$\psi_{P_U}$ is $K^2 / 2$-Lipschitz, and $q_1(0) = \epsilon s_n = o_n(1)$ so $\psi_{P_U}(\alpha q_1(1)) = \psi_{P_U}\big(\alpha (q_1(1) - q_1(0))\big) + o_n(1)$, where $o_n(1)$ is a quantity that goes to $0$ as $n \to \infty$, uniformly in $\epsilon \in [1,2]^2$.
	Notice that by convexity of the functions $\psi_{P_U}$ and $\psi_{P_V}$, we get
	$$
	\psi_{P_U}(\alpha q_1(1))
	=
	\psi_{P_U}\Big(\alpha \int_0^1 q'_1(t)\Big) + o_n(1)
	\leq \int_0^1 \psi_{P_U}(\alpha q_1'(t)) dt + o_n(1).
	$$
	and similarly: $\psi_{P_V}(q_2(1)) \leq \int_0^1 \psi_{P_V}(q_2'(t)) dt + o_n(1)$.
	We get by Proposition~\ref{prop:key_uv}
	\begin{align}
		F_n &\leq \int_{[1,2]^2} \int_0^1\Big(\psi_{P_U}(\alpha q_1'(t,\epsilon)) + \alpha \psi_{P_V}(q'_2(t,\epsilon)) - \frac{\alpha}{2} q_1'(t,\epsilon)  q_2'(t,\epsilon) \Big) dt d\epsilon + o_n(1). \nonumber
		\\
			&= \int_{[1,2]^2} \int_0^1 \mathcal{F}(q_2'(t,\epsilon),q_1'(t,\epsilon)) dt d\epsilon + o_n(1). \label{eq:upper_uv}
	\end{align}
	Since we chose $R_2 = Q$ and $R_1 = 2 \alpha \psi_{P_V}' \circ Q = 2 \alpha \psi_{P_V}' \circ R_2$, Equation~\eqref{eq:cauchy} gives:
	$$
	\forall \epsilon \in [1,2]^2, \ \forall t \in [0,1],
	\qquad q_1'(t,\epsilon) = 2 \alpha \psi'_{P_V}(q_2'(t,\epsilon)).
	$$
	By convexity of $\psi_{P_V}$, this gives that for all $\epsilon \in [1,2]^2$ and all $t \in [0,1]$ we have
	$$
	\mathcal{F}(q_2'(t,\epsilon),q_1'(t,\epsilon)) = \inf_{q_2 \geq 0} \mathcal{F}(q_2, q_1'(t,\epsilon)) \leq \sup_{q_1 \geq 0} \inf_{q_2 \geq 0} \mathcal{F}(q_2,q_1).
	$$
	Together with \eqref{eq:upper_uv}, this concludes the proof.
\end{proof}


\subsection{Concentration of the free energy: proof of Lemma~\ref{lem:v_n_uv}}\label{sec:technical_uv}

In this section, we prove Lemma~\ref{lem:v_n_uv}: we show that the perturbed free energy concentrates around its mean, uniformly in the perturbation. Lemma~\ref{lem:v_n_uv} will follow from Lemma~\ref{lem:concentration_gauss} and Lemma~\ref{lem:concentration_uv} below.
Let $\E_z$ denote the expectation with respect to the Gaussian random variables $\bbf{Z}, \bbf{Z}^{(u)}, \bbf{Z}^{(v)}$.

\begin{lemma}\label{lem:concentration_gauss}
	There exists a constant $C > 0$, that only depends on $K,\alpha$, such that for all $t \in [0,1]$, $B \geq 0$ and $(r_1,r_2) \in [0,B]^2$,
	$$
	\E \left|
	\phi_t(r_1,r_2) - \E_z \phi_t(r_1,r_2)
	\right|
	\leq C n^{-1/2} s_n^{-1} \sqrt{1 + B s_n} \,.
	$$
\end{lemma}
\begin{proof}
	Let $(r_1,r_2) \in [0,B]^2$ and consider $\bbf{U}$ and $\bbf{V}$ to be fixed (i.e.\ we first work conditionally on $\bbf{U},\bbf{V}$).
	Consider the function
	$$
	f:(\bbf{Z}, \bbf{Z}^{(u)}, \bbf{Z}^{(v)}) \mapsto \phi_t(r_1,r_2) \,.
	$$
	It is not difficult to verify that 
	$$
	\| \nabla f \|^2 \leq \frac{C}{n s_n^2} (1 + B s_n)
	$$
	for some constant $C > 0$ that depends only on $K$ and $\alpha$.
	The Gaussian Poincaré inequality (see~\cite{boucheron2013concentration} Chapter 3) gives then 
	$$
	\E_z \left(\phi_t(r_1,r_2) - \E_z \phi_t(r_1,r_2)\right)^2 \leq \frac{C}{n s_n^2} (1 + B s_n)\,.
	$$
	We obtain the lemma by integration over $\bbf{U},\bbf{V}$ and Jensen's inequality.
\end{proof}
\\
\begin{lemma}\label{lem:concentration_uv}
	There exists a constant $C > 0$, that only depends on $K,\alpha$, such that for all $t \in [0,1]$, $B \geq 0$ and $(r_1,r_2) \in [0,B]^2$,
	$$
	\E \left|
	\E_z \phi_t(r_1,r_2) - \E \phi_t(r_1,r_2)
	\right|
	\leq C n^{-1/2} s_n^{-1} \sqrt{1 + B s_n} \,.
	$$
\end{lemma}
\begin{proof}
	It is not difficult to verify that the function
	$$
	g:(\bbf{U},\bbf{V}) \mapsto \E_z \phi_t(r_1,r_2) 
	$$
	verifies a ``bounded difference property'' (see~\cite{boucheron2013concentration}, Section 3.2) because the components of $\bbf{U}$ and $\bbf{V}$ are bounded by a constant $K > 0$.
	Then Corollary 3.2 from~\cite{boucheron2013concentration} (which is a corollary from the Efron-Stein inequality) implies that for all $t \in [0,1]$ and $r_1,r_2 \in [0,B]$
	$$
	\E \left(
		\E_z \phi_t(r_1,r_2) - \E \phi_t(r_1,r_2)
	\right)^2
	\leq C n^{-1}s_n^{-2}(1 + s_n B) \,.
	$$
	for some constant $C$ depending only on $K$ and $\alpha$. We conclude the proof using Jensen's inequality.
\end{proof}

\section{Proof of Theorem~\ref{th:limit_mmse_vv}}\label{sec:proof_mmse_vv}

In order to prove Theorem~\ref{th:limit_mmse_vv}, we are going to consider the following model with side information to obtain a lower bound on the MMSE.
	Suppose that we observe for $\gamma \geq 0$
	\begin{equation}\label{eq:model_mix}
	\begin{cases}
		\bbf{Y}_{\lambda} &= \sqrt{\frac{\lambda}{n}} \bbf{U} \bbf{V}^{\sT} + \bbf{Z} \\
		\bbf{Y}_{\gamma}' &= \sqrt{\frac{\gamma}{n}} \bbf{U} \bbf{U}^{\sT} + \bbf{Z}'
	\end{cases}
\end{equation}
where $(Z_{i,j}' = Z_{j,i}')_{i\leq j} \iid \cN(0,1)$ are independent from everything else. Define the corresponding free energy
 $$
 F_n(\lambda,\gamma) = \frac{1}{n} \E \log \int dP_U^{\otimes n}(\bbf{u}) dP_V^{\otimes m}(\bbf{v})
 \exp\Big(\sum_{1 \leq i \leq j \leq n} \sqrt{\frac{\gamma}{n}} Y_{i,j}' u_i u_j - \frac{\gamma u_i^2 u_j^2}{2}
 +
 \sum_{i,j} \sqrt{\frac{\lambda}{n}}Y_{i,j} u_i v_j - \frac{\lambda u_i^2 v_j^2}{2n}
 \Big) \,.
 $$
 \begin{proposition}\label{prop:f_eps_uu}
	 Recall that $\psi_{P_U}^*$ (resp.\ $\psi_{P_V}^*$) denotes the monotone conjugate (Definition~\ref{def:monotone_conjugate} in Appendix~\ref{sec:app_convex}) of $\psi_{P_U}$ (resp.\ $\psi_{P_V}$).
	 For all $\lambda, \gamma \geq 0$, we have
	 \begin{equation}\label{eq:lim_f_eps_uu}
	F_n(\lambda,\gamma)
	\xrightarrow[n \to \infty]{} 
	f(\lambda,\gamma) \defeq 
	\sup_{q_u,q_v \geq 0} \Big\{ 
		\frac{\gamma q_u^2}{4} + \frac{\alpha \lambda q_u q_v}{2} - \psi_{P_U}^*(q_u/2) - \alpha \psi_{P_V}^*(q_v/2)
	\Big\}.
	 \end{equation}
\end{proposition}
Proposition~\ref{prop:f_eps_uu} is proved at the end of this section.
Before we deduce Theorem~\ref{th:limit_mmse_vv} from Proposition~\ref{prop:f_eps_uu}, let us just mention that Proposition~\ref{prop:f_eps_uu} allows to precisely derive the information-theoretic limits for the model~\eqref{eq:model_mix}, by the ``I-MMSE'' relation (Proposition~\ref{prop:i_mmse}).

\begin{corollary}\label{cor:mmse_mix}
	For almost all $\gamma > 0$ the supremum of Proposition~\ref{prop:f_eps_uu} is achieved at a unique $q_u^*(\lambda,\gamma,\alpha)$ and 
	$$
	\MMSE^{(u)}_n(\lambda,\gamma) \defeq 
	\frac{1}{n^2} \E \Big[\sum_{1 \leq i , j \leq n} \big(U_i U_j - \E [U_i U_j| \bbf{Y}_{\lambda},\bbf{Y}'_{\gamma}]\big)^2\Big]
	 \xrightarrow[n \to \infty]{} \E[U^2]^2 - q_u^*(\lambda,\gamma,\alpha)^2 \,.
	$$
\end{corollary}
The model~\eqref{eq:model_mix} was considered in \cite{deshpande2018contextual}, in the special case $P_U = \frac{1}{2} \delta_{-1} + \frac{1}{2} \delta_{+1}$ and $P_V = \cN(0,1)$.
Theorem~6 from \cite{deshpande2018contextual} shows that one can estimate $\bbf{U}\bbf{U}^{\sT}$ better than a random guess if and only if $\gamma^2 + \alpha \lambda^2 > 1$. Corollary~\ref{cor:mmse_mix} above is more precise and general because it gives the precise expression of the minimum mean squared error for any prior $P_U,P_V$. In particular the boundary $\gamma^2 + \alpha \lambda^2 = 1$ is not expected to be the information-theoretic threshold for sufficiently sparse or unbalanced priors, see the phase diagram of Figure~\ref{fig:phase_diagram} for a similar scenario.
\\

Let us now deduce Theorem~\ref{th:limit_mmse_vv} from Proposition~\ref{prop:f_eps_uu}.
By the ``I-MMSE'' relation of Proposition~\ref{prop:i_mmse}:
\begin{equation}\label{eq:ras1}
\MMSE^{(u)}_n(\lambda) = 
\MMSE^{(u)}_n(\lambda,\gamma = 0) = \E_{P_U}[U^2] - 4 \frac{\partial F_n}{\partial \gamma}(\lambda,0^+).
\end{equation}
The sequence of convex functions $(F_n(\lambda,\cdot))_{n \geq 1}$ converges pointwise to $f(\lambda,\cdot)$ on $\R_{\geq 0}$. Thus, by Proposition~\ref{prop:deriv_convex}:
\begin{equation}\label{eq:ras2}
\limsup_{n \to \infty} \frac{\partial F_n}{\partial \gamma}(\lambda,0^+) \leq 
\frac{\partial f}{\partial \gamma}(\lambda,0^+).
\end{equation}
We need therefore the following lemma:
\begin{lemma}\label{lem:ras3}
	For all $\alpha > 0$ and all $\lambda \in D_{\alpha}$, \ $\displaystyle \frac{\partial f}{\partial \gamma}(\lambda, 0^+) = \frac{q_u^*(\alpha,\lambda)^2}{4}$.
\end{lemma}
\begin{proof}
	Let $\alpha > 0$ and $\lambda \in D_{\alpha}$. Then the supremum of \eqref{eq:lim_f_eps_uu} is uniquely achieved at $(q_u^*(\lambda,\alpha),q_v^*(\lambda,\alpha))$ because the couples achieving this supremum are by Proposition~\ref{prop:sup_inf} precisely the couples achieving the supremum of $\mathcal{F}(\lambda,\alpha,\cdot,\cdot)$ over $\Gamma(\lambda,\alpha)$. The lemma follows then from the ``envelope theorem'' of Proposition~\ref{prop:envelope_compact}.
\end{proof}
\\

From Lemma~\ref{lem:ras3} and equations \eqref{eq:ras1}-\eqref{eq:ras2} above, we conclude:
	$$
	\liminf_{n \to \infty} \MMSE_{n}^{(u)}(\lambda) 
	\geq \E_{P_U}[U^2]^2 -q_u^*(\lambda,\alpha)^2.
	$$
	Let $(\bbf{u},\bbf{v})$ sampled from the posterior distribution of $(\bbf{U},\bbf{V})$ given $\bbf{Y}$, independently of everything else. Then $\MMSE_{n}^{(u)}(\lambda) = \E_{P_U}[U^2]^2 - \E \big[ (\bbf{u} \cdot \bbf{U})^2 \big] + o_n(1)$. This gives (the corresponding result for $V$ is obtained by symmetry):
	\begin{equation}\label{eq:upper_bound_overlap}
		\limsup_{n \to \infty} \E \big[ (\bbf{u} \cdot \bbf{U})^2 \big] \leq q_u^*(\lambda,\alpha)^2 \qquad \text{and} \qquad
		\limsup_{n \to \infty} \E \big[ (\bbf{v} \cdot \bbf{V})^2 \big] \leq q_v^*(\lambda,\alpha)^2 \,.
	\end{equation}
	Now, we know by Proposition~\ref{prop:mmse_uv} that
	$$
	\E_{P_U}[U^2] \E_{P_V}[V^2] - \E \big[ (\bbf{u} \cdot \bbf{U}) (\bbf{v} \cdot \bbf{V}) \big] = \MMSE_n(\lambda) \xrightarrow[n \to \infty]{}\E_{P_U}[U^2] \E_{P_V}[V^2] - q_u^* q_v^* \,,
	$$
	which gives $\E \big[ (\bbf{u} \cdot \bbf{U}) (\bbf{v} \cdot \bbf{V}) \big] \to q_u^* q_v^*$. By Cauchy-Schwarz inequality we have
	$$
\E \big[ (\bbf{u} \cdot \bbf{U}) (\bbf{v} \cdot \bbf{V}) \big]^2 
\leq
\E \big[ (\bbf{u} \cdot \bbf{U})^2 \big] \,
\E \big[ (\bbf{v} \cdot \bbf{V})^2 \big]
	$$
	which gives, by taking the liminf:
	$$
	(q_u^* q_v^*)^2 \leq 
	\Big(\liminf_{n \to \infty} \E \big[ (\bbf{u} \cdot \bbf{U})^2 \big] \Big)
	\Big(\liminf_{n \to \infty} \E \big[ (\bbf{v} \cdot \bbf{V})^2 \big] \Big) \,.
	$$
	Combining this with~\eqref{eq:upper_bound_overlap}, we get that $\lim \E \big[ (\bbf{u} \cdot \bbf{U})^2 \big] = (q_u^*)^2$
	and the relation $\MMSE_{n}^{(u)}(\lambda) = \E_{P_U}[U^2]^2 - \E \big[ (\bbf{u} \cdot \bbf{U})^2 \big] + o_n(1)$ gives the result.
\\

\subsubsection{Proof of Proposition~\ref{prop:f_eps_uu}}
	It suffices to prove the result in the case where $P_U$ and $P_V$ have bounded support, because we can then proceed by approximation as in Section~\ref{sec:approximation}. From now, we suppose to be in that case. Since the dependency in $\gamma$ can be incorporated in the prior $P_U$ and the one in $\lambda$ in the prior $P_V$, we only have to prove Proposition~\ref{prop:f_eps_uu} in the case $\gamma = \lambda = 1$. In the sequel we will therefore remove the dependencies in $\lambda,\gamma$.
	Define for $r \geq 0$
	$$
	L_n(r) = \frac{1}{n} \E \log \int dP_U^{\otimes n}(\bbf{u}) dP_V^{\otimes m} (\bbf{v}) \exp \Big(H_n(\bbf{u},\bbf{v}) + \sum_{i=1}^n \sqrt{r} Z_i'' u_i + r U_i u_i - \frac{r}{2} u_i^2 \Big) \,,
	$$
	where $Z''_i \iid \cN(0,1)$, independently of everything else and where the Hamiltonian $H_n(\bbf{u},\bbf{v})$ is defined by~\eqref{eq:hamiltonian} (with $\lambda = 1$). $L_n$ is the free energy (expected log-partition function) for observing jointly $\bbf{Y} = \frac{1}{\sqrt{n}} \bbf{U} \bbf{V}^{\sT} + \bbf{Z}$ and $\bbf{Y}'' = \sqrt{r} \bbf{U} + \bbf{Z}''$. By an straightforward extension of Theorem~\ref{th:rs_formula_uv} we have for all $r \geq 0$:
	\begin{equation}\label{eq:lim_rs_side}
	L_n(r) \xrightarrow[n \to \infty]{} L(r)
	\end{equation}
	where
	$$
	L(r) = 
	\sup_{q_u \geq 0}
	\inf_{q_v \geq 0} \Big\{
		\psi_{P_U}(\alpha q_v + r) + \alpha \psi_{P_V}(q_u) - \frac{\alpha q_u q_v}{2} 
	\Big\}
	\,.
	$$
	\begin{lemma}
		\begin{equation}
			F_n(\lambda,\gamma) \xrightarrow[n \to \infty]{}
			\sup_{r \geq 0} \Big\{ L(r) - \frac{r^2}{4} \Big\} \,.
		\end{equation}
	\end{lemma}
	\begin{proof}
	We will follow the same steps than in Section~\ref{sec:proof_rs_uv}: we will therefore only present the main steps. 
	Let $r: [0,1] \to \R_{\geq 0}$ be a differentiable function. 
	For $0 \leq t \leq 1$ we consider the following observation channel
\begin{equation}\label{eq:interpolation_uv_uu}
	\left\{
		\begin{array}{llcll}
			\bbf{Y} &=& \sqrt{1 / n} \, \bbf{U} \bbf{V}^{\sT} &+& \bbf{Z} \\
			\bbf{Y}'_t &=& \sqrt{(1-t) / n} \, \bbf{U} \bbf{U}^{\sT} &+& \bbf{Z}' \\
			\bbf{Y}''_t &=& \sqrt{r(t)} \, \bbf{U} &+& \bbf{Z}'' \,,
		\end{array}
	\right.
\end{equation}
We will denote (analogously to~\eqref{eq:interpolation_free_uv}) by $f_n(t)$ the interpolating free energy and by $\langle \cdot \rangle_{n,t}$ (analogously to~\eqref{eq:interpolation_gibbs_uv}) corresponding Gibbs measure.
We have the analog of Equation~\eqref{eq:convex_guerra} and Lemma~\ref{lem:der_f_interpolation_uv}:
\begin{equation}\label{eq:der_interpolation_mix}
	f_n'(t) =
- \frac{1}{4} \E \Big\langle\big(\bbf{u} \cdot \bbf{U} - r'(t)\big)^2 \Big\rangle_{n,t} + \frac{r'(t)^2}{4}
+o_n(1),
\end{equation}
where $o_n(1) \to 0$, uniformly in $t \in [0,1]$. 
By taking $r(t) = r t$ for all $t \in [0,1]$, we obtain
$$
F_n = f_n(0) = f_n(1) - \int_0^1 f_n'(t) dt \geq L_n(r) - \frac{r^2}{4} \,.
$$
Therefore $\liminf F_n \geq \liminf_{n \to \infty} \big\{ L_n(r) - \frac{r^2}{4} \big\}$ which gives $\liminf F_n \geq L(r) - \frac{r^2}{4}$ for all $r \geq 0$, hence $\liminf F_n \geq \sup_{r \geq 0} \big\{ L(r) - \frac{r^2}{4} \big\}$.
\\

To prove the converse upper-bound we proceed as in Section~\ref{sec:proof_rs_uv} and chose $r$ to be solution $r(\cdot; \epsilon)$ of the Cauchy problem:
$$
\begin{cases}
	r(0) = \epsilon n^{-1/32} \\
r'(t) = \E \langle \bbf{u} \cdot \bbf{U} \rangle_{n,t}
\end{cases}
$$
where $\epsilon \in [1,2]$ is a parameter. The analog of Proposition~\ref{prop:overlap_uv} holds:
$$
\int_1^2 \E \Big\langle \big(\bbf{u} \cdot \bbf{U} - \E \langle \bbf{u} \cdot \bbf{U} \rangle_{n,t} \big)^2 \Big\rangle_{n,t} d\epsilon \leq \frac{C}{n^{1/8}}
$$
for some constant $C>0$. Using \eqref{eq:der_interpolation_mix} we get
\begin{align*}
F_n 
&= f_n(0) + o_n(1) = \int_1^2 \Big(f_n(1) - \int_0^1 f_n'(t) dt \Big) d\epsilon + o_{n}(1)
\\
&= \int_1^2 \Big( L_n(r(1,\epsilon)) - \int_0^1 \frac{r'(t,\epsilon)^2}{4} dt \Big)d\epsilon + o_n(1)
\\
&\leq\int_1^2 \int_0^1 \Big( L_n(r'(t,\epsilon)) - \frac{r'(t,\epsilon)^2}{4} \Big) dt d\epsilon + o_n(1)
\leq \sup_{0 \leq r \leq \rho_u} \Big\{ L_n(r) - \frac{r^2}{4} \Big\} + o_n(1),
\end{align*}
where $\rho_u = \E_{P_U}[U^2]$.
The free energy $L_n$ is (by the usual arguments, see Section~\ref{sec:i_mmse}) convex and non-decreasing and converges to $L$ which is thus convex (therefore continuous) and non-decreasing. By Dini's second theorem we get that the convergence in~\eqref{eq:lim_rs_side} is uniform in $r$ over all compact subsets of $\R_{\geq 0}$. We conclude
	$$
	\limsup_{n \to \infty} F_n \leq \sup_{0 \leq r \leq \rho_u} \big\{L(r) - \frac{r^2}{4} \big\}
	\leq \sup_{r \geq 0} \big\{L(r) - \frac{r^2}{4} \big\}.
	$$
\end{proof}
\\

In order to prove Proposition~\ref{prop:f_eps_uu}, it remains to show that 
$$
\sup_{r \geq 0} \Big\{ L(r) - \frac{r^2}{4} \Big\} 
=
	\sup_{q_u,q_v \geq 0} \Big\{ 
		\frac{\gamma q_u^2}{4} + \frac{\alpha \lambda q_u q_v}{2} - \psi_{P_U}^*(q_u/2) - \alpha \psi_{P_V}^*(q_v/2)
	\Big\}.
$$
This is a consequence of the following Lemma:
\begin{lemma}
	Let $f,g$ be two non-decreasing lower semi-continuous convex functions on $\R_{\geq 0}$, such that $f(0)$ and $g(0)$ are finite. 
	Let $f^*$ and $g^*$ denote their monotone conjugate (see Definition~\ref{def:monotone_conjugate} in Appendix~\ref{sec:app_convex}).
	Then
	$$
	\sup_{r \geq 0} \sup_{q_1 \geq 0} \inf_{q_2 \geq 0} \Big\{f(q_1) + g(r + q_2) - q_1 q_2 - \frac{r^2}{2} \Big\}
	=
	\sup_{q_1,q_2 \geq 0}  \Big\{\frac{q_1^2}{2} + q_1 q_2 - f^*(q_2) - g^*(q_1)\Big\}
	$$
\end{lemma}
\begin{proof}
	Let $r \geq 0$.
	Let us write $g_r: q \mapsto g(q + r)$.
	By Proposition~\ref{prop:sup_inf}, we have
	\begin{align*}
\sup_{q_1 \geq 0} \inf_{q_2 \geq 0} \Big\{f(q_1) + g(r + q_2) - q_1 q_2\Big\}
&=
\sup_{q_1,q_2 \geq 0} \Big\{q_1 q_2 - f^*(q_2) - g^*_r(q_1) \Big\}
\\
&=
\sup_{q_2 \geq 0} \Big\{g(r+q_2) - f^*(q_2) \Big\}
\\
&=\sup_{q_1,q_2 \geq 0} \Big\{q_1 (q_2 + r)  - f^*(q_2) - g^*(q_1) \Big\},
	\end{align*}
	where we used Proposition~\ref{prop:inverse} for the two last equalities. Therefore
\begin{align*}
	\sup_{r \geq 0} \sup_{q_1 \geq 0} \inf_{q_2 \geq 0} \Big\{f(q_1) + g(r + q_2) - q_1 q_2 - \frac{r^2}{2} \Big\}
	&=
	\sup_{q_1,q_2 \geq 0} \Big\{ \sup_{r \geq 0} \Big\{ q_1 r - \frac{r^2}{2} \Big\} + q_1 q_2 - f^*(q_2) - g^*(q_1) \Big\}
	\\
	&=
	\sup_{q_1,q_2 \geq 0} \Big\{\frac{q_1^2}{2} + q_1 q_2 - f^*(q_2) - g^*(q_1) \Big\}.
\end{align*}
\end{proof}

%% file: appendix.tex
\section{Proofs of some basic properties of the MMSE and the free energy}

\subsection{Proof of Proposition~\ref{prop:mmse_dec}} \label{sec:proof_mmse_dec}
	Let $0 < \lambda_2 \leq \lambda_1$. Define $\Delta_1 = \lambda_1^{-1}$, $\Delta_2 = \lambda_2^{-1}$ and
	$$
	\begin{cases}
		\bbf{Y}_1 = \bbf{X} + \sqrt{\Delta_1} \bbf{Z}_1 \\
		\bbf{Y}_2 = \bbf{X} + \sqrt{\Delta_1} \bbf{Z}_1  + \sqrt{\Delta_2 - \Delta_1} \bbf{Z}_2 \,,
	\end{cases}
	$$
	where $\bbf{X} \sim P_X$ is independent from $\bbf{Z}_1,\bbf{Z}_2 \iid \cN(0,\Id_n)$. Now, by independence between $(\bbf{X},\bbf{Y}_1)$ and $\bbf{Z}_2$ we have
	\begin{align*}
		\MMSE(\lambda_1) 
		&= \E \big\| \bbf{X} - \E [\bbf{X} | \bbf{Y}_1] \big\|^2
		= \E \big\| \bbf{X} - \E [\bbf{X} | \bbf{Y}_1, \bbf{Z}_2] \big\|^2
		= \E \big\| \bbf{X} - \E [\bbf{X} | \bbf{Y}_1, \bbf{Y}_2] \big\|^2
		\\
		&\leq \E \big\| \bbf{X} - \E [\bbf{X} | \bbf{Y}_2] \big\|^2
		= \MMSE(\lambda_2) \,.
	\end{align*}
	Next, notice that
	\begin{equation}\label{eq:upper_bound_mmse}
		\MMSE(\lambda_1)
		= \E \big\| \bbf{X} - \E [\bbf{X} | \bbf{Y}_1] \big\|^2
		\leq \E \big\| \bbf{X} - \E [\bbf{X}] \big\|^2 = \MMSE(0)\,.
	\end{equation}
	This shows that the $\MMSE$ is non-increasing on $\R_+$. 
	It remains to prove the last point:
	$$
	0 \leq \MMSE(\lambda)
	= \E \| \bbf{X} - \E[\bbf{X}| \bbf{Y}]\|^2
	\leq \E \| \bbf{X} - \frac{1}{\sqrt{\lambda}} \bbf{Y}\|^2 = \frac{n}{\lambda} \xrightarrow[\lambda \to +\infty]{} 0 \,.
	$$

	\subsection{Proof of Proposition~\ref{prop:mmse_cont}} \label{sec:proof_mmse_cont}
	We start by proving that $\MMSE$ is continuous at $\lambda = 0$.
	Let $\lambda\geq0$ and consider $\bbf{Y},\bbf{X},\bbf{Z}$ as given by~\eqref{eq:inference}.
	By dominated convergence one has almost surely that 
	$$
	\E[\bbf{X} | \bbf{Y}]
	= 
	\frac{\int dP_X(\bbf{x}) \bbf{x} e^{-\frac{1}{2}\| \sqrt{\lambda} \bbf{x} - \bbf{Y}\|^2}}{\int dP_X(\bbf{x}) e^{-\frac{1}{2}\| \sqrt{\lambda} \bbf{x} - \bbf{Y}\|^2}}
	\xrightarrow[\lambda \to 0]{} \E[\bbf{X}] \,.
	$$
	Then by Fatou's Lemma we get
	$$
	\liminf_{\lambda \to 0} \MMSE(\lambda) \geq \E \Big[\liminf_{\lambda \to 0}
		\big\| \bbf{X} - \E[\bbf{X}| \bbf{Y}] \big\|^2
	\Big] = \E \big\| \bbf{X} - \E[\bbf{X}] \big\|^2 \,.
	$$
	Combining this with the bound $\MMSE(\lambda) \leq \E \| \bbf{X} - \E[\bbf{X}]\|^2$ gives $\MMSE(\lambda) \xrightarrow[\lambda \to 0]{} \E \| \bbf{X} -\E[\bbf{X}]\|^2$. This proves that the $\MMSE$ is continuous at $\lambda=0$.
	\\

	Let us now prove that the $\MMSE$ is continuous on $\R_+^*$. We need here a technical lemma:
	\begin{lemma}\label{lem:ui1}
		For all $\lambda > 0$, $p \geq 1$
		$$
		\E \| \bbf{X} - \langle \bbf{x} \rangle_{\lambda} \|^{2p} \leq \frac{2^p (2p !)}{\lambda ^{p} p!} n^{p+1} \,.
		$$
	\end{lemma}
	\begin{proof}
		We reproduce here the proof from~\cite{guo2011estimation}, Proposition~5. We start with the equality
		$$
		\sqrt{\lambda} \left(\bbf{X} - \langle \bbf{x} \rangle_{\lambda}\right)
		=
		\sqrt{\lambda} \bbf{X} - \E [ \sqrt{\lambda} \bbf{X} | \bbf{Y} ]
		=
		\bbf{Y} - \bbf{Z} - \E [ \bbf{Y} - \bbf{Z} | \bbf{Y} ]
		=
		\E [ \bbf{Z} | \bbf{Y} ] -\bbf{Z} \,.
		$$
		We have therefore
		\begin{align*}
			\E \| \bbf{X} - \langle \bbf{x} \rangle_{\lambda} \|^{2p} 
			&= \frac{1}{\lambda ^{p}} 
			\E \big\| \E [ \bbf{Z} | \bbf{Y} ] -\bbf{Z}\big\|^{2p}
			\leq 
			\frac{2^{2p-1}}{\lambda ^{p}} 
			\E \big[  \|\E [ \bbf{Z} | \bbf{Y} ]\|^{2p} +\|\bbf{Z}\|^{2p} \big]
			\leq
			\frac{2^{2p}}{\lambda ^{p}} 
			\E \|\bbf{Z}\|^{2p} \,.
		\end{align*}
		It remains to bound
		$$
		\E \|\bbf{Z}\|^{2p}
		\leq n^p \E \left[\sum_{i=1}^n Z_i^{2p} \right]
		= n^{p+1} \frac{(2p)!}{2^p p!} \,.
		$$
	\end{proof}
	\\

	Let $\lambda_0 > 0$. The family of random variables $\big(\| \bbf{X} - \langle \bbf{x}\rangle_{\lambda}\|^2\big)_{\lambda \geq \lambda_0}$ is bounded in $L^2$ by Lemma~\ref{lem:ui1} and is therefore uniformly integrable. The function $\lambda \mapsto \| \bbf{X} - \langle \bbf{x}\rangle_{\lambda}\|^2$ is continuous on $[\lambda_0,+\infty)$, the uniform integrability ensures then that $\MMSE: \lambda \mapsto  \E\| \bbf{X} - \langle \bbf{x}\rangle_{\lambda}\|^2$ is continuous over $[\lambda_0,+\infty)$. This is valid for all $\lambda_0 >0$: we conclude that $\MMSE$ is continuous over $(0,+\infty)$.

	\subsection{Proof of the I-MMSE relation: Proposition~\ref{prop:i_mmse}} \label{sec:proof_i_mmse}
	\begin{align*}
		\MMSE(\lambda)
		= \E \| \bbf{X} - \langle \bbf{x} \rangle_{\lambda} \|^2
		= \E \| \bbf{X} \|^2 + \E \| \langle \bbf{x} \rangle_{\lambda} \|^2 - 2 \E \langle \bbf{x}^{\sT} \bbf{X} \rangle_{\lambda}
	\end{align*}
	Now, by the Nishimori property $\E \| \langle \bbf{x} \rangle_{\lambda} \|^2 = \E \big\langle (\bbf{x}^{(1)})^{\sT} \bbf{x}^{(2)} \big\rangle_{\lambda} = \E \langle \bbf{x}^{\sT} \bbf{X} \rangle_{\lambda}$. Thus
	\begin{equation}\label{eq:mmse_overlap}
		\MMSE(\lambda)
		= \E \| \bbf{X} \|^2 -  \E \langle \bbf{x}^{\sT} \bbf{X} \rangle_{\lambda} \,.
	\end{equation}
	By~\eqref{eq:mmse_overlap} and~\eqref{eq:f_i}, it suffices now to prove the second equality in~\eqref{eq:i_mmse}. This will follow from the lemmas below.

	\begin{lemma}\label{lem:free_energy_cont}
		The free energy $F$ is continuous at $\lambda = 0$.
	\end{lemma}
	\begin{proof} For all $\lambda \geq 0$,
		$$
		F(\lambda) = \E \log \int dP_X(\bbf{x}) e^{-\frac{1}{2} \| \bbf{Y}-\sqrt{\lambda} \bbf{x}\|^2 + \frac{1}{2} \|\bbf{Y}\|^2}
		= \E \log \int dP_X(\bbf{x}) e^{-\frac{1}{2} \| \sqrt{\lambda}\bbf{X}-\sqrt{\lambda} \bbf{x} + \bbf{Z}\|^2}
		+ \lambda \E \|\bbf{X}\|^2 + n \,.
		$$
		By dominated convergence $\int dP_X(\bbf{x}) e^{-\frac{1}{2} \| \sqrt{\lambda}\bbf{X}-\sqrt{\lambda} \bbf{x} + \bbf{Z}\|^2} \xrightarrow[\lambda \to 0]{} e^{-\frac{1}{2}\|\bbf{Z}\|^2}$. Jensen's inequality gives
		\begin{align*}
			\left|\log \int \!\!dP_X(\bbf{x}) e^{-\frac{1}{2} \| \sqrt{\lambda}\bbf{X}-\sqrt{\lambda} \bbf{x} + \bbf{Z}\|^2} \right|
			&=
			- \log \int \!\!dP_X(\bbf{x}) e^{-\frac{1}{2} \| \sqrt{\lambda}\bbf{X}-\sqrt{\lambda} \bbf{x} + \bbf{Z}\|^2} 
			\\
			&\leq \frac{1}{2} \! \int \!\!dP_X(\bbf{x}) \| \sqrt{\lambda}\bbf{X}-\sqrt{\lambda} \bbf{x} + \bbf{Z}\|^2
			\leq \frac{3}{2}\big(\|\bbf{X}\|^2 + \E \|\bbf{X}\|^2 + \|\bbf{Z}\|^2 \big),
		\end{align*}
		for all $\lambda \in [0,1]$. One can thus apply the dominated convergence theorem again to obtain that $F$ is continuous at $\lambda=0$.
	\end{proof}
	\\

	\begin{lemma}\label{lem:der_f0}
		For all $\lambda \geq 0$,
		$$
		F(\lambda) - F(0) = \frac{1}{2} \int_{0}^{\lambda} \E \langle \bbf{x}^{\sT} \bbf{X} \rangle_{\gamma} d\gamma \,.
		$$
	\end{lemma}
	\begin{proof}
		Compute for $\lambda > 0$
		$$
		\frac{\partial}{\partial \lambda} \log \cZ(\lambda,\bbf{Y}) = 
		\left\langle \frac{1}{2 \sqrt{\lambda}} \bbf{x}^{\sT} \bbf{Z} + \bbf{x}^{\sT} \bbf{X} - \frac{1}{2} \|\bbf{x}\|^2 \right\rangle_{\lambda} \,.
		$$
		Since $\E \|\bbf{X}\|^2 < \infty$, the right-hand side is integrable and one can apply Fubini's theorem to obtain
		$$
		F(\lambda_2) - F(\lambda_1) = \int_{\lambda_1}^{\lambda_2} \E 
		\left\langle \frac{1}{2 \sqrt{\lambda}} \bbf{x}^{\sT} \bbf{Z} + \bbf{x}^{\sT} \bbf{X} - \frac{1}{2} \|\bbf{x}\|^2 \right\rangle_{\lambda} d\lambda\,.
		$$
		By Gaussian integration by parts, we have for all $i \in \{1, \dots, n \}$ and $\lambda >0$
		$$
		\E Z_i \langle x_i \rangle_{\lambda} = \E \frac{\partial}{\partial Z_i} \langle x_i \rangle_{\lambda}
		= \E \left[ \big\langle \sqrt{\lambda} x_i^2 \big\rangle_{\lambda} - \sqrt{\lambda} \big\langle x_i \big\rangle_{\lambda}^2 \right] 
		= \sqrt{\lambda} \E \left[ \big\langle x_i^2 \big\rangle_{\lambda} - \big\langle x_i X_i \big\rangle_{\lambda} \right] 
		\,,
		$$
		where the last equality comes from the Nishimori property (Proposition~\ref{prop:nishimori}). We have therefore
		$$
		F(\lambda_2) - F(\lambda_1) = \frac{1}{2} \int_{\lambda_1}^{\lambda_2} \E 
		\left\langle  \bbf{x}^{\sT} \bbf{X} \right\rangle_{\lambda} d\lambda \,.
		$$
		By Lemma~\ref{lem:free_energy_cont}, $F$ is continuous at $0$ so we can take the limit $\lambda_1 \to 0$ to obtain the result.
	\end{proof}
	\\

	By Proposition~\ref{prop:mmse_cont}, the function $\lambda \mapsto \MMSE(\lambda)$ is continuous over $\R_+$. By~\eqref{eq:mmse_overlap} we deduce that $\lambda \mapsto \E \langle \bbf{x}^{\sT} \bbf{X} \rangle_{\lambda}$ is continuous over $\R_+$ and therefore Lemma~\ref{lem:der_f0} proves~\eqref{eq:i_mmse}.
	\\

	It remains only to show that $F$ is strictly convex when $P_X$ differs from a Dirac mass. 
	We proceed by truncation. For $N \in \N$ and $x \in \R$ we write $x^{(N)} = x \, \1(-N \leq x \leq N)$. We extend this notation to vectors $\bbf{x} = (x_1,\dots, x_n) \in \R^n$ by $\bbf{x}^{(N)} = (x_1^{(N)}, \dots, x_n^{(N)})$.
	
	For $\bbf{X} \sim P_X$ we define $P_X^{(N)}$ as the distribution of $\bbf{X}^{(N)}$. $F^{(N)}$, $\MMSE^{(N)}$  and $\langle \cdot \rangle_{\lambda,N}$ will denote respectively the corresponding free energy, MMSE and posterior distribution.
	One can compute the second derivative (since $\bbf{X}^{(N)}$ is bounded, one can easily differentiate under the integral sign) and again, using Gaussian integration by parts and the Nishimori identity one obtains:
	\begin{equation}\label{eq:psi_seconde}
		F^{(N)\prime\prime}(\lambda) = \frac{1}{2} \E \left[
			\Tr \left(\left( \langle \bbf{x}\bbf{x}^{\sT} \rangle_{\lambda,N} - \langle \bbf{x} \rangle_{\lambda,N}\langle \bbf{x} \rangle_{\lambda,N}^{\sT}
			\right)^2\right)
	\right].
\end{equation}
By Cauchy-Schwarz inequality, we have for all positive, semi-definite matrix $\bbf{M} \in \R^{n \times n}$, $\Tr(\bbf{M})^2 \leq n \Tr(\bbf{M}^2)$. Hence
\begin{align*}
		F^{(N)\prime \prime}(\lambda) 
		&\geq \frac{1}{2n} \E \left[
			\Tr \left( \langle \bbf{x}\bbf{x}^{\sT} \rangle_{\lambda,N} - \langle \bbf{x} \rangle_{\lambda,N}\langle \bbf{x} \rangle_{\lambda,N}^{\sT}
			\right)^2
	\right]
	\\
		&\geq \frac{1}{2n} \E \left[
			\Tr \left( \langle \bbf{x}\bbf{x}^{\sT} \rangle_{\lambda,N} - \langle \bbf{x} \rangle_{\lambda,N}\langle \bbf{x} \rangle_{\lambda,N}^{\sT}
			\right)
	\right]^2
		= \frac{1}{2n} \MMSE^{(N)}(\lambda)^2,
\end{align*}
	by Jensen's inequality. 
	Let now $0<s<t$. 
	By integrating~\eqref{eq:psi_seconde} we get
	\begin{equation}\label{eq:der_sec_int}
		F^{(N)\prime}(t) - F^{(N)\prime}(s) \geq \frac{1}{2n} 
		\int_s^t 
		\MMSE^{(N)}(\lambda)^2
		d\lambda \,.
	\end{equation}
	The sequence of convex functions $(F^{(N)})_N$ converges (by Proposition \ref{prop:free_wasserstein}) to $F$ which is differentiable. Proposition \ref{prop:deriv_convex} gives that the derivatives $(F^{(N)\prime})_N$ converge to $F'$ and therefore $\MMSE^{(N)}$ converges to $\MMSE$. Therefore, equation~\eqref{eq:der_sec_int} gives
	$$
		F'(t) - F'(s) \geq \frac{1}{2n} 
		\int_s^t 
		\MMSE(\lambda)^2
		d\lambda 
		\geq \frac{1}{2n} (t-s) \MMSE(t)^2
		\,.
	$$
	If $P_0$ is not a Dirac measure, then the last term is strictly positive: this concludes the proof.

	\subsection{Pseudo-Lipschitz continuity of the free energy with respect to the Wasserstein distance}\label{sec:free_wasserstein}

	Let $P_{1}$ and $P_{2}$ be two probability distributions on $\R^n$, that admits a finite second moment. We denote by $W_2(P_1,P_2)$ the Wasserstein distance of order $2$ between $P_1$ and $P_2$.
	For $i = 1,2$ the free energy is defined as
	$$
	F_{P_i}(\lambda) = \E \log \int dP_i(\bbf{x}) \exp \Big(\sqrt{\lambda} \bbf{x}^{\sT} \bbf{Z} + \lambda \bbf{x}^{\sT} \bbf{X}- \frac{\lambda}{2} \|\bbf{x}\|^2 \Big) \,,
	$$
	where the expectation is with respect to $(\bbf{X},\bbf{Z}) \sim P_i \otimes \cN(0,\Id_n)$.

	\begin{proposition}\label{prop:free_wasserstein}
		For all $\lambda \geq 0$,
		$$
		\big| F_{P_1}(\lambda) - F_{P_2}(\lambda) \big| \leq \frac{\lambda}{2} \big(\sqrt{\E_{P_1}\|\bbf{X}\|^2} + \sqrt{\E_{P_2}\|\bbf{X}\|^2}\big)W_2(P_1,P_2) \,.
		$$
	\end{proposition}
A similar result was proved in~\cite{wu2012functional} but with a weaker bound for the $W_2$ distance.

	\begin{proof}
		Let $\epsilon > 0$. Let us fix a coupling $Q$ of $\bbf{X}_1 \sim P_1$ and $\bbf{X}_2 \sim P_2$ such that
		$$
		\big( \E \| \bbf{X}_1 - \bbf{X}_2 \|^2\big)^{1/2} \leq W_2(P_1,P_2) + \epsilon \,.
		$$
		Let us consider for $t \in [0,1]$ the observation model
		$$
		\begin{cases}
			\bbf{Y}^{(t)}_1 &=\  \sqrt{\lambda t} \bbf{X}_1  + \bbf{Z}_1 \,, \\
			\bbf{Y}^{(t)}_2 &=\  \sqrt{\lambda (1-t)} \bbf{X}_2 + \bbf{Z}_2 \,,
		\end{cases}
		$$
		where $\bbf{Z}_1,\bbf{Z}_2 \iid \cN(0,\Id_n)$ are independent from $(\bbf{X}_1,\bbf{X}_2) \sim Q$. Define
		$$
		f(t) = \E \log \int\! dQ(\bbf{x}_1,\bbf{x}_2) \exp \Big(
			\sqrt{\lambda t} \bbf{x}_1^{\sT} \bbf{Y}_1^{(t)} - \frac{\lambda t}{2} \|\bbf{x}_1\|^2 
			+
			\sqrt{\lambda (1-t)} \bbf{x}_2^{\sT} \bbf{Y}_2^{(t)} - \frac{\lambda (1-t)}{2} \|\bbf{x}_2\|^2 
		\Big).
		$$
		We have $f(0) = F_{P_2}(\lambda)$ and $f(1)=F_{P_1}(\lambda)$. By an easy extension of the I-MMSE relation~\eqref{eq:i_mmse} we have for all  $t \in [0,1]$:
		$$
		f'(t) = \frac{\lambda}{2}\E \Big\langle \bbf{X}_1^{\sT} \bbf{x}_1 - \bbf{X}_2^{\sT} \bbf{x}_2 \Big\rangle_{t} \,,
		$$
		where $\langle \cdot \rangle_t$ denotes the expectation with respect to $(\bbf{x}_1,\bbf{x}_2)$ sampled from the posterior distribution of $(\bbf{X}_1,\bbf{X}_2)$ given $\bbf{Y}_1^{(t)},\bbf{Y}_2^{(t)}$, independently of everything else.
		We have then
		\begin{align*}
			| \frac{2}{\lambda} f'(t) | & =
			\Big| \E \Big\langle \bbf{X}_1^{\sT} (\bbf{x}_1 - \bbf{x}_2) - (\bbf{X}_2 - \bbf{X}_1)^{\sT} \bbf{x}_2 \Big\rangle_{t} \Big|
			\\
			&\leq
			\Big( \E \|\bbf{X}_1\|^2 \E \big\langle \| \bbf{x}_1 - \bbf{x}_2\|^2\big\rangle_{t} \Big)^{1/2} +
				\Big( \E \big\langle\|\bbf{x}_2\|^2\big\rangle_{t}  \E  \|\bbf{X}_2 - \bbf{X}_1\|^2   \Big)^{1/2}
			\\
			&=
			\Big( \E \|\bbf{X}_1\|^2 \E \| \bbf{X}_1 - \bbf{X}_2\|^2 \Big)^{1/2} +
				\Big( \E \|\bbf{X}_2\|^2  \E  \|\bbf{X}_2 - \bbf{X}_1\|^2   \Big)^{1/2}
				\\
				& \leq 
				\Big(
				\big(\E \|\bbf{X}_1\|^2 \big)^{1/2}
				+
				\big(\E \|\bbf{X}_2\|^2\big)^{1/2}
			\Big) (W_2(P_1,P_2) + \epsilon)\,,
		\end{align*}
		where we used successively the Cauchy-Schwarz inequality and the Nishimori property (Proposition~\ref{prop:nishimori}). We then let $\epsilon \to 0$ to obtain the result.
	\end{proof}

	\section{Convex analysis results}\label{sec:app_convex}

\begin{proposition}\label{prop:deriv_convex}
	Let $I \subset \R$ be an interval, and let $(f_n)_{n \geq 0}$ be a sequence of convex functions on $I$ that converges pointwise to a function $f$. Then for all $t \in I$ for which these inequalities have a sense
	$$
	f'(t^-) \leq \liminf_{n \to \infty} f_n'(t^-) \leq \limsup_{n \to \infty} f_n'(t^+) \leq f'(t^+).
	$$
\end{proposition}
\begin{proof}
	Let $t \in I$ and $h > 0$. By convexity
	$$
	f_n'(t^+) \leq \frac{f_n(t+h) - f_n(t)}{h} \xrightarrow[n \to \infty]{} \frac{f(t+h) - f(t)}{h} \xrightarrow[h \to 0]{} f'(t^+).
	$$
	The first inequality follows from the same arguments.
\end{proof}

	\subsection{Basic results on the monotone conjugate}

	\begin{definition}\label{def:monotone_conjugate}
		We define the \textit{monotone conjugate} (see \cite{rockafellar2015convex} p.110) of a non-decreasing convex function $f: \R_+ \to \R$ by:
		\begin{equation}\label{eq:def_monotone}
			f^*(x) = \sup_{y \geq 0}\, \{ xy - f(y) \}.
		\end{equation}
	\end{definition}

	The most fundamental result on the monotone  conjugate is the analog of the Fenchel-Moreau theorem:
\begin{proposition}[\cite{rockafellar2015convex}~Theorem~12.4]\label{prop:inverse}
	Let $f$ be a non-decreasing lower semi-continuous convex function on $\R_+$ such that $f(0)$ is finite. Then $f^*$ is another such function and $(f^*)^* = f$.
\end{proposition}

\begin{proposition}\label{prop:der_conjugate}
	Let $f$ be a non-decreasing lower semi-continuous convex function on $\R_+$ such that $f(0)$ is finite. Then for all $x,y \geq 0$:
	$$
	x \in \partial f^*(y)
	\ \Longleftrightarrow \ 
	f(x) + f^*(y) = xy
	\ \Longleftrightarrow \
	y \in \partial f(x).
	$$
\end{proposition}
\begin{proof}
	Let $x \in \partial f^*(y)$. We get that $(f^*)^*(x) = x y - f^*(y)$ and therefore that
	$f^*(y) = x y - f(x)$, by Proposition \ref{prop:inverse}. This gives that $x$ maximizes $s \mapsto s y - f(s)$ over $\R_+$ and thus $y \in \partial f(x)$. 
	It remains to show that $y \in \partial f(x) \implies x \in \partial f^*(y)$. This follows from Proposition \ref{prop:inverse} and the implication $ x \in \partial f^*(y) \implies y \in \partial f(x)$ that we just showed.
\end{proof}

	\subsection{A supremum formula}
The goal of this section is to prove:
\begin{proposition}\label{prop:sup_inf}
	Let $f,g$ be two convex Lipschitz functions on $\R_+$.
	For $(q_1,q_2) \in \R_+$ we define $\varphi(q_1,q_2) = f(q_1) + g(q_2) - q_1 q_2$
	and $\psi(q_1,q_2) = q_1 q_2 - f^*(q_2) - g^*(q_1)$.
	Then the set $\Gamma = \big\{(q_1,q_2) \in \R_+^2 \, \big| \, q_1 \in \partial g(q_2), \ q_2 \in \partial f(q_1) \big\}$ is non-empty and:
	\begin{equation}\label{eq:sup_inf}
		\sup_{(q_1,q_2) \in \Gamma} \varphi(q_1,q_2) \
		= \sup_{q_1,q_2 \geq 0} \psi(q_1,q_2) \ 
		=\sup_{q_1 \geq 0} \inf_{q_2 \geq 0} \varphi(q_1,q_2),
	\end{equation}
	and the two first suprema above are achieved and precisely at the same couples $(q_1,q_2)$.
	\\

	If moreover $f$ and $g$ are both differentiable and strictly convex, then the same result holds for $\Gamma$ replaced by
	\begin{equation}\label{eq:simpli_gamma}
	\widetilde{\Gamma} = \big\{ (q_1,q_2) \in \R_+^2 \, \big| \, q_2 = f'(q_1) \ \text{and} \ q_1 = g'(q_2) \big\}.
	\end{equation}
\end{proposition}
\begin{proof}
	Let $L_f$ (resp.\ $L_g$) be the Lipschitz constant of $f$ (resp.\ $g$). For $x > L_f$, $f^*(x) = +\infty$ and (since $f^*$ is lower semi-continuous by Proposition \ref{prop:der_conjugate}) $f^*(x) \to +\infty$ as $x \to L_f$. Analogously, $g(x) \to +\infty$ as $x \to L_g$. The function $\psi$ is therefore continuous on $[0,L_g) \times [0,L_f)$ and goes to $-\infty$ on the border $\big(\{L_g\} \times [0,L_f]\big) \cup \big( [0,L_g] \times \{ L_f \} \big)$.

	The functions $\psi$ achieves therefore its maximum at some $(q_1,q_2) \in [0,L_g) \times [0,L_f)$. $(q_1,q_2)$ verifies then $q_2 \in \partial g^*(q_1)$ and $q_1 \in \partial f^*(q_2)$ which gives $(q_1,q_2) \in \Gamma$ by Proposition \ref{prop:der_conjugate}. The set $\Gamma$ is therefore non-empty and 
	$$
	\sup_{q_1,q_2 \geq 0} \psi(q_1,q_2) \leq \sup_{(q_1,q_2)\in \Gamma} \varphi(q_1,q_2).
	$$
	By definition of the conjugates $f^*$ and $g^*$ we have for all $q_1,q_2 \geq 0$
	$$
	\begin{cases}
		f(q_1) + f^*(q_2) \geq q_1 q_2 \\
		g(q_2) + g^*(q_1) \geq q_1 q_2.
	\end{cases}
	$$
	We get that $\varphi(q_1,q_2) \geq \psi(q_1,q_2)$ with equality if and only if $(q_1,q_2) \in \Gamma$, by Proposition \ref{prop:der_conjugate}. This gives in particular that 
	$$
	\sup_{q_1,q_2 \geq 0} \psi(q_1,q_2) \geq \sup_{(q_1,q_2)\in \Gamma} \varphi(q_1,q_2).
	$$
	Hence, both supremum are equal and are achieved over the same couples because we have seen that all couple $(q_1,q_2)$ that achieves the supremum of $\psi$ is in $\Gamma$.

	We consider now the second equality.
	Using the definition of the monotone conjugate \eqref{eq:def_monotone} and Proposition \ref{prop:inverse}:
	\begin{align*}
		\sup_{q_1 \geq 0} \inf_{q_2 \geq 0} \varphi(q_1,q_2) 
		&= \sup_{q_1 \geq 0} \big\{ f(q_1) - g^*(q_1) \big\}
		= \sup_{q_1 \geq 0} \sup_{q_2 \geq 0} \big\{ q_1 q_2 - f^*(q_2) - g^*(q_1) \big\}.
	\end{align*}

	Let us now prove the second part of the Proposition: we now assume that $f$ and $g$ are differentiable, strictly convex.
	Let $(q_1,q_2) \in \Gamma$ be a couple that achieves the maximum of $\varphi$ over $\Gamma$. It suffices to show that $(q_1,q_2) \in \widetilde{\Gamma}$.
	If $q_1 > 0$ and $q_2 >0$, then this is trivial because $f$ and $g$ are differentiable.

	Suppose now that $q_1 = 0$ (the case $q_2 = 0$ follows by symmetry). Since $0 = q_1 \in \partial g(q_2)$ and $g$ is strictly increasing, we get that $q_2 = 0$,
	so that $\sup \psi = \varphi(q_1,q_2) = f(0) + g(0)$.
	Notice that for $f^*(f'(0)) = -f(0)$ and $g^*(g'(0)) = -g(0)$ so
	\begin{equation}\label{eq:lower_bound0}
		f(0) + g(0) = \sup \psi \geq \psi\big(g'(0),f'(0)\big) = f(0) + g(0) + f'(0)g'(0) \geq f(0) + g(0).
	\end{equation}
	We get that $(g'(0),f'(0))$ achieves the supremum of $\psi$, which implies that $(g'(0),f'(0)) \in \Gamma$. This gives that $g'(0) = 0 = f'(0)$ by strict convexity of $f$ and $g$.
	This proves that $(q_1,q_2) = (0,0) \in \widetilde{\Gamma}$.
\end{proof}

	\section{Differentiation of a supremum of functions} \label{sec:appendix_envelope}

We recall in this section two results about the differentiation of a supremum of functions from Milgrom and Segal~\cite{milgrom2002envelope}.
Let $X$ be a set of parameters and consider a function $f: X \times [0,1] \to \R$. Define, for $t \in [0,1]$
\begin{align*}
	V(t) &= \sup_{x \in X} f(x,t) \,,\\
	X^*(t) &= \big\{ x \in X \, \big| \, f(x,t) = V(t) \big\} \,.
\end{align*}

\begin{proposition}[Theorem~1 from~\cite{milgrom2002envelope}\,] \label{prop:envelope}
	Let $t \in [0,1]$ such that $X^*(t) \neq \emptyset$. Let $x^* \in X^*(t)$ and suppose that $f(x^*,\cdot)$ is differentiable at $t$, with derivative $f_t(x^*,t)$.
	\begin{itemize}
		\item  If $t > 0$ and if $V$ is left-hand differentiable at $t$, then $V'(t^-) \leq f_t(x^*,t)$.
		\item  If $t < 0$ and if $V$ is right-hand differentiable at $t$, then $V'(t^+) \geq f_t(x^*,t)$.
		\item  If $t \in (0,1)$ and if $V$ is differentiable at $t$, then $V'(t) = f_t(x^*,t)$.
	\end{itemize}
\end{proposition}

\begin{proposition}[Corollary~4 from~\cite{milgrom2002envelope}\,] \label{prop:envelope_compact}
	Suppose that $X$ is nonempty and compact. Suppose that for all $t\in [0,1]$, $f(\cdot,t)$ is continuous. Suppose also that $f$ admits a partial derivative $f_t$ with respect to $t$ that is continuous in $(x,t)$ over $X \times [0,1]$. Then
	\begin{itemize}
		\item $\displaystyle V'(t^+) = \max_{x^* \in X^*(t)} f_t(x^*,t)$ for all $t\in [0,1)$ and $\displaystyle V'(t^-) = \min_{x^* \in X^*(t)} f_t(x^*,t)$ for all $t\in (0,1]$.
		\item $V$ is differentiable at $t \in (0,1)$ is and only if \,$\displaystyle \Big\{ f_t(x^*,t) \, \Big| \, x^* \in X^*(t) \Big\}$
			is a singleton. In that case $V'(t) = f_t(x^*,t)$ for all $x^* \in X^*(t)$.
	\end{itemize}
\end{proposition}